%% file: trustregion_afem.tex
\documentclass[11pt]{article}
\usepackage{amsthm}
\usepackage[margin=1in]{geometry}
\usepackage[square,numbers]{natbib}
\usepackage{algorithm}
\input{bmacros}           
\newenvironment{msccodes}{%
  \subsection*{Mathematics Subject Classification}
	}{}
\newenvironment{keywords}{%
  \subsection*{Keywords}
	}{}
\newtheorem{assumption}{Assumption}
\newtheorem{condition}{Condition}
\newtheorem{proposition}{Proposition}
\newtheorem{theorem}{Theorem}
\usepackage{cleveref}
\crefname{assumption}{Assumption}{Assumptions}
\Crefname{assumption}{Assumption}{Assumptions}
\crefname{condition}{Condition}{Conditions}
\Crefname{condition}{Condition}{Conditions}

\title{An online adaptive finite-element method for nonsmooth PDE-constrained optimization%
\thanks{%
This work is partially supported by the Office of Naval Research (ONR) under Award 
NO: N00014-24-1-2147, NSF grant DMS-2408877, the Air Force Office of Scientific
Research (AFOSR) under Award NO: FA9550-25-1-0231. This research was sponsored,
in part, by the Department of Energy Office of Science under the Advanced
Scientific Computing Research ASCEND and Early Career Research Programs.  Sandia
National Laboratories is a multimission laboratory managed and operated by
National Technology and Engineering Solutions of Sandia, LLC., a wholly owned
subsidiary of Honeywell International, Inc., for the U.S.~Department of
Energy's National Nuclear Security Administration under contract DE-NA0003525.
This paper describes objective technical results and analysis. Any subjective
views or opinions that might be expressed in the paper do not necessarily
represent the views of the U.S.~Department of Energy or the United States
Government.
}}

\author{
Harbir Antil\footnotemark[2] \and
Robert J. Baraldi\footnotemark[3] \and 
Rohit Khandelwal\footnotemark[4] \and
Drew P. Kouri\footnotemark[3]
} 

\ifpdf
\hypersetup{
  pdftitle={Adaptive Trust Region Method},
  pdfauthor={Harbir Antil, Robert J. Baraldi, Rohit Khandelwal and, Drew P. Kouri}
}
\fi

\graphicspath{{./../math-comp-draft/}{figs/} }
\DeclareGraphicsExtensions{.pdf}

\begin{document}
\maketitle
\renewcommand{\thefootnote}{\fnsymbol{footnote}}
\footnotetext[2]{Department of Mathematical Sciences and Center for Mathematics and Artificial Intelligence, 
George Mason University, Fairfax, VA, 22030, USA; hantil@gmu.edu}
\footnotetext[4]{Faculty of Mathematical Sciences, South Asian University, New Delhi, India, 110068; rkhandel@sau.int}
\footnotetext[3]{Optimization and Uncertainty Quantification, Sandia National
	Laboratories, PO Box 5800, Albuquerque, 87185-1320, NM, USA;
	\{rjbaral@sandia.gov, dpkouri@sandia.gov\}}

\begin{abstract}
\input{abstract}
\end{abstract}


\begin{msccodes}
49M37, 49K20, 49J20, 65M12, 65M15, 65M60, 90C30, 93C20
\end{msccodes}


\section{Introduction}
\label{sec:introduction}
\input{introduction}

\section{Preliminary Results}
\label{sec:model_problem}
\input{preliminary}

\section{Adaptive Finite Elements}
\label{sec:adaptive}
\input{adaptivefem}

\section{Trust-Region Algorithm}
\label{sec:afemtrust}
\input{adaptivetrust}

\section{Numerical Examples}
\label{sec:numerical_examples}
\input{numerical_examples}

\bibliographystyle{abbrvnat}
\bibliography{references}

\end{document}

%% file: bmacros.tex
\usepackage{latexsym,amsmath,amssymb,amsfonts}
\usepackage{amsopn}
\usepackage{amsbsy}
\usepackage{dsfont}
\usepackage{mathtools} 
\usepackage{stmaryrd}

\usepackage{graphicx}
\usepackage{thmtools}
\usepackage[caption=false]{subfig}
\usepackage{epstopdf}
\usepackage{algorithmic}
\usepackage{enumitem}
\usepackage{psfrag}
\usepackage{graphics}
\usepackage{xcolor}
\usepackage{soul}       
\usepackage[textsize=small]{todonotes}
\usepackage{comment}
\usepackage{lipsum}
\usepackage{hyperref} 
\usepackage{tikz} 
\usetikzlibrary{calc}

\usepackage{booktabs,array}

\ifpdf
  \DeclareGraphicsExtensions{.eps,.pdf,.png,.jpg}
\else
  \DeclareGraphicsExtensions{.eps}
\fi

\definecolor {uclablue}   {RGB} {39,116,174}
\definecolor {masongreen} {RGB} {0,102,51}

\DeclareMathOperator*{\argmin}{argmin}

\newcommand{\prox}[1]{\mathrm{Prox}_{#1}}

\newcommand{\cZ}[0]{\mathcal{Z}}
\newcommand{\cL}[0]{\mathcal{L}}

\newcommand{\cU}[0]{\mathcal{U}}
\newcommand{\cY}[0]{\mathcal{Y}}
\newcommand{\cX}[0]{\mathcal{X}}
\newcommand{\cE}[0]{\mathcal{E}}

\newcommand{\cG}[0]{\mathcal{G}}

\newcommand{\cT}[0]{\mathcal{T}}
\newcommand{\cV}[0]{\mathcal{V}}
\newcommand{\cW}[0]{\mathcal{W}}

\newcommand{\dv}{{\rm div}}
\newcommand{\norm}[1]{\left\|#1\right\|}
\newcommand{\inner}[2]{\left\langle#1,#2\right\rangle}

\def\R{\mathbb{R}}

\def\cN{\mathcal{N}}

\def\cV{\mathcal{V}}
\def\cM{\mathcal{M}}
\def\cT{\mathcal{T}}
\def\cE{\mathcal{E}}

\newcommand\rowincludegraphics[2][]{\raisebox{-0.45\height}{\includegraphics[#1]{#2}}}

\newcommand{\ared}{\textup{ared}}
\newcommand{\pred}{\textup{pred}}
\newcommand{\cred}{\textup{cred}}

\newcommand{\kgrad}{\kappa_{\textup{der}}}
\newcommand{\krad}{\kappa_{\textup{rad}}}
\newcommand{\kval}{\kappa_{\textup{val}}}
\newcommand{\kfcd}{\kappa_{\textup{fcd}}}
\newcommand{\kadj}{\kappa_{\textup{adj}}}
\newcommand{\kcurv}{\kappa_{\textup{curv}}}

%% file: abstract.tex
We present a trust-region-based adaptive finite-element algorithm for
numerically solving a class of nonsmooth PDE-constrained optimization problems
that includes problems with sparsifying regularizers and convex constraints.
In particular, we consider the class of problems whose objective function is
the sum of a smooth, possibly nonconvex, function and a nonsmooth extended
real-valued convex function.  Our method combines the robustness of inexact
trust-region algorithms for nonsmooth problems with the efficiency of adaptive
finite-element discretizations.  Starting from a coarse mesh, the algorithm
automatically refines the discretization based on reliable \textit{a posteriori}
error estimators for both the state and adjoint equations, systematically
controlling the accuracy of the computed smooth objective function value and
gradient.  This adaptivity mechanism balances computational cost and solution
accuracy, enabling high resolution of localized phenomena and sparsity
structures in the state and control variables. We demonstrate the performance
of our algorithm through numerical experiments on representative control and
topology optimization examples.

%% file: introduction.tex
We consider the optimization problem
\begin{align}
\label{eq:prob}
  \min_{z \in \cZ} \; J(S(z),z) + \phi(z),
\end{align}
where $\cZ$ is a Hilbert space, $\cU$ is a Banach space, the objective function
consists of a smooth function $J:\cU\times\cZ\to\R$ and a proper, closed and
convex---potentially nonsmooth---function $\phi:\cZ \to(-\infty,+\infty]$, and
$u=S(z)\in\cU$ is the solution to a system of partial differential equations
(PDEs), i.e., $c(u,z) = 0\in\cV^*$.  Here, $\cV$ is a reflexive Banach space,
$\cV^*$ its dual space, and $c:\cU\times\cZ\to\cV^*$ represents a weak form of
the PDE, which we assume to be sufficiently regular as to admit derivatives of
the solution map $z\mapsto S(z)$. Commonly $\cV=\cU$ \cite{ulbrich2008book},
but we do not make this assumption.

Algorithms for solving optimization problems constrained by PDEs
\eqref{eq:prob} must balance convergence robustness with the computational cost
of solving the PDE. In applications where solutions may exhibit sharp interfaces
(e.g., topology optimization \cite{MR2208515}), memory limitations
\cite{MAlshehri_HAntil_Eherberg_DPKouri_2024a,antil2025random,muthukumar2021randomized},
or high-dimensional uncertainties
\cite{garreis2017constrained,kouri2013trust,kouri2014inexact,MR3981385,zou2022locally},
solvers must be able to handle inexactness arising
from coarse discretizations or early termination of iterative linear/nonlinear
PDE solvers.
Many optimization algorithms can leverage various forms of inexactness,
such as trust-region \cite{antil2025random,kouri2020matrix,MR3981385}, 
linesearch \cite{dontchev2013newt,shi2005inexact}, and 
sequential quadratic programming (SQP) methods \cite{heinkenschloss2014matrix,heinkenschloss2002sqp,MR3648951,MR3069101,MR2765487}. 
Modern inexact trust-region algorithms build upon the early developments of
Mor\'{e} \cite{more1983recent} and Carter \cite{RGCarter_1991a,RGCarter_1993a},
and control inexactness automatically during optimization. 
Examples include: inexact linear solves \cite{MR4620236}, 
adaptive quadrature \cite{kouri2013trust,kouri2014inexact}, reduced-order
models \cite{fahl2003reduced,MR3981385,zou2022locally},
and randomized compression \cite{antil2025random,muthukumar2021randomized}. 
Similarly, SQP methods can be modified to incorporate inexactness in both
the solution of the quadratic subproblem and objective values and gradients 
\cite{byrd2008sqp,heinkenschloss2014matrix,heinkenschloss2002sqp,MR2765487}. 
In particular, \cite{MR3648951,MR3069101,MR2765487} bear the closest
resemblance to the present work in their algorithmic control of adaptive finite
elements; however, these methods do not handle nonsmoothness aside from convex
control constraints \cite{MR3069101}. Similar analysis was employed by
\cite{bellavia2018lm} for nonlinear least-squares methods with inexact
objective values and gradients under similar assumptions as \cite{MR4620236}. 
Inexactness in step computation for nonsmooth methods has been studied in 
\cite{dontchev2013newt,li2017proxnewt}, but the objective value and gradient 
are exact. We build upon recent advances in nonsmooth trust-region algorithms
\cite{MR4620236} for infinite-dimensional optimization problems
\eqref{eq:prob}.  
Specifically, \cite[Algorithm~1]{MR4620236} demonstrates
rapid convergence \cite{kouri2023localconv}, subproblem solver flexibility
\cite{MR4849972}, and inexactness \cite{MR4620236,antil2025random}. With
respect to the latter, \cite[Algorithm~1]{MR4620236} converges even when the
objective value and gradient are evaluated up to some algorithmically
determined tolerance. We utilize this algorithm to
provide a robust foundation for challenging PDE-constrained optimization
problems via adaptive finite-element discretizations using residual-based error
indicators of the state and adjoint equations. 

Adaptive finite-element methods (AFEM) \cite{MR1885308,rannacher2003afem,MR4793681,MR1284252} provide an
effective mechanism for reducing computational cost while maintaining
sufficient accuracy required for solving the PDE. 
AFEM typically exploits three types of adaptivity (or combinations thereof): 
$(i)$ $h$-adaptivity that refines the mesh \cite{demkowicz2002hp,pardo2010hp}; 
$(ii)$ $p$-adaptivity that increases the polynomial order of the elements \cite{demkowicz2002hp}; and
$(iii)$ $r$-adaptivity that relocates mesh nodes \cite{russell2010move,carstensen2014adapt}.
Guiding the aforementioned types of refinement are various estimators 
that quantify the discretization error
associated with the approximation of the infinite-dimensional space.
The estimator and adaptivity combination
is often tailored to the optimization problem; c.f. \cite{becker2000afem,besier2012qoi,hintermueller2014dw,vexler2008qoi}.
For instance, goal-oriented AFEM aims to perform adaptivity by 
minimizing an error most relevant to a desired quantity of 
interest rather than a global norm, e.g.,
energy norms in the case of Navier-Stokes simulation 
\cite{besier2012qoi} and variational inequalities \cite{hintermueller2014dw}.
A more mechanical approach involves solving a full optimization 
problem, refining, then repeating the process until a particular refinement 
level or error is attained
\cite{becker2004fe,becker2005mesh,vexler2005out,vexler2008qoi};
this technique of optimize-then-refine often lacks rigorous convergence
guarantees. 
In concert with these techniques, 
a variety of {\it a posteriori} error estimators exist that indicate which cells should be refined \cite{rannacher2003afem}.
A common type is the averaging 
ZZ-estimator \cite{feischl2014zz,zhu1987zz}, which was 
derived for higher-order elements in \cite{MR3648951,MR2765487}.
Along with ZZ estimators, residual-based error estimators can be used for online refinement. 
For instance, the aforementioned Ulbrich and Ziems work \cite{MR2765487} proposed an
inexact adaptive SQP method with residual-driven refinement for
nonlinear PDE-constrained optimization. This SQP algorithm was later generalized
for time-dependent PDEs in \cite{MR3648951} and control constraints in \cite{MR3069101}. 
More recently, \cite{MR4122501} develop AFEM for sparse PDE-constrained optimization 
and \cite{MR4893377} performed adaptivity for regularized problems \eqref{eq:prob}.
In our work, we use {\it reliable}, residual-based estimators that involve the
state or adjoint equations.  Reliable estimators provide upper bounds on the
true, infinite-dimensional error and hence indicate where the mesh should be
refined to improve the approximation quality \cite{rannacher2003afem,carstensen2014adapt}.

Our paper is organized in the following way.  In \Cref{sec:model_problem}, we
discuss the problem formulation and illustrate the general concepts behind our
finite-element implementation in \Cref{sec:adaptive}. In \Cref{sec:afemtrust},
we introduce the nonsmooth trust-region algorithm and discuss how it handles
inexactness.  We conclude in \Cref{sec:numerical_examples}, demonstrating the
performance of our approach on numerical examples.

%% file: preliminary.tex
Let $\cW$ and $\cX$ be real Banach spaces and $\cY$ be a real Hilbert space.
We denote by $\cL(\cW,\cX)$ the space of continuous linear operators mapping
$\cW$ into $\cX$. When $\cX=\cW$, we simplify denote $\cL(\cX,\cX)=\cL(\cX)$
and when $\cX=\R$, we denote $\cL(\cW,\R)=\cW^*$ (i.e., the topological dual
space of $\cW$).  We denote the norm on $\cW$ by $\|\cdot\|_{\cW}$ and the
duality pairing between $\cW$ and $\cW^*$ by
$\langle\cdot,\cdot\rangle_{\cW^*,\cW}$.  Moreover, we denote the inner product
on $\cY$ by $\langle\cdot,\cdot\rangle_{\cY}$ and we assume the norm
$\|\cdot\|_{\cY}$ is the usual Hilbert norm.  For the product space $\cW\times\cX$,
we employ the norm 
\[
  \|(w,x)\|_{\cW\times\cX} := \|w\|_{\cW}+\|x\|_{\cX} \quad \mbox{for } (w,x) \in \cW\times\cX.
\]
For any proper, closed and convex function $\psi:\cY\to(-\infty,+\infty]$, we
denote the {\em effective domain} of $\psi$ by
$\text{dom}\,\psi\coloneqq \{y\in\cY~\vert~\psi(y) < +\infty\}$ and the
{\em proximity operator} of $\psi$ by
\begin{align}
  \prox{r\psi}(y)\coloneqq \argmin_{y' \in \cY}\; \left\{\tfrac{1}{2r} \|y'-y\|_{\mathcal{Y}}^2+\psi(y')\right\}, \qquad r>0.
\end{align}
Finally, for a Fr\'echet-differentiable function $h:\cW\to\R$, we denote the
derivative of $h$ at $w\in\cW$ by $h'(w)\in\cW^*$. 

For the convergence analysis of the trust-region method in \cite{MR4620236}
applied to the general optimization problem
\begin{equation}\label{eq:gen}
  \min_{z\in\cZ} \{F(z) \coloneqq f(z) + \phi(z)\}
\end{equation}
we require the following basic assumptions on the problem data
$\cZ$, $f$ and $\phi$.
\begin{assumption}[General Assumptions]
\label{as:prob-data}
The following conditions hold for the data in \eqref{eq:gen}.
\begin{enumerate}
\item $\cZ$ is a Hilbert space and the function $\phi:\cZ\to(-\infty,+\infty]$
      is proper, closed and convex.
\item There exists an open set $\cN_{\cZ}\subseteq\cZ$, containing
      $\text{dom}\,\phi$, on which the function $f:\cZ\to\R$ is Fr\'{e}chet
      differentiable with Lipschitz continuous gradient.  We denote the
      Lipschitz modulus for the gradient by $L_f>0$.
\item The objective function $F=f+\phi$ is bounded below.
\end{enumerate}
\end{assumption}
We recall from \cite[Lemma~1]{MR4620236} 
that any local minimizer $\bar{z}\in\cZ$ for \eqref{eq:gen} 
satisfies
\[
 \bar z = \prox{r\phi}(\bar z-r\nabla f(\bar z))
\]
for any fixed $r>0$ and 
$\nabla f(x) \in \cZ$ the Reisz representation of the derivative
$f'(x)\in\cZ^*$.  With this in mind, we define
the function $\Psi:\cZ\times(0,+\infty)\to[0,+\infty)$ by
\begin{equation}\label{eq:stopCrit}
  \Psi(z,r)\coloneqq \frac{1}{r}\norm{\prox{r\phi}(z-r\nabla f(z))-z}_{\mathcal{Z}}.
\end{equation}
Moreover, we say that $z\in\cZ$ is {\em stationary} or is a
{\em stationary point} for \eqref{eq:gen} if $\Psi(z,r) = 0$ for any fixed
$r>0$. 

Given the particular form of our target optimization problem \eqref{eq:prob},
where $f(z)=J(S(z),z)$ with $S(z)$ the solution to a system of PDEs, we
postulate the following assumptions, which allow us to prove
that \Cref{as:prob-data} holds, and importantly facilitate error estimation based on
the state and adjoint residuals.
\begin{assumption}[Regularity Properties for \eqref{eq:prob}]
\label{as:jac}
The following conditions hold for the data in \eqref{eq:prob}.
\begin{enumerate}
\item $\cZ$ is a Hilbert space, $\cU$ and $\cV$ are Banach spaces with
      $\cV$ reflexive, and $\phi:\cZ\to(-\infty,+\infty]$ is proper, closed and
      convex.
\item The objective function $J:\cU\times\cZ\to\R$ and PDE map
      $c:\cU\times\cZ\to\cV^*$ are continuously Fr\'{e}chet differentiable.
      We denote the partial derivatives of $J$ by $J_u$ and $J_z$,
      and analogously for $c$.
\item There exists open convex sets $\cN_{\cZ}\subseteq\cZ$ and
      $\cN_{\cU}\subseteq\cU$ satisfying $\text{dom}\,\phi\subset\cN_{\cZ}$
      and for all $z\in\cN_{\cZ}$, there exists a unique $u\in\cN_{\cU}$
      satisfying $c(u,z) = 0$.  For fixed $z\in\cN_{\cZ}$, we denote the unique
      solution to $c(u,z) = 0$ by $u=S(z)$ and refer to $S:\cN_{\cZ}\to\cN_{\cU}$ as the
      control-to-state map.
\item The objective function $(u,z)\mapsto J(u,z)+\phi(z)$ is bounded from
      below on $\cN\coloneqq\cN_{\cU}\times\cN_{\cZ}$.
\item There exists a positive constant $\sigma_0>0$ such that the state
      Jacobian $c_u(\bar u,\bar z)\in\cL(\cU,\cV^*)$ satisfies
      \begin{subequations}\label{eq:infsup}
      \begin{align}
        &\sup_{\substack{v\in\cV \\ v \neq 0}}
        \int_0^1\frac{\langle c_u(\bar u + t(\bar{\bar u}-\bar u),\bar z)u,v\rangle_{\cV^*,\cV}}{\|v\|_{\cV}}\,\textup{d}t \ge \sigma_0\|u\|_{\cU} \quad\,\forall\,u\in\cU,\,\bar{\bar{u}}\in\cN_{\cU} \label{eq:infsup1} \\
        &\sup_{\substack{u\in\cU \\ u \neq 0}}
        \frac{\langle c_u(\bar u,\bar z)u,v\rangle_{\cV^*,\cV}}{\|u\|_{\cU}} \ge \sigma_0\|v\|_{\cV} \quad\,\forall\,v\in\cV \label{eq:infsup2}
      \end{align}
      \end{subequations}
      for all $(\bar u,\bar z)\in\cN$.
\item For fixed $u\in\cN_{\cU}$, the PDE map $c(u,\cdot)$ is Lipschitz continuous
      on $\cN_{\cZ}$ and the associated Lipschitz modulus $L_c>0$ is independent
      of $u$.
\item The partial Jacobians $c_u$ and $c_z$, and the partial derivatives of the
      objective function $J_u$ and $J_z$ are Lipschitz continuous on $\cN$.  We
      denote the Lipschitz modulii of these quantities by $L_{c_u}>0$, etc.
\item The partial derivatives $c_u$, $c_z$, and $J_u$ are uniformly bounded on
      $\cN$, i.e., there exists $\sigma_1\ge\sigma_0$ such that 
      \begin{equation}\label{eq:infsup3}
        \|c_u(\bar u,\bar z)\|_{\cL(\cU,\cV^*)} \le \sigma_1,\quad
        \|c_z(\bar u,\bar z)\|_{\cL(\cZ,\cV^*)} \le \sigma_1,\quad
        \|J_u(\bar u,\bar z)\|_{\cZ^*} \le \sigma_1
      \end{equation}
      for all $(\bar u, \bar z)\in\cN$.
\end{enumerate}
\end{assumption}

Some comments are in order regarding \Cref{as:jac}.  Conditions~1,~2~and~4 are
standard, while condition~3 is effectively the result of the implicit function
theorem.  Condition~5 is somewhat non-standard.  For example,
\eqref{eq:infsup1} is stronger than the traditional inf-sup
condition (cf.\ \cite{ern2004theory}).  That is, by setting
$\bar{\bar{u}}=\bar{u}$, we obtain
\begin{equation}\label{eq:infsup1-a}
  \sup_{\substack{v\in\cV \\ v\neq 0}}\;\frac{\langle c_u(\bar{u},\bar{z})u, v\rangle_{\cV^*,\cV}}{\|v\|_{\cV}} \ge \sigma_0\|u\|_{\cU} \quad\forall\, u\in\cU.
\end{equation}
In addition, we require that the inf-sup conditions \eqref{eq:infsup} hold
uniformly on $\cN$.  On the one hand, these conditions ensure that
$c_u(\bar{u},\bar{z})$ is bijective for all $(\bar{u},\bar{z})\in\cN$
\cite[Corollary~A.45]{ern2004theory}.  On the other hand, these conditions will
facilitate the development of error bounds. Of course, if $u\mapsto c(u,z)$ is
linear for fixed $z\in\cN_{\cZ}$, then \eqref{eq:infsup1} simplifies to the
standard inf-sup condition \eqref{eq:infsup1-a} and verifying \eqref{eq:infsup}
can be done using standard techniques. Another common setting where
\eqref{eq:infsup1} is satisfied is when $\cU$ is a Hilbert space, $\cU=\cV$
and $c_u(\bar{u},\bar{z})$ is uniformly coercive on $\cN$, i.e.,
\[
  \langle c_u(\bar{u},\bar{z})u,u \rangle_{\cU^*,\cU} \ge \sigma_0 \|u\|_{\cU}^2 \quad\forall\, u\in\cU
\]
for all $(\bar{u},\bar{z})\in\cN$.  As we will see, the remaining conditions
will facilitate proving the Lipschitz continuity of the gradient of the
reduced objective function $z\mapsto J(S(z),z)$ as well as developing error
bounds for guiding the AFEM discretization during optimization.

By introducing the control-to-state map $S(\cdot)$, we can define the reduced
objective functions $f:\cZ\to\R$ and $F:\cZ\to(-\infty,+\infty]$ by $f(z)\coloneqq J(s(z),z)$
and $F(z)\coloneqq f(z) + \phi(z)$, respectively. Under \Cref{as:jac}, we have
that $f$ is continuously differentiable with derivative given by
\[
  f'(z) = c_z(S(z),z)^*\lambda + J_z(S(z),z),
\]
where $\lambda\in\cV$ solves the adjoint equation
\[
  c_u(S(z),z)^*\lambda + J_u(S(z),z) = 0.
\]
By \Cref{as:jac} and \cite[Remark~2.7]{ern2004theory}, $c_u(S(z),z)^*$ is
bijective and so the adjoint variable $\lambda$ is unique. To simplify the
presentation, we introduce the general adjoint residual operator
$\cG:\cU\times\cZ\times\cV\to\cU^*$ defined by
\begin{align*}
 \cG(u, z, \lambda) \coloneqq J_u(u,z) + c_u(u,z)^* \lambda
\end{align*}
and the general adjoint solution operator $\Lambda:\cU\times\cZ\to\cV$
defined implicitly by
\[
  \cG(u,z,\Lambda(u,z)) = 0.
\]
Moreover, we introduce the function $g:\cU\times\cZ\times\cV\to\cZ^*$ by 
\begin{align*}
  g(u,z,\lambda)\coloneqq c_z(u,z)^*\lambda + J_z(u,z)
\end{align*}
and note that $f'(z) = g(S(z),z,\Lambda(S(z),z))$.

As stated above, we can leverage \Cref{as:jac} to verify \Cref{as:prob-data} as
in the following result.
\begin{proposition}\label{prop:assum}
  If \Cref{as:jac} holds, then \Cref{as:prob-data} holds for \eqref{eq:prob}.
\end{proposition}
\begin{proof}
  Note that \Cref{as:prob-data}.1 holds by \Cref{as:jac}.1 and that
  \Cref{as:prob-data}.3 holds by \Cref{as:jac}.4.  We now prove
  \Cref{as:prob-data}.2.  Let $z,\,\bar{z}\in\cN_{\cZ}$ be arbitrary and
  define $u=S(z)$, $\bar u=S(\bar z)$, $\lambda = \Lambda(u,z)$ and
  $\bar\lambda=\Lambda(\bar u,\bar z)$.  We first bound $\|u-\bar u\|_{\cU}$
  and $\|\lambda-\bar\lambda\|_{\cV}$.  For the state bound, \eqref{eq:infsup1}
  ensures that
  \begin{align}
    \langle c(\bar u,z), v\rangle_{\cV^*,\cV} &= \langle c(\bar u, z) - c(u,z),v\rangle_{\cV^*,\cV} \nonumber \\
    &= \int_0^1 \langle c_u(u + t(\bar u-u),z)(\bar u - u), v\rangle_{\cV^*,\cV}\,\textup{d}t \label{eq:res-int}
  \end{align}
  for all $v\in\cV$.  Here, we have applied the mean-value theorem to
  $w\mapsto \langle c(w,z), v\rangle_{\cV^*,\cV}$ using, e.g.,
  \cite[Theorem~4.2]{lang2005undergraduate} and \cite[Theorem~3.7.12]{hille1957functional}.
  Dividing by $\|v\|_{\cV}$, passing to the supremum over $\cV\ni v\neq 0$ and
  applying \eqref{eq:infsup1} yields
  \begin{equation}\label{eq:state-res-bnd}
    \|c(\bar u,z)\|_{\cV} \ge \sigma_0 \|\bar u - u\|_{\cU}.
  \end{equation}
  In addition, \Cref{as:jac}.6 ensures that
  \[
    \|c(\bar u, z)\|_{\cY^*} = \|c(\bar u, z) - c(\bar u, \bar z)\|_{\cY^*} \le L_c \|z-\bar z\|_{\cZ}.
  \]
  Combining these bounds yields
  \[
    \|\bar u - u\|_{\cU} \le \sigma_0^{-1} L_c \|\bar z-z\|_{\cZ}.
  \]
  We employ similar arguments for the adjoint bound.  In particular, employing
  the adjoint equation, we have that
  \[
    \|c_u(\bar u,\bar z)^*(\lambda-\bar\lambda)\|_{\cU^*} \le \|c_u(\bar u,\bar z)-c_u(u,z)\|_{\cL(\cU,\cV^*)}\|\lambda\|_{\cV} + \|J_u(\bar u, \bar z)-J_u(u,z)\|_{\cU^*}.
  \]
  Applying \eqref{eq:infsup2} and \Cref{as:jac}.7 yields
  \[
    \begin{aligned}
    \|\lambda-\bar\lambda\|_{\cV}&\le \sigma_0^{-1} (L_{c_u}\|\lambda\|_{\cV}+L_{J_u})\|(\bar u,\bar z)-(u,z)\|_{\cU\times\cZ} \\
    &\le \sigma_0^{-1} (L_{c_u}\|\lambda\|_{\cV}+L_{J_u})(1+\sigma_0^{-1}L_c)\|\bar z-z\|_{\cZ}.
    \end{aligned}
  \]
  Finally, \eqref{eq:infsup2} ensures that 
  \[
    \sigma_0\|\lambda\|_{\cV} \le \|c_u(\bar u,\bar z)^*\lambda\|_{\cU^*} = \|J_u(\bar u,\bar z)\|_{\cU^*},
  \]
  where the right-hand side is bounded by \Cref{as:jac}.8.  Consequently, there
  exists constants $C_u>0$ and $C_\lambda>0$ such that
  \begin{equation}\label{eq:state-adjoint-bounds}
    \|u-\bar u\|_{\cU}\le C_u\|z-\bar z\|_{\cZ}
    \qquad\text{and}\qquad
    \|\lambda-\bar\lambda\|_{\cV} \le C_\lambda \|z-\bar z\|_{\cZ}.
  \end{equation}
  Using \eqref{eq:state-adjoint-bounds}, we can bound the difference between
  the derivatives $f'(z)$ and $f'(\bar z)$.  In particular, we have that
  \[
    \begin{aligned}
    \|f'(z)-f'(\bar z)\|_{\cZ^*} \le\; &\|(c_z(u,z)-c_z(\bar u,\bar z))^*\lambda\|_{\cZ^*}
    + \|c_z(\bar u,\bar z)^*(\lambda-\bar\lambda)\|_{\cZ^*} \\
    &+ \|J_z(u,z)-J_z(\bar u,\bar z)\|_{\cZ^*}.
    \end{aligned}
  \]
  Here, note that the first term on the right-hand side is bounded by
  \[
    \|(c_z(u,z)-c_z(\bar u,\bar z))^*\lambda\|_{\cZ^*} \le L_{c_z}\|(u,z)-(\bar u,\bar z)\|_{\cU\times\cZ}\|\lambda\|_{\cV}
  \]
  since $c_z$ is Lipschitz continuous on $\cN$. Similarly, the third term is
  bounded by
  \[
    \|J_z(u,z)-J_z(\bar u,\bar z)\|_{\cZ^*}\le L_{J_z} \|(u,z)-(\bar{u},\bar{z})\|_{\cU\times\cZ}
  \]
  since $J_z$ is Lipschitz continuous on $\cN$.  In contrast, the second term
  is bounded by
  \[
    \|c_z(\bar u,\bar z)^*(\lambda-\bar\lambda)\|_{\cZ^*} \le \sigma_1\|\lambda-\bar\lambda\|_{\cV}.
  \]
  Combining these bounds with \eqref{eq:state-adjoint-bounds} yields the
  desired result.
\end{proof}

The value and gradient of $f$ are critical components for most modern
optimization algorithms and computing these values inexactly can derail or
even halt the progress of the algorithm.  In this setting, trust-region
methods are ideal as they can robustly handle inexact evaluations of
$f$ and $f'$, while maintaining strong convergence guarantees
\cite{Baraldi2025}. Trust-region methods ensure convergence by providing
tolerances for inexact computations that depend on the progress of the
optimization.  In the context of \eqref{eq:prob}, we will use adaptive
mesh refinement with residual-based error indicators to achieve these
tolerances by levaraging the error bounds provided by the following
proposition.
\begin{proposition} \label{prop1}
  Suppose \Cref{as:jac} holds and let $z\in\cN_{\cZ}$, $u=S(z)$,
  $\lambda=\Lambda(u,z)$, $\bar u\in\cN_{\cU}$ and $\bar\lambda\in\cV$.  Then,
  there exists $\alpha_i>0$, $i=0,1,\ldots,6$, such that the following state
  and adjoint bounds hold
  \begin{subequations}\label{ercnt}
  \begin{align}
    \alpha_0 \norm{u-\bar u}_{\cU} &\le \norm{c(\bar u,z)}_{\cV^*} \le \alpha_1\norm{u-\bar u}_{\cU} \label{ercnt-state} \\
    \norm{\lambda-\bar\lambda}_{\cV} &\le \alpha_2(1+\norm{\bar\lambda}_{\cV})\norm{c(\bar u,z)}_{\cV^*} + \alpha_3\norm{\cG(\bar u,z,\bar\lambda)}_{\cU^*}, \label{ercnt-adjoint}\\
    \intertext{as well as the objective function value and gradient bounds}
    |J(\bar u,z)-f(z)| &\le \alpha_4\norm{c(\bar u, z)}_{\cV^*} \label{ercnt-val} \\
    \norm{g(\bar u,z,\bar\lambda)-f'(z)}_{\cZ^*} &\le \alpha_5(1+\norm{\bar\lambda}_{\cV}) \norm{c(\bar u,z)}_{\cV^*} + \alpha_6 \norm{\cG(\bar u,z,\bar\lambda)}_{\cU^*}. \label{ercnt-grad}
  \end{align}
  \end{subequations}
\end{proposition}
\begin{proof}
  The proof of this result is similar to the proof of
  \cite[Proposition~A.2]{MR3981385}. We first bound the difference between $u$
  and $\bar u$.  To this end, the arguments in the proof of \Cref{prop:assum}
  apply here to yield \eqref{eq:state-res-bnd}. This and \Cref{as:jac}.8
  applied to \eqref{eq:res-int} yield \eqref{ercnt-state} with
  $\alpha_0=\sigma_0$ and $\alpha_1=\sigma_1$. Next we bound the difference
  between $\lambda$ and $\bar\lambda$. By the linearity of the adjoint
  residual, we have that
  \[
    \|\cG(u, z,\bar\lambda)\|_{\cU^*} = \|\cG(u, z,\bar\lambda)-\cG(u, z, \lambda)\|_{\cU^*}
      = \|c_u(u, z)^*(\bar\lambda-\lambda)\|_{\cU^*}
  \]
  and therefore \Cref{as:jac}.5~and~\ref{as:jac}.8 yield
  \[
    \sigma_0 \norm{\lambda - \bar \lambda}_{\cV} \le \norm{\cG(u, z, \bar\lambda)}_{\cU^*} 
    \le \sigma_1\norm{\lambda - \bar \lambda}_{\cV}.
  \]
  The bound \eqref{ercnt-val} follows from \Cref{as:jac}.8 and
  \eqref{ercnt-state}.  In particular, $J(\cdot,z)$ is Lipschitz continuous on
  $\cN_{\cU}$ since $J_u$ is uniformly bounded on $\cN_{\cU}$. Now, consider 
  \begin{align*}
    \norm{\cG(u, z, \bar\lambda)}_{\cU^*}
    & \le \norm{\cG(\bar u, z, \bar\lambda)}_{\cU^*} + \norm{\cG(u, z, \bar\lambda) - \cG(\bar u, z, \bar\lambda)}_{\cU^*} \\
    & \le \norm{\cG(\bar u, z, \bar\lambda)}_{\cU^*} + \norm{(c_u(u,z) - c_u(\bar u, z))^*\bar\lambda}_{\cU^*} \\
    & \quad \quad \quad + \norm{J_u(u,z) - J_u(\bar u, z)}_{\cU^*}\\
    & \le \norm{\cG(\bar u, z, \bar\lambda)}_{\cU^*} + (L_{c_u}\norm{\bar\lambda}_{\cV} + L_{J_u})\norm{u - \bar u}_{\cU}.
  \end{align*}
  Combining these two bounds yields \eqref{ercnt-adjoint}. Finally, we have
  that 
  \begin{align*}
    \norm{g(\bar u, z, \bar\lambda) - f'(z)}_{\cZ^*}
    & \le \norm{g(\bar u,z,\bar\lambda) - g(\bar u, z, \lambda)}_{\cZ^*} + \norm{g(\bar u, z, \lambda) - g(u, z, \lambda)}_{\cZ^*}\\
    & = \norm{J_z(\bar u, z) - J_z(u, z) + (c_z(\bar u,z)-c_z(u, z))^*\lambda}_{\cZ^*} \\
    &\quad \quad \quad + \norm{c_z(u, z)^*(\bar\lambda - \lambda)}_{\cZ^*} \\
    & \le (L_{J_z}+L_{c_z}\norm{\lambda}_{\cV})\norm{u - \bar u}_{\cU} + \sigma_1 \norm{\lambda - \bar\lambda}_{\cV}.
  \end{align*}
  Recall that $\|\lambda\|_{\cV}$ is bounded independent of $z$ using
  \Cref{as:jac}.8 and \eqref{eq:infsup2}.  Consequently, \eqref{ercnt-grad}
  follows from \eqref{ercnt-state} and \eqref{ercnt-adjoint}.
\end{proof}

%% file: adaptivefem.tex
Let $\Omega\subset\R^d$, $d=1,2,3$, be an open, connected set with boundary
$\partial\Omega$.
For \eqref{eq:prob}, the Banach spaces $\cU$ and $\cV$ consists of measurable
functions defined on $\Omega$. Let $\cT^h$ be a partition of the domain
$\Omega$ into regular triangles such that $\bar{\Omega}=\bigcup_{T\in\cT^h}T$.
The parameter $h$ is the cell-wise constant function defined by
$h\vert_T=h_T=\textup{diam}(T)$. Moreover, let $\cU^h \subset \cU$ and
$\cV^h\subset\cV$ be finite-dimensional subspaces of cell-wise defined
functions on the mesh $\cT^h$.  Moreover let $\cE^h$ denote the set of edges of
$\cT^h$. Let $\cE_\Omega^h\coloneqq \cE^h \backslash \{ E \in \cE^h~\vert~E\subset \partial \Omega\}$
denote the set of interior edges on domain $\Omega$.
Additionally, we denote the finite-element state and adjoint
approximations by $u^h\in\cU^h$ and $\lambda^h\in\cV^h$, respectively.  When
the optimization variable $z\in\cZ$ is defined on $\Omega$, we will leverage
nested, finite-dimensional (e.g., piecewise constant) approximations to ensure
convergence to an infinite-dimensional solution.  As such, we will not consider 
discretizations of $z$ within the AFEM error analysis.

We denote the discretized PDE constraint by $c^h:\cU^h\times\cZ\to(\cV^h)^*$,
which is given by the finite-element discretization
\[
  \inner{c^h(u^h, z)}{v^h}_{(\cV^h)^*,\cV^h} \coloneqq \inner{c(u^h, z)}{v^h}_{\cV^*,\cV}
\]
for all $u^h\in\cU^h$ and $v^h\in\cV^h$. Additionally, we define $g^h$, $S^h$,
and $\Lambda^h$ to be the finite-element approximations of $g$, $S$, and
$\Lambda$, respectively.  Given $z\in\cZ$, let $u^h = S^h(z)$ and
$\lambda^h = \Lambda^h(S^h(z), z)$.  We approximate $f'(z)$ by
\[
  g^h(z) \coloneqq g(u^h, z, \lambda^h) \approx f'(z) = g(S(z), z, \lambda(S(z), z)).
\]
\sloppy
According to \Cref{prop1}, we can bound the error between $g^h(z)$ and $f'(z)$
by $\|c(u^h,z)\|_{\cV^*}$ and $\|\cG(u^h,z,\lambda^h)\|_{\cU^*}$. However,
these infinite-dimensional residual norms are not directly computable. To
permit inexact state and adjoint solves as in \cite{MR3069101} (i.e., using
iterative linear/nonlinear solvers), we instead bound the state residual
norm as
\begin{align*}
  \norm{c(S^h(z), z)}_{\cV^*} & = \norm{c(S^h(z), z) - c^h(S^h(z),z) + c^h(S^h(z),z)}_{\cV^*} \\
  & \le \norm{c(S^h(z), z) - c^h(S^h(z),z)}_{\cV^*} + \norm{c^h(S^h(z), z)}_{(\cV^h)^*},
\end{align*}
and likewise for the adjoint residual. To further bound these quantities, we
assume that there exist \textit{reliable error estimators} $\xi_c^h$ and $\xi_\cG^h$ for
the state and adjoint, respectively, such that $\xi_{c}^h\to0$ and
$\xi_\cG^h\to0$ as $h\to0$ for arbitrary fixed $z\in\cN_{\cZ}$.  We can then
employ these estimators to produce the bounds
\begin{subequations}\label{eq:errest}
\begin{align}
  \norm{c(S^h(z),z)}_{\cV^*}
  & \leq C_1 \kappa_{c_1}\xi_{c}^h + \kappa_{c_2}\norm{c^h(S^h(z),z)}_{(\cV^h)^*} \label{est3} \\
  \norm{\cG(S^h(z),z,\lambda^h)}_{\cU^*} 
  & \leq C_2  \kappa_{\cG_1}\xi_{\cG}^h + \kappa_{\cG_2}\norm{\cG^h(S^h(z), z, \lambda^h)}_{(\cU^h)^*},  \label{est1}
\end{align}
\end{subequations}
with finite, unknown constants $\kappa_{c_i}>0$ and $\kappa_{\cG_i} > 0$ and positive constants $C_i >0 $ for $i=1,2$, independent of mesh parameter $h$.
Through adaptive refinement, we generate a hierarchy of meshes with parameter
$h_k$ at each optimization iteration $k$.  We denote the mesh parameter at each
sub-iteration within the refinement by $h_{k,\ell}$, which we then test against
the value $\tau^{\text{val}}_k>0$ and derivative $\tau^{\text{der}}_k$
tolerances provided by the trust-region algorithm. In particular, at the $k$-th
trust-region iteration, we refine the mesh as well as the state and adjoint
approximations to satisfy
\begin{subequations}
\begin{equation}
\label{eq:error_crit_state}
  \kappa_{c_1} \xi_c^{h_{k,\ell}} + \kappa_{c_2}\norm{c^{h_{k,\ell}}(u^{h_{k,\ell}}, z)}_{(\cV^{h_{k,\ell}})^*} \le \tau^{\text{val}}_k, 
\end{equation}
for the objective value at current $z=z_k$ and trial $z=z_k^+$
iterates, and for the derivative
\begin{align}
  \xi_c^{h_{k,\ell}} + \xi_\cG^{h_{k,\ell}} 
  + \norm{\cG^{h_{k,\ell}}(u^{h_{k,\ell}}, z_k, \lambda^{h_{k,\ell}})}_{(\cU^{h_{k,\ell}})^*}&  \nonumber \\
  + \norm{c^{h_{k,\ell}}(u^{h_{k,\ell}},z_k)}_{(\cV^{h_{k,\ell}})^*}
  & \le  \tau^{\text{der}}_k. \label{eq:error_crit_tot}
\end{align}
\end{subequations}
We summarize our mesh refinement procedure for evaluating the objective
function $f(z)$ and its derivative in \Cref{alg:afem-val,alg:afem-grad}. 
\begin{algorithm}[!ht]
\caption{AFEM Objective Function Evaluation}
\label{alg:afem-val}
\begin{algorithmic}[1]
\REQUIRE{Initial mesh $\cT^{h_{k,0}}$ with parameter $h_{k,0}\gets h_k$,
	 iterates $z_k$ and $z_k^+$, 
	 tolerances $\tau_{\max}^{\textup{val}}, \tau_k^{\textup{val}}>0$, 
	 and refinement fraction $\theta\in(0,1)$.}
\FOR{$\ell = 0,1,2,\dots$}
  \STATE{\textbf{Solve:} On $\cT^{h_{k,\ell}}$,
	       compute $u^{h_{k,\ell}}\approx S(z_k)$
         and $u_+^{h_{k,\ell}}\approx S(z_k^+)$ satisfying
         \[
           \begin{aligned}
           \norm{c^{h_{k,\ell}}(u^{h_{k,\ell}}, z_k)}_{(\cV^{h_{k,\ell}})^*}
             &\le \min\{\tau_{\max}^{\textup{val}},\tau_k^{\textup{val}}\} \\
           \norm{c^{h_{k,\ell}}(u_+^{h_{k,\ell}}, z_k^+)}_{(\cV^{h_{k,\ell}})^*}
             &\le \min\{\tau_{\max}^{\textup{val}},\tau_k^{\textup{val}}\}
           \end{aligned}
         \]
         }
  \STATE{\textbf{Estimate:} Evaluate 
	       $\xi_{c}^{h_{k,\ell}}$ and $\xi_{c,+}^{h_{k,\ell}}$ 
				 at $z_k$ and $z_k^+$, respectively.}
  \STATE{\textbf{Check:} If 
    $\xi_c^{h_{k,\ell}} \le \min\{\tau_{\max}^{\textup{val}},\tau_k^{\textup{val}}\}$ and  
    $ \xi_{c,+}^{h_{k,\ell}} \le \min\{\tau_{\max}^{\textup{val}},\tau_k^{\textup{val}}\}$, 
    {\bf break}.}
  \STATE{\textbf{Mark:} Select minimal subset
         $\cM_\ell\subset\cT^{h_{k,\ell}}$ for refinement via the 
         D\"orfler (bulk-chasing) strategy \cite{MR1393904}:
         \begin{align*}
           \sum_{T\in \cM_\ell} \xi_T^2 
           \;\ge\; \theta \sum_{T\in \cT^{h_{k,\ell}}} \xi_T^2,
         \end{align*}
	 where $\xi_T$ denotes the local element indicator on $T$.}
  \STATE{\textbf{Refine:} Refine marked elements $\cM_\ell$ (e.g.,
         using newest-vertex bisection \cite{MR1284252}) to obtain a
	 conforming mesh $\cT^{h_{k,\ell+1}}$.}
\ENDFOR
\RETURN{Approximate objective function evaluations at $z_k$ and $z_k^+$ via
  \[
    F_k(z_k) = J(u^{h_{k,\ell}},z_k) + \phi(z_k) \qquad\text{and}\qquad F_k(z_k^+)= J(u_+^{h_{k,\ell}},z_k^+)+\phi(z_k^+).
  \]}
\end{algorithmic}
\end{algorithm}
In the subsequent section, it will become clear why two objective function
values are approximated by \Cref{alg:afem-val}.  These values, $z_k$ and
$z_k^+$, coincide with the current and trial iterates generated by our
trust-region method.  To ensure sufficient decrease, we require that our
approximation of the difference of the objective function at these values is
sufficiently accurate.
\begin{algorithm}[!ht]
\caption{AFEM Derivative Evaluation}
\label{alg:afem-grad}
\begin{algorithmic}[1]
\REQUIRE{Initial mesh $\cT^{h_{k,0}}$ with parameter $h_{k,0}\gets h_k$,
	 $z_k$, 
	 tolerances $\tau_{\max}^{\textup{der}}, \tau_k^{\textup{der}}>0$, 
	 and refinement fraction $\theta\in(0,1)$.}
\FOR{$\ell = 0,1,2,\dots$}
  \STATE{\textbf{Solve:} On $\cT^{h_{k,\ell}}$, compute 
	       $u^{h_{k,\ell}}\approx S(z_k)$ and
         $\lambda^{h_{k,\ell}}\approx\Lambda(S(z_k),z_k)$ 
         satisfying
         \[
           \begin{aligned}
           \norm{c^{h_{k,\ell}}(u^{h_{k,\ell}}, z_k)}_{(\cV^{h_{k,\ell}})^*}
             &\le \min\{\tau_{\max}^{\textup{der}},\tau_k^{\textup{der}}\} \\
           \norm{\cG^{h_{k,\ell}}(u^{h_{k,\ell}}, z_k, \lambda^{h_{k,\ell}})}_{(\cU^{h_{k,\ell}})^*}
             &\le \min\{\tau_{\max}^{\textup{der}},\tau_k^{\textup{der}}\}. 
           \end{aligned}
         \]
         }
  \STATE{\textbf{Estimate:} Evaluate 
	$\xi_{c}^{h_{k,\ell}}$ and $\xi_{\cG}^{h_{k,\ell}}$ for the
	 state and adjoint at $z_k$, respectively.}
  \STATE{\textbf{Check:} If  
    $
      \xi_c^{h_{k,\ell}}+\xi_{\cG}^{h_{k,\ell}}\le\min\{\tau_{\max}^{\textup{der}},\tau_k^{\textup{der}}\},
    $ 
    {\bf break}.}
  \STATE{\textbf{Mark:} Select minimal subset
         $\cM_\ell\subset\cT^{h_{k,\ell}}$ for refinement via the 
         D\"orfler (bulk-chasing) strategy:
         \begin{align*}
           \sum_{T\in \cM_\ell} \xi_T^2 
           \;\ge\; \theta \sum_{T\in \cT^{h_{k,\ell}}} \xi_T^2,
         \end{align*}
	 where $\xi_T$ is the combined state and adjoint local element indicator on $T$.}
  \STATE{\textbf{Refine:} Refine marked elements $\cM_\ell$ (e.g.,
         using newest-vertex bisection) to obtain a
         conforming mesh $\cT^{h_{k,\ell+1}}$.}
\ENDFOR
\RETURN{Derivative approximation $g_k=g(u^{h_{k,\ell}},z_k,\lambda^{h_{k,\ell}})$ and the total error estimator
  \[
     \xi=\xi_c^{h_{k,\ell}}+\xi_{\cG}^{h_{k,\ell}}
	+\norm{c^{h_{k,\ell}}(u^{h_{k,\ell}}, z_k)}_{(\cV^{h_{k,\ell}})^*}
	+\norm{\cG^{h_{k,\ell}}(u^{h_{k,\ell}}, z_k, \lambda^{h_{k,\ell}})}_{(\cU^{h_{k,\ell}})^*}.
  \]}
\end{algorithmic}
\end{algorithm}
Notice that \Cref{alg:afem-val,alg:afem-grad} use the same refinement
mechanisms (i.e., bulk chasing and newest-vertex bisection).  The primary
difference between the two being that \Cref{alg:afem-val} only refines the
state approximation, while \Cref{alg:afem-grad} refines both the state and
adjoint approximations.

%% file: adaptivetrust.tex
At the $k$-th iteration, classical trust-region methods compute a trial iterate
$z_k^+$ by approximately minimizing a local model of the objective function
around the current iterate $z_k$ within the current trust region, i.e., the
ball of radius $\Delta_k>0$ centered at $z_k$.  To handle nonsmooth objective
functions with the form of $F$, \cite{MR4620236} employs the model
$m_k(z)\coloneqq f_k(z)+\phi(z)$, where $f_k:\cZ\to\R$ is a smooth
approximation of $f$ around $z_k$.  For computational convenience, we
utilize quadratic models of the form
\[
  f_k(z) \coloneqq \frac12\inner{B_k(z - z_k)}{z - z_k}_{\cZ^*,\cZ} + \inner{g_k}{z - z_k}_{\cZ^*,\cZ},
\]
where $B_k = B_k^*\in\cL(\cZ,\cZ^*)$ models curvature information (i.e., the
Hessian or a secant approximation thereof) and $g_k \approx f'(z_k)$.
As in \cite[Algorithm~1]{MR4620236}, we compute trial iterates $z_k^+$ that
approximately solve the subproblem
\begin{align}
  \min_{z \in \cZ}\;m_k(z) \quad \text{subject to} \quad \norm{z-z_k}_{\mathcal{Z}} \le \Delta_k, \label{problem1}
\end{align}
using the proximal methods introduced in \cite{MR4849972}.
In particular, we require that $z_k^+$ satisifies the trust-region constraint
(up to a constant)
\begin{align*} 
  \norm{z^+-z_k}_{\cZ} \le \krad \Delta_k
\end{align*} 
and the fraction of Cauchy decrease condition
\begin{align} \label{cdc}
  m_k(z_k)- m_k(z_k^+) \ge \kfcd \Psi_k \min \left \{\Delta_k, \frac{\Psi_k}{1+\|B_k\|_{\cL(\cZ,\cZ^*)}}\right \},
\end{align}
where $\Psi_k$, defined as
\begin{equation}\label{eq:psik}
  \Psi_k\coloneqq t_k^{-1}\norm{\prox{t_k\phi}(z_k-t_k \nabla f_k(z_k))-z_k}_{\cZ},
\end{equation}
indicates the stationarity of the iterate $z_k$ for the model $m_k$.
Here, $\krad>0$ and $\kfcd>0$ are constants independent of $k$.  In our
numerical examples, we choose the parameter $t_k$ to be Cauchy point step
length, as in \cite[Algorithm~5]{MR4620236} and \cite[Section~5.1]{MR4849972},
although one could simply take $t_k$ to be a constant. The fraction of Cauchy
decrease condition~\eqref{cdc} ensures that $z_k^+\in {\rm dom}\,{\phi}$ since
the left-hand side would be $-\infty$ otherwise.

Traditionally, one decides whether to accept or reject the trial iterate
$z_k^+$ based on the ratio of actual and predicted reductions
\begin{align}
\tilde\rho_k\coloneqq \frac{\ared_k}{\pred_k}\coloneqq\frac{F(z_k)-F(z_k^+)}{m_k(z_k)-m_k(z_k^+)}.
\end{align}
Concretely, given user-specified parameters $0 < \eta_1 < \eta_2 <1$, we
accept the trial iterate $z_{k+1}=z_k^+$ if $\tilde\rho_k \ge \eta_1$ and
otherwise we reject the trial iterate $z_{k+1}=z_k$.  Moreover, if
$\tilde\rho_k < \eta_1$, we decrease the trust-region radius, i.e.,
$\Delta_{k+1}<\Delta_k$, and if $\tilde\rho_k\ge\eta_2$, we increase the
trust-region radius, i.e., $\Delta_{k+1}>\Delta_k$.

For \eqref{eq:prob}, exact evaluations of the objective function $F$
are not possible since they require the solution of a system of PDEs
to evaluate $S(z)$ and hence $f(z)=J(S(z),z)$.  For this problem, we replace
$F$ with an AFEM approximation $F_k$ and employ the computed
reduction 
\begin{align*}
  \cred_k \coloneqq F_k(z_k) - F_k(z_k^+).
\end{align*}
For step acceptance and trust-region radius update, we replace $\tilde\rho_k$
with the ratio of computed and predicted reductions
\begin{align} 
\label{comp}
  \rho_k \coloneqq \frac{\cred_k}{\pred_k}.
\end{align}
In order to guarantee convergence of the trust-region method, we require that
the AFEM approximation of the objective function is sufficiently accurate as
defined in the following condition.
\begin{condition}[Inexact Value] 
\label{as:objred}
There exists a constant $\kval>0$, independent of $k$, such that
the AFEM approximation of the objective function value satisfies
\begin{align} \label{objred1}
  |\ared_k-\cred_k| \leq \kval \zeta_k
  \quad\text{and}\quad \zeta_k^j \leq \gamma \min \{ \pred_k , \epsilon_k \}
	\quad \forall\,k.
\end{align}
Here, $j \in (0,1)$ and $\gamma<\min\{\eta_1,1-\eta_2\}$ are independent of $k$,
and $\{\epsilon_k\}\subset[0,+\infty)$ satisfies $\epsilon_k\to 0$.
\end{condition}
We satisfy \Cref{as:objred} using the mesh refinement algorithm listed as
\Cref{alg:afem-val} with
$\tau_k^{\textup{val}}=\bar{\kappa}_{\textup{val}}[\eta\min\{\pred_k,\epsilon_k\}]^{1/j}$
for some postive user-specified constant $\bar{\kappa}_{\textup{val}}>0$.

As with the evaluation of $f$, the derivative $f'$ also cannot be compute
exactly.  However, the model derivative $g_k$ will heavily influence the search
direction and overall performance of the algorithm.  For this reason, we
require that the approximation $g_k\coloneqq f_k'(z_k) \approx f'(z_k)$
satisfies the following accuracy requirement.
\begin{condition}[Inexact Derivative]
\label{as:grad}
There exists $\kgrad>0$, independent of $k$, such that the
model derivative $g_k$ satisfies
\begin{equation} 
\label{eq:grad}
  \norm{g_k - f' (z_k)}_{\mathcal{Z}^*} \le \kgrad \min \{\Psi_k, \Delta_k\} \quad \forall\, k.
\end{equation}
\end{condition}
Note that the quantities in the right-hand side of the inexact objective
function condition \eqref{objred1} are available at the time of evaluating the
computed reduction, while the right-hand side of the inexact derivative
condition \eqref{eq:grad} depends on $\Psi_k$, which in turn depends on $g_k$.
Consequently, \eqref{eq:grad} must be satisfied iteratively as described in
\Cref{alg:afem-grad-it}.
\begin{algorithm}[!ht]
\caption{Iterative AFEM Derivative Evaluation}
\label{alg:afem-grad-it}
\begin{algorithmic}[1]
\REQUIRE{
Initial mesh $\cT^{h_{k,0}}$ with parameter $h_{k,0}\gets h_k$,
	 $z_k$, 
	 tolerances $\tau_{\max}^{\textup{der}}, \tau_k^{\textup{der}}>0$, 
	 refinement fraction $\theta\in(0,1)$,
   $\bar{\kappa}_{\textup{der}}>0$, and 
   $\Delta_k$.}
\STATE{\textbf{Set tolerance:} Set $\tau=\bar{\kappa}_{\textup{der}}\Delta_k$
       and choose $\xi > \tau$.}
\WHILE{$\xi > \tau$}
  \STATE{\textbf{Derivative computation:} Compute $g_k \approx f'(z_k)$
    and error $\xi$ using \Cref{alg:afem-grad} with tolerance
    $\tau_k^{\textup{der}}=\tau$.}
  \STATE{\textbf{Tolerance update:} Compute $\Psi_k$
    \eqref{eq:psik} and update 
    $\tau = \bar{\kappa}_{\textup{der}}\min\{\Psi_k,\Delta_k\}$.}
\ENDWHILE
\RETURN{The approximate derivative $g_k$.}
\end{algorithmic}
\end{algorithm}

Combining \Cref{alg:afem-val,alg:afem-grad-it} with \cite[Algorithm~2]{MR4620236},
we arrive at the trust-region AFEM solution procedure for \eqref{eq:prob}
listed in \Cref{alg}.
\begin{algorithm}[!ht]
  \caption{Trust-region AFEM algorithm }
  \label{alg}
\begin{algorithmic}[1]
  \REQUIRE{ 
  Inital guess $z_0 \in \text{dom}\,\phi$ and 
	radius $\Delta_0>0$, 
	and constants $0< \eta_1 <\eta_2 <1$,
  $0 < \gamma_1 \le \gamma_2 <1 \le \gamma_3$, $\theta\in(0,1)$,
  $\bar{\kappa}_{\textup{val}}>0$, $\bar{\kappa}_{\textup{der}}>0$,
  $\tau_{\max}^{\textup{val}}>0$, $\tau_{\max}^{\textup{der}}>0$,
  $0 < \gamma < \min\{\eta_1,1-\eta_2\}$ and $j\in(0,1)$, and the positive
  sequence $\{\epsilon_k\}\subset(0,\infty)$.} 
  \FOR{$k = 0,1,2, \cdots$}
    \STATE{\textbf{Model selection:} 
      Compute $g_k$ via \Cref{alg:afem-grad-it} and $B_k=B_k^*\in\cL(\cZ,\cZ^*)$.}
    \STATE{ \textbf{Step computation:} 
      Compute trial iterate $z_k^+$ satisfying \eqref{cdc}.}
    \STATE{\textbf{Step acceptance:} 
      Evaluate $\cred_k$ via \Cref{alg:afem-val} with
      $\tau_k^{\textup{val}}=\bar{\kappa}_{\textup{val}}[\eta\min\{\pred_k,\epsilon_k\}]^{1/j}.$}
    \IF{$\rho_k < \eta_1$}
      \STATE{$z_{k+1} \leftarrow z_k$}
      \STATE{$\Delta_{k+1}\in[\gamma_1\Delta_k,\gamma_2\Delta_k]$}\label{alg:rad_reduct}
    \ELSE
      \STATE{$z_{k+1} \leftarrow z^+$}
      \IF{$\rho_k\in[\eta_1,\eta_2)$}
        \STATE{$\Delta_{k+1}\in[\gamma_2\Delta_k,\Delta_k]$}
      \ELSE
        \STATE{$\Delta_{k+1}\in[\Delta_k,\gamma_3 \Delta_k]$}
      \ENDIF
    \ENDIF
  \ENDFOR
\end{algorithmic}
\end{algorithm}
The subsequent result summarizes the various global convergence results in
\cite[Theorems~1~and~2]{MR4849972}. 
\begin{theorem}[Convergence of \Cref{alg}]\label{th:conv}
  Let $\{z_k\}\subset\cZ$ be the sequence of iterates generated by \Cref{alg}
  and let \Cref{as:jac} holds. 
	In addition, suppose that the sequence of model 
	Hessians $\{B_k\}\subset\cL(\cZ, \cZ^*)$ satisfy
  \begin{equation}\label{eq:hessbnd}
    \sum_{k=1}^\infty \frac{1}{1 + \displaystyle{\max_{i=1,\ldots,k}} \|B_i\|_{\cL(\cZ,\cZ^*)}}=+\infty
  \end{equation}
  and the sequence of AFEM adjoint variables $\{\lambda^{h_k}\}$ is bounded
  by $\kadj>0$. Then
  \begin{equation*}
    \liminf_{k\to\infty} \; \Psi_k = 0 
    \qquad\text{and}\qquad
    \liminf_{k\to\infty} \; \Psi(z_k,t_k) = 0.
  \end{equation*}
  Moreover, if there exists $t_{\max}>0$ such that $t_k\le t_{\max}$ for all
  $k$, then
  \[
    \liminf_{k\to\infty} \; \Psi(z_k,t) = 0 \quad\forall\, t>0.
  \]
  Finally, if there exists $\kcurv>0$ such that
  $\|B_k\|_{\cL(\cZ,\cZ^*)} \le\kcurv$ for all $k$, then
  \begin{equation*}
    \lim_{k\to\infty} \; \Psi_k = 0 
    \qquad\text{and}\qquad
    \lim_{k\to\infty} \; \Psi(z_k,t) = 0 \quad\forall\, t>0.
  \end{equation*}
\end{theorem}
\begin{proof}
  This results follows \cite[Theorems~1~and~2]{MR4849972} if
  \Cref{as:objred,as:grad} are satisfied by \Cref{alg:afem-val,alg:afem-grad}.
  However, \Cref{as:objred,as:grad} are verified by \Cref{prop1} and
  \eqref{eq:errest}.  In particular, \eqref{objred1} holds with
  $\kval=2\alpha_4\max\{\kappa_{c_1},\kappa_{c_2}\}\bar{\kappa}_{\textup{val}}$
  and \eqref{eq:grad} holds with
  $\kgrad=3\max\{\alpha_5(1+\kadj),\alpha_6\}\max\{\kappa_{c_1},\kappa_{c_2},\kappa_{\cG_1},\kappa_{\cG_2}\}\bar{\kappa}_{\textup{der}}$.
\end{proof}
Note that the Hessian growth condition \eqref{eq:hessbnd} is weaker than
the typical requirement of bounded model Hessians,
cf.~\cite{MR4620236,MR933734,MR967689},  and is satisfied for certain classes
of problems when using, e.g., safeguarded secant approximations
\cite{MR933734}.  Moreover, the approximate adjoints are typically bounded
since the sequence of true adjoints is bounded via \Cref{as:jac}.  This is the
case when, e.g., the discretized adjoint Jacobians $c_u^{h_k}(u,z)^*$ satisfy
the discrete analogue of the uniform inf-sup condition \eqref{eq:infsup2} and
the discretized states satisfy $u^{h_k}\in\cN_{\cU}$.  In addition,
\Cref{th:conv} ensures that the iterates of \Cref{alg} will satisfy
\[
  \Psi_k \le \varepsilon
\]
for fixed $\varepsilon>0$ after only finitely many iterations.  For worst-case
complexity analysis, see \cite{MR4620236,leconte2025complexity}. Furthermore,
under additional assumptions, one can prove that the \Cref{alg} iterates
converge superlinearly, even quadratically
\cite[Theorem~3]{kouri2023localconv}.

%% file: numerical_examples.tex
We now illustrate our approach on two numerical examples.  The first is a
sparse-control problem governed by Poisson's equation on the L-shaped domain. 
The second is the more challenging, heat-conduction topology optimization 
problem studied in \cite{MR2208515}. Throughout this section, we denote the
physical domain by $\Omega\subset\R^2$ and the number of degrees of freedom
(DoFs) by $N_{\rm dof}$. For both examples, $\cU=\cV$, $\cU^h=\cV^h$, and
$\cZ=L^2(\Omega)$. Moreover, we represent the control $z\in \cZ$ at each
iteration by a piecewise constant function defined on the associated mesh
$\cT^h$, thus generating a sequence of controls defined on a sequence of nested
meshes. For all results, we terminate \Cref{alg} when $\Psi_k\le 10^{-6}$ and
we set the mesh refinement parameter $\theta=0.05$.  Moreover, we set
$\Delta_0=50$, $\eta_1=0.05$, $\eta_2=0.9$, $\gamma_1=0.25$, $\gamma_2=1.0$,
$\gamma_3=2.5$, $\tau_{\max}^{\textup{val}}=1.0$,
$\tau_{\max}^{\textup{der}}=1.0$, $\gamma= \epsilon_k = 1-10^{-3}$, and $j=0.9$.
We specify the values of $\bar{\kappa}_{\textup{val}}$ and
$\bar{\kappa}_{\textup{der}}$ in each example.

\subsection{A Posteriori Error Estimation for Elliptic PDEs}
Since both applications are governed by elliptic PDEs, we now review a
posteriori error estimation for the general elliptic PDE
\begin{subequations}\label{eq:elliptic}
\begin{align}
  -\nabla\cdot(A\nabla u) + bu &= q \quad \text{in }\Omega \\
  u &= 0 \quad \text{on }\Gamma_D \qquad \mbox{and} \qquad
  (A\nabla u)\cdot n = 0 && \text{on }\Gamma_N,
\end{align}
\end{subequations}
where $\Gamma_D$ and $\Gamma_N$ are disjoint subsets of the boundary
$\partial\Omega$ satisfying $\partial\Omega=\Gamma_D\cup\Gamma_N$.  Moreover,
$A:\Omega\to\R^{2\times 2}$ is a symmetric matrix function that is uniformly
coercive and bounded, i.e., there exists $0 < a_{\min} \le a_{\max} < \infty$
such that
\[
  v^\top A v \ge a_{\min} \quad\text{and}\quad \|A\|\le\alpha_{\max}
  \;\;\text{a.e.},
\]
and $b\in L^\infty(\Omega)$ is nonnegative a.e. We denote the solution space
for \eqref{eq:elliptic} by
\[
  \cU = H^1_{\Gamma_D}(\Omega) \coloneqq \{u\in H^1(\Omega)\,\vert\,u\vert_{\Gamma_D} = 0\}
\]
and assume that $q\in\cU^*$.  We further assume that either
$\Gamma_D\neq\emptyset$ or that there exists $\Omega_b\subseteq\Omega$ with
postive measure on which $b>0$ a.e. Under these assumptions, the bilinear form
associated with the weak form of \eqref{eq:elliptic} is symmetric,
continuous and coercive.  Consequently, the Lax-Milgram Lemma guarantees
existence and uniqueness of solutions.  We discretize \eqref{eq:elliptic}
using continuous finite elements defined on a mesh $\cT^h$.  For each mesh
element $T\in\cT^h$, we denote the space of degree $s\in\mathbb{N}$ polynomials
on $T$ by $\mathbb{P}_s(T)$ and the finite-element solution space by
\begin{equation}\label{eq:fe_space}
  \cU^h = \{ v \in C(\bar{\Omega})\cap\cU~\vert~ v|_T \in \mathbb{P}_s(T) ~\forall~ T \in \cT^h \}.
\end{equation}
Let $u\in\cU$ denote the solution to \eqref{eq:elliptic} and $u^h\in\cU^h$ its
finite-element approximation. As in \cite{MR4793681}, there exists a positive
constant $C_3$, independent of mesh parameter $h$, such that the error between
$u$ and $u^h$ satisfies
\begin{align} \label{reliability-U}
  \|u-u^h\|_{\cU}^2 \leq C_3 \sum_{i=1}^{3} (\xi_{\cU,i}^h)^2,
\end{align}
where $\xi_{\cU,i}^h$, $i=1,2,3$, are \textit{reliable a posteriori} error
estimators for the volume residual
\begin{subequations}
\begin{equation*}
  (\xi_{\cU,1}^h)^2\coloneqq \sum_{T \in \cT^h}~h_T^2 \|q+\nabla\cdot(A\nabla u^h) - bu^h\|^2_{L^2(T)},
\end{equation*}
and the edge jump residuals
\begin{equation*}
  (\xi_{\cU,2}^h)^2\coloneqq \sum_{e \in \cE^h_\Omega}~h_e \|[(A\nabla u^h)\cdot n]\|_{L^2(e)}^2
  \quad\text{and}\quad
  (\xi_{\cU,3}^h)^2\coloneqq \sum_{e \in \cE^h_{\partial\Omega}}~h_e\|(A\nabla u^h) \cdot n\|_{L^2(e)}^2.
\end{equation*}
\end{subequations}
Here, $\cE^h_{\partial\Omega}\coloneqq\{E\in\cE^h\,\vert\,E\in\Gamma_N\}$
denotes the set of Neumann edges and $[\cdot]$ denotes the jump across interior
edge $e$.  Similarly, as in \cite{MR3022214,MR1270622}, the following
\textit{reliable a posteriori} $L^\infty$-error estimate holds: for a positive
constant $C_4 >0$, independent of $h$, we have
\begin{align} \label{reliability}
  \|u-u^h\|_{L^\infty(\Omega)} \leq C_4 |\log h|^2 \sum_{i=1}^{3} \xi_{\infty,i}^h,
\end{align}
where the volume residual $\xi_{\infty,1}^h$ is
\begin{subequations}
\begin{equation*}
  \xi_{\infty,1}^h\coloneqq \underset{T \in \cT^h}{\max}~h_T^2 \|q+\nabla\cdot(A\nabla u^h) - bu^h\|_{L^{\infty}(T)}
\end{equation*}
and the edge jump residuals are
\begin{equation*}
  \xi_{\infty,2}^h\coloneqq \underset{e \in \cE^h_\Omega }{\max}~h_e \|[(A\nabla u^h)\cdot n]\|_{L^{\infty}(e)}
  \quad\text{and}\quad
  \xi_{\infty,3}^h\coloneqq \underset{e \in \cE^h_{\partial\Omega}}{\max}~h_e\|(A\nabla u^h)\cdot n\|_{L^{\infty}(e)}.
\end{equation*}
\end{subequations}

\subsection{Poisson Control on the L-Shaped Domain}\label{sec:poisson}
\label{subscn:ncvxdom}
Let $\Omega \subset \mathbb{R}^2$ be the nonconvex L-shaped domain and
consider the following distributed optimal control problem:
\begin{equation}
\label{eq:poisson_costfunctional}
\min_{z\in\cZ} \;
\frac12 \|S(z)-u_d\|^2_{L^2(\Omega)}
+ \frac{\alpha}{2}\|z\|^2_{L^2(\Omega)}
+ \beta\|z\|_{L^1(\Omega)},
\end{equation}
where the state $S(z) = u\in\cU=H^1_0(\Omega)$ solves the Poisson
equation
\begin{subequations} \label{eq:poissonequation}
\begin{align}
  -\Delta u = z \quad \text{in } \Omega, \qquad \mbox{and} \qquad
          u =0 \quad \text{on } \partial\Omega,
\end{align}
\end{subequations}
and the regularization parameters are $\alpha=10^{-4}$ and $\beta=10^{-2}$. 
With respect to \eqref{eq:prob}, $J(S(\cdot),\cdot)$ is the data misfit and
$L^2$ regularization and $\phi$ is the $L^1$ regularization. 
For the discretized state $u^h$, we utilize the computable error
estimators in \eqref{reliability-U} with $A=\mathbb{I}$ the identity matrix,
$b\equiv 0$, $q=z$ and $\Gamma_D=\partial\Omega$.  Similarly, we use
\eqref{reliability-U} for the discretized adjoint $\lambda^h$, where again
$A=\mathbb{I}$, $b\equiv 0$ and $\Gamma_D=\partial\Omega$, but $q=u^h-u_d$.
Note that for this example, $\xi^h_{\cU,3}=0$ for both the state and adjoint.
We discretize \eqref{eq:poissonequation} using quadratic finite elements on
a triangular mesh because, as observed in \cite[7.1.2]{MR2765487}, 
the error estimators for linear elements decay too slowly; 
see the first row of \Cref{tbl:poissonControl} for a depiction of the initial grid.
We start with 225 state DoFs
and cap refinement at 10,000 DoFs. 
We additionally set the inexact constants to be
$\bar{\kappa}_{\textup{val}}=\bar{\kappa}_{\textup{der}}=10^6$.
\begin{figure}[!ht]
\centering
\begin{tabular}{@{}c|c|@{}c@{}c@{}c}
\toprule
$N_{\rm dof}$ & $k$   & Grid & $z$ & $u^h$ \\
\hline
225          & 0 & \rowincludegraphics[width=.25\linewidth]{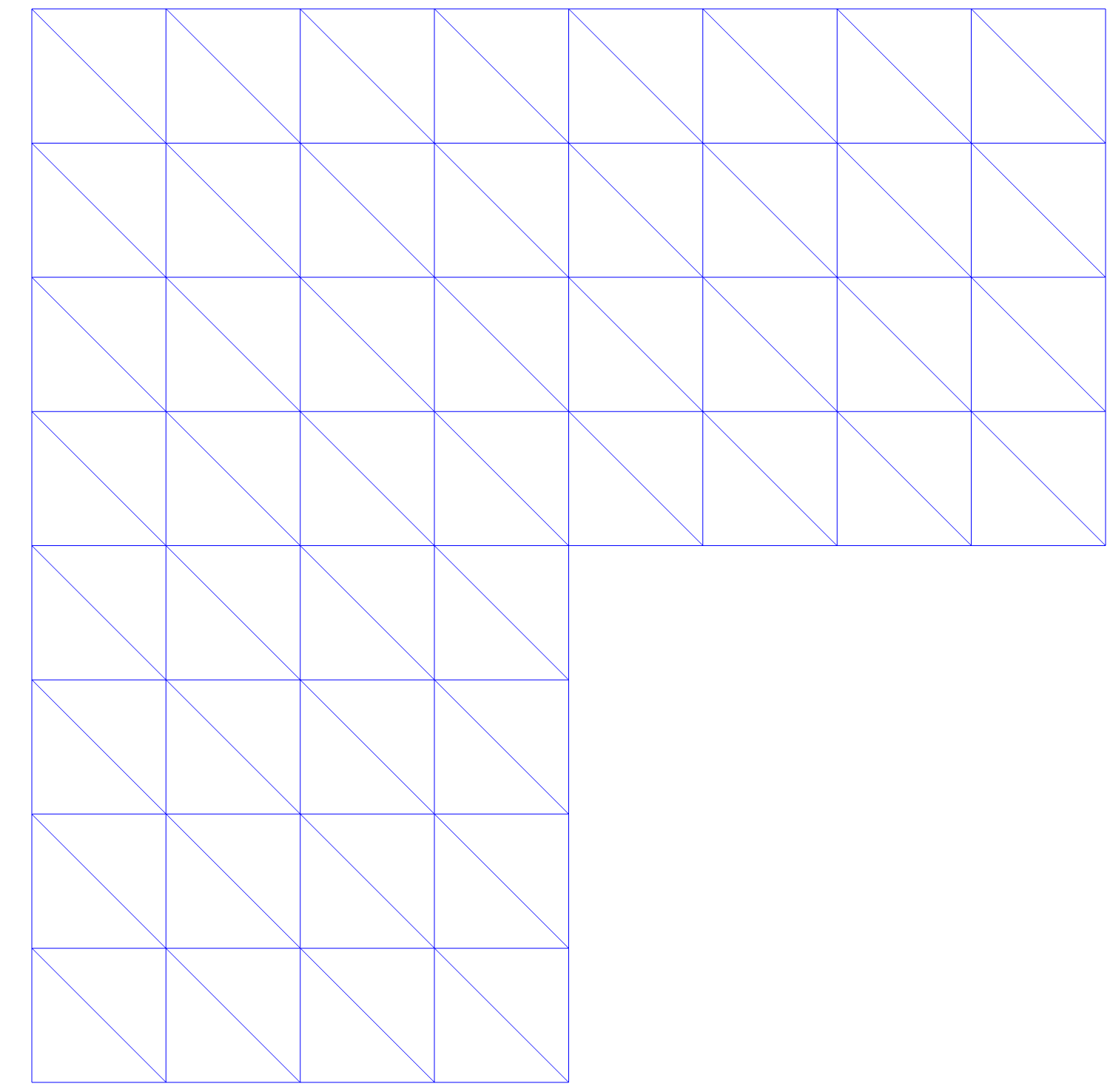}
                 & \rowincludegraphics[width=.25\linewidth]{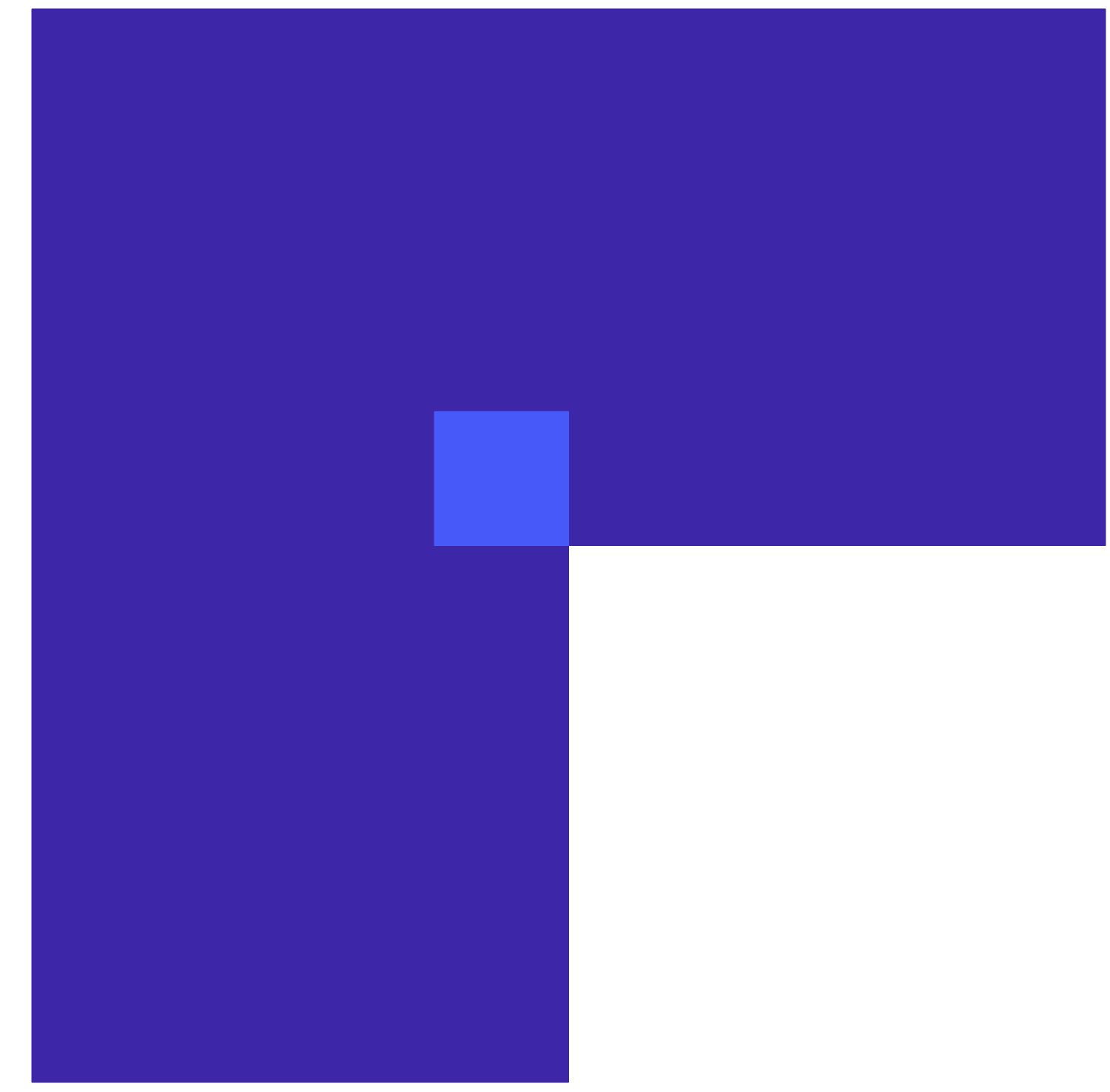}
                 & \rowincludegraphics[width=.25\linewidth]{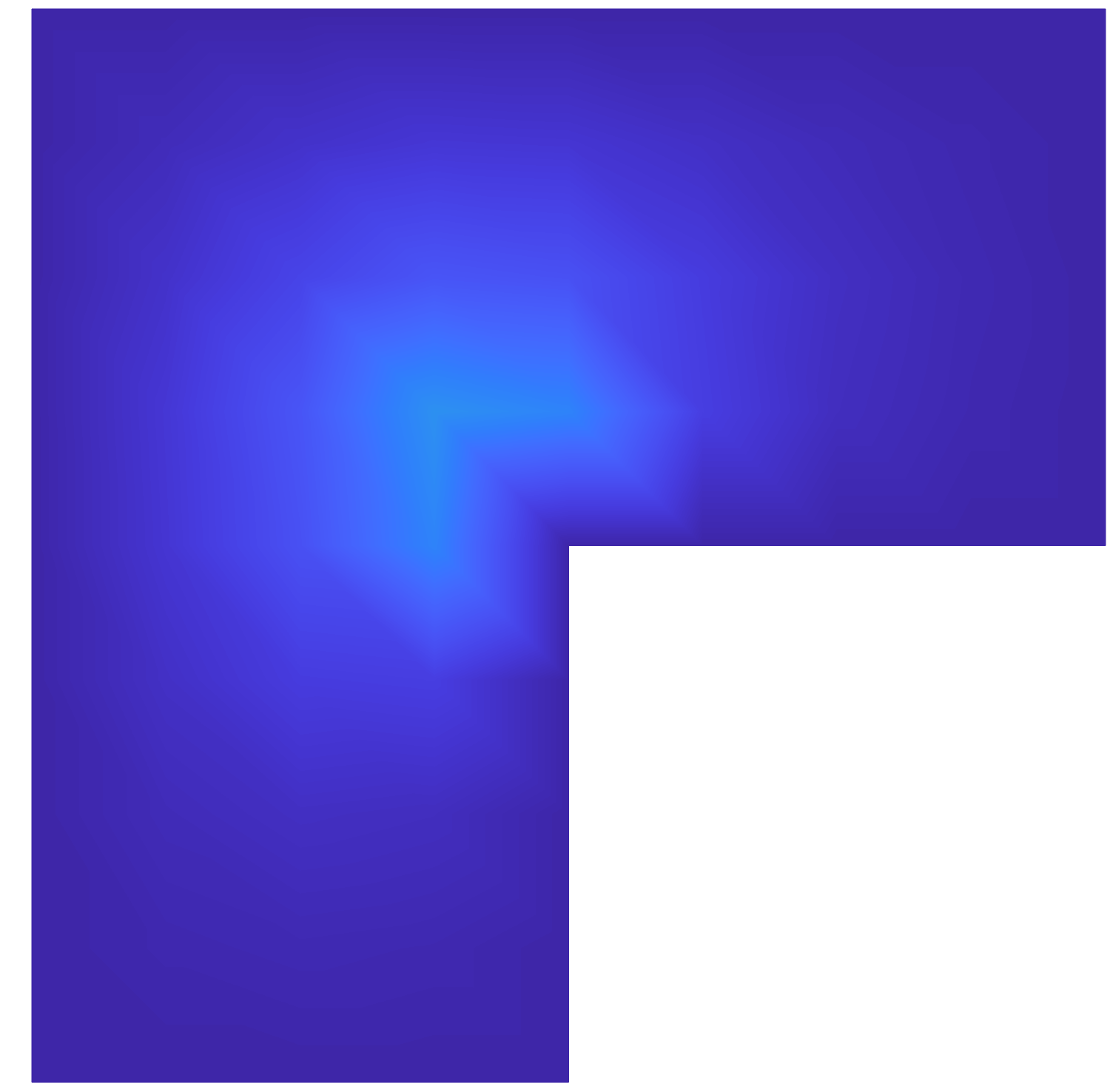}\\
600          & 4 & \rowincludegraphics[width=.25\linewidth]{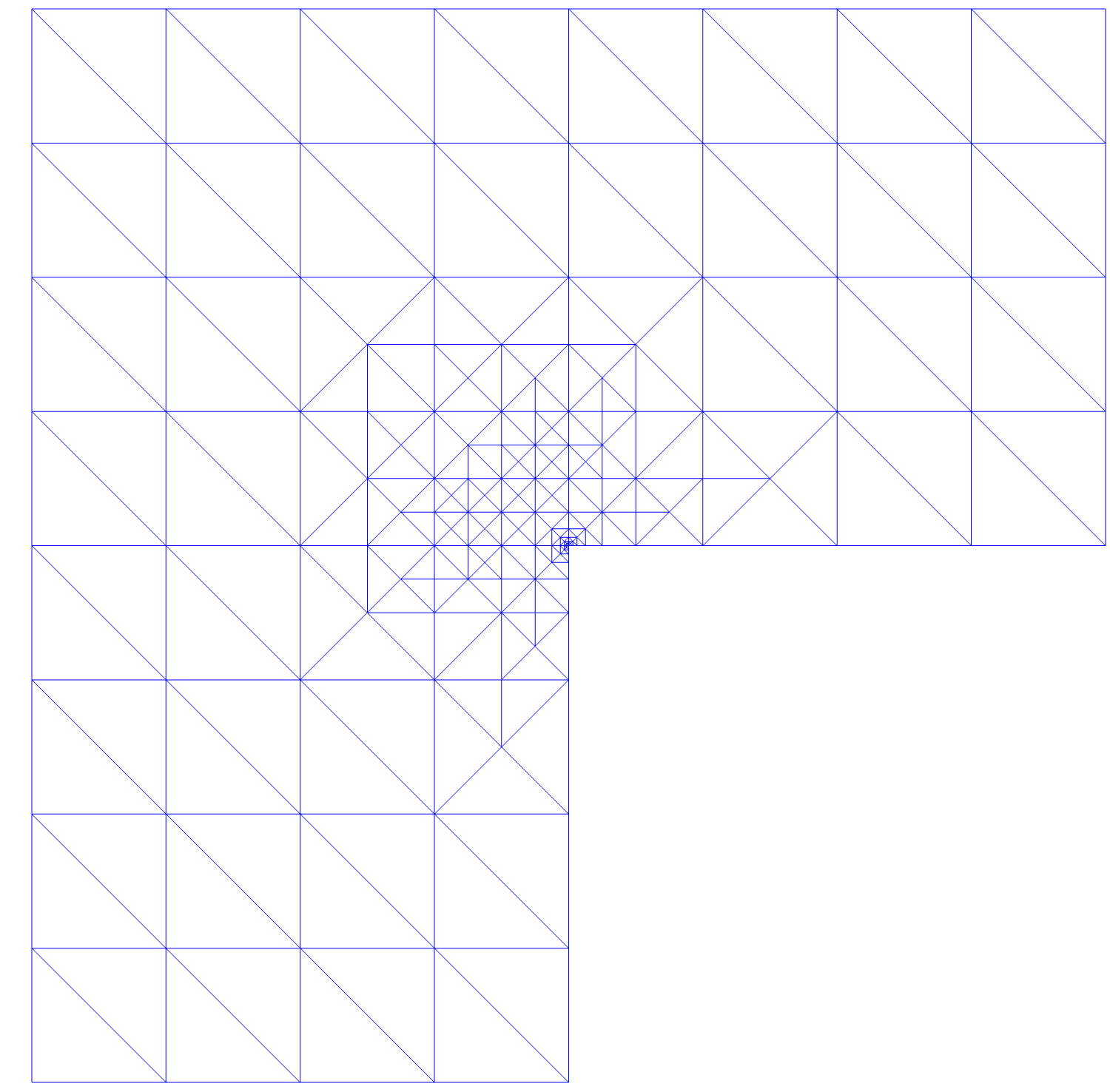}
                 & \rowincludegraphics[width=.25\linewidth]{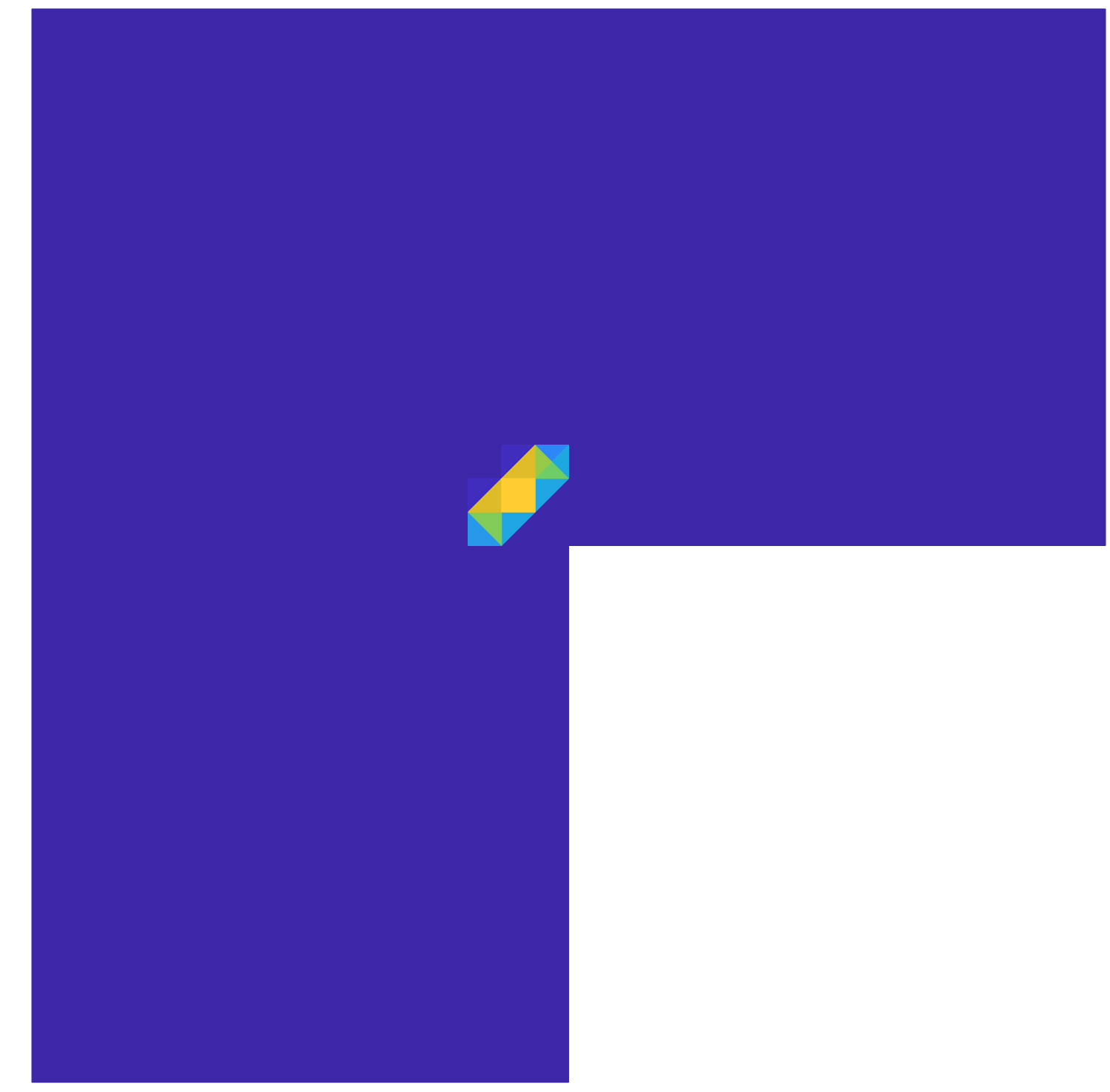}
                 & \rowincludegraphics[width=.25\linewidth]{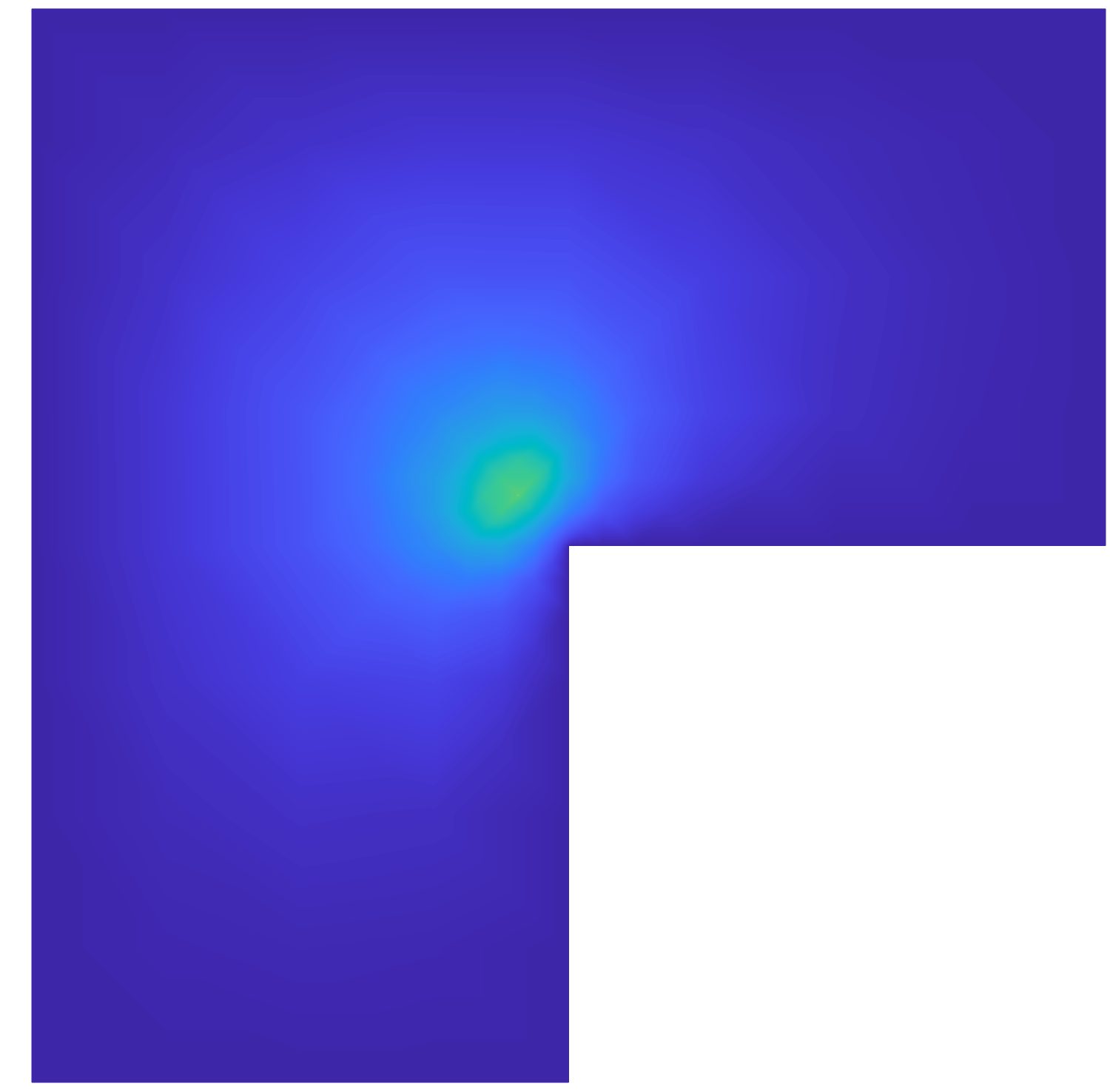}\\
$2\cdot10^3$ & 5 & \rowincludegraphics[width=.25\linewidth]{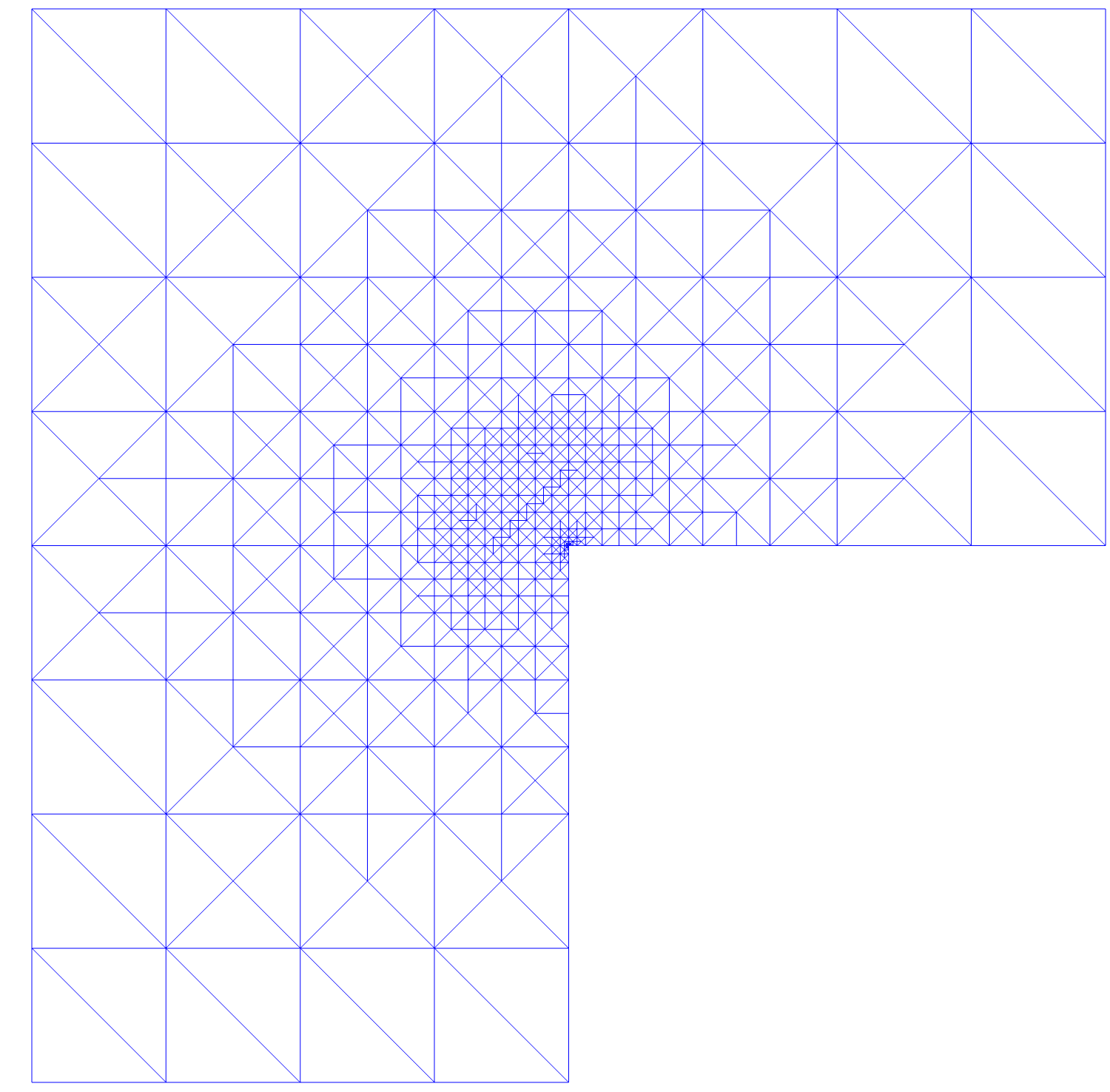}
                 & \rowincludegraphics[width=.25\linewidth]{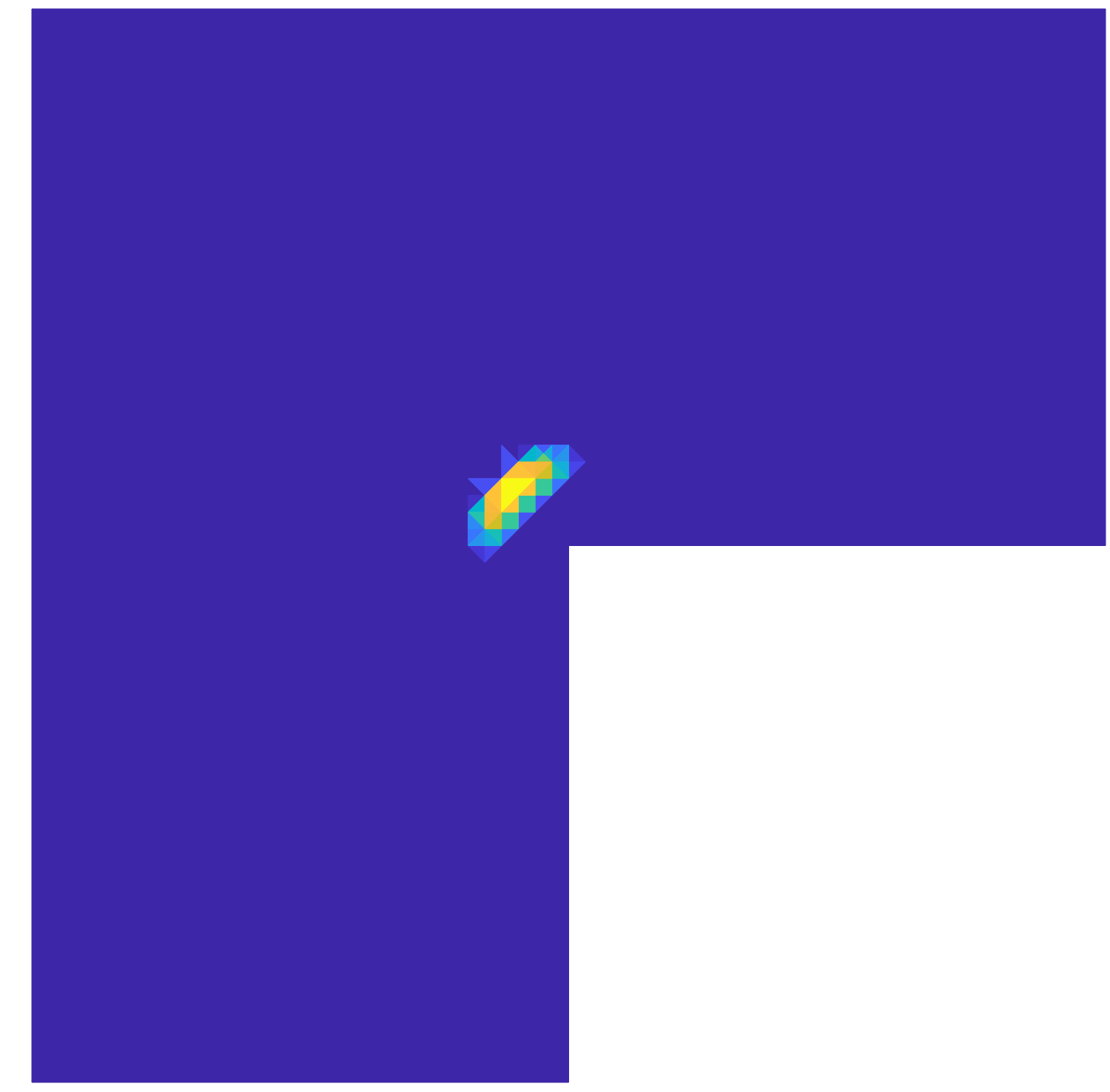}
                 & \rowincludegraphics[width=.25\linewidth]{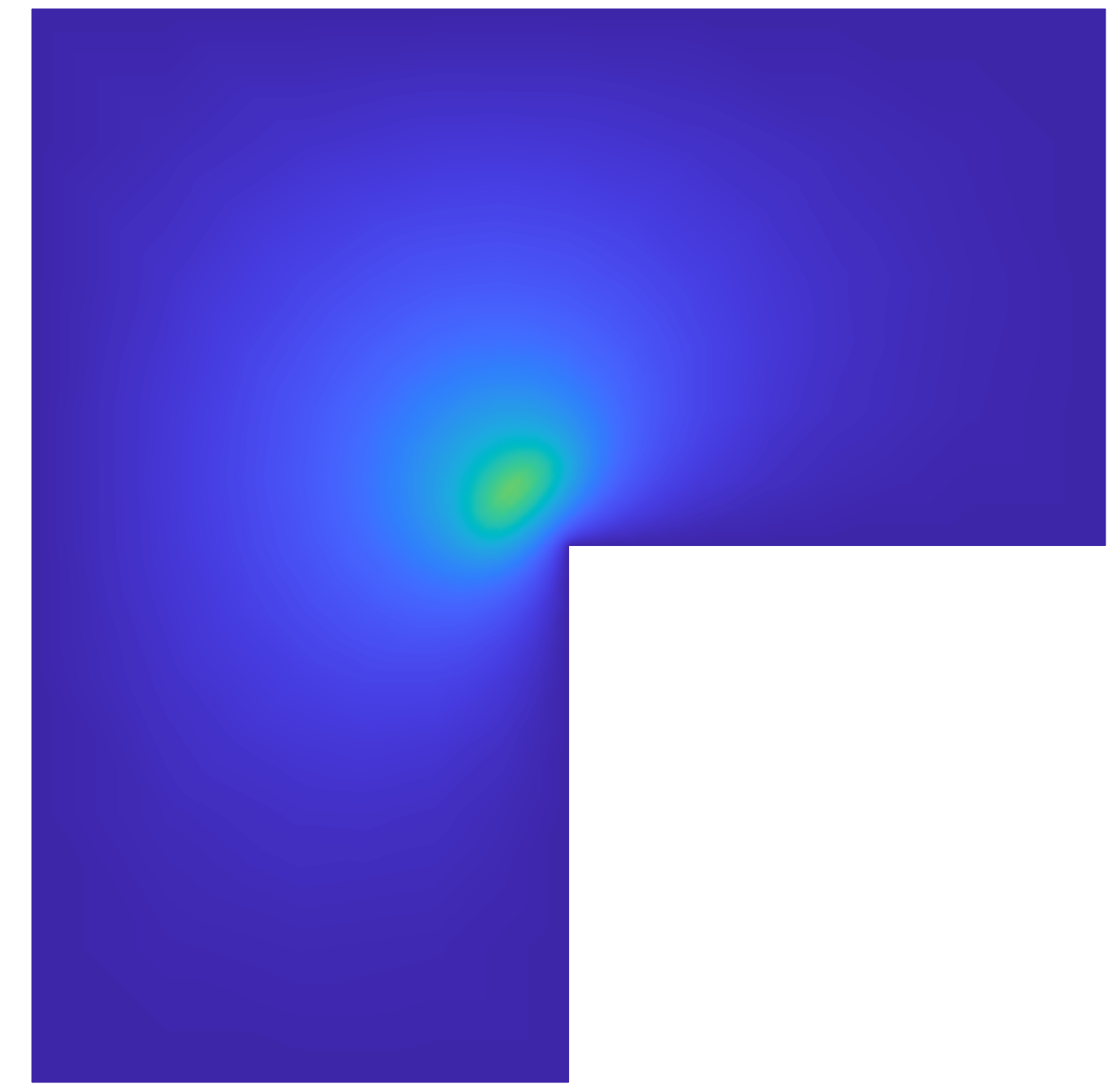}\\
$10^4$       & 8 & \rowincludegraphics[width=.25\linewidth]{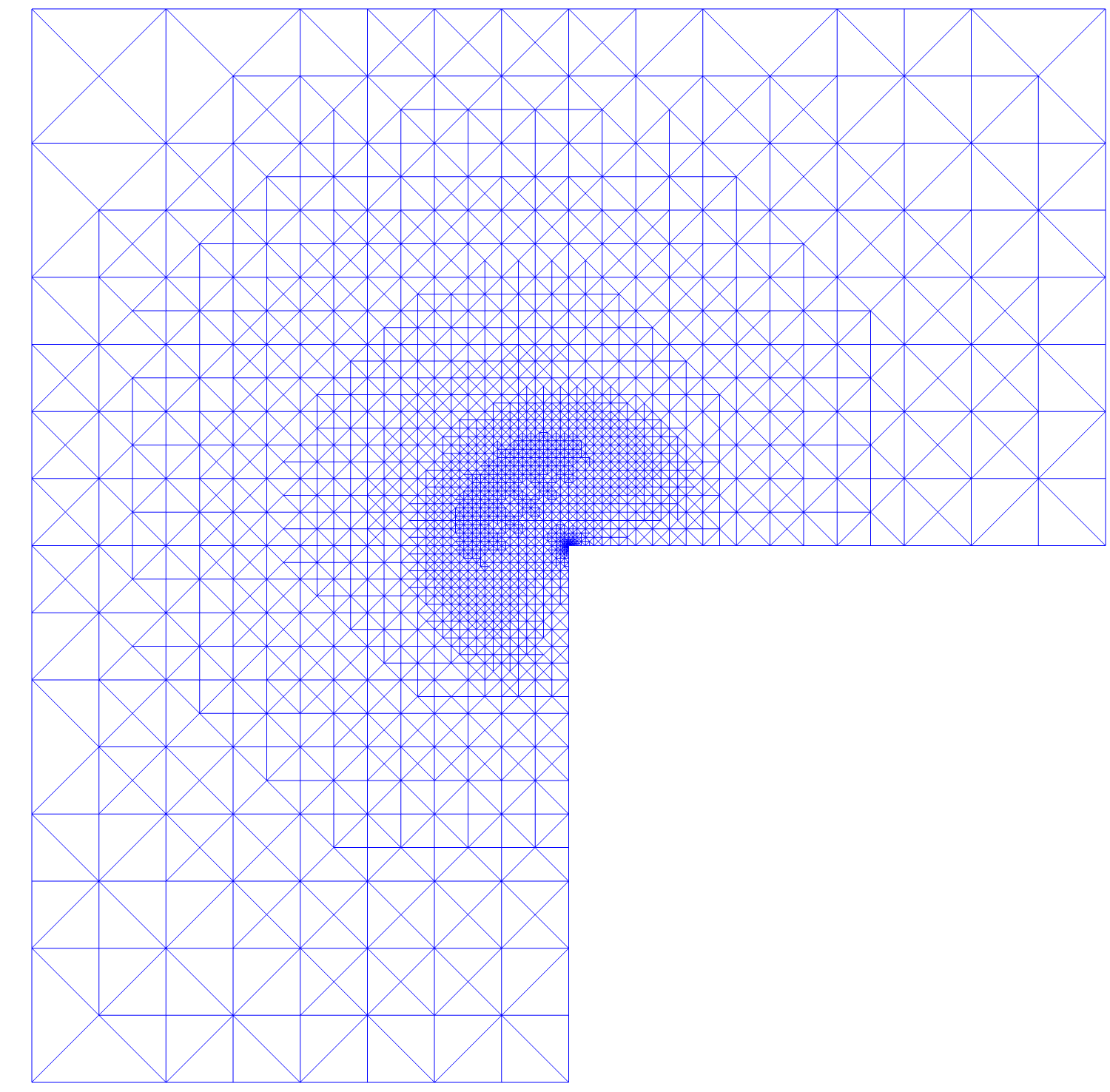}
                 & \rowincludegraphics[width=.25\linewidth]{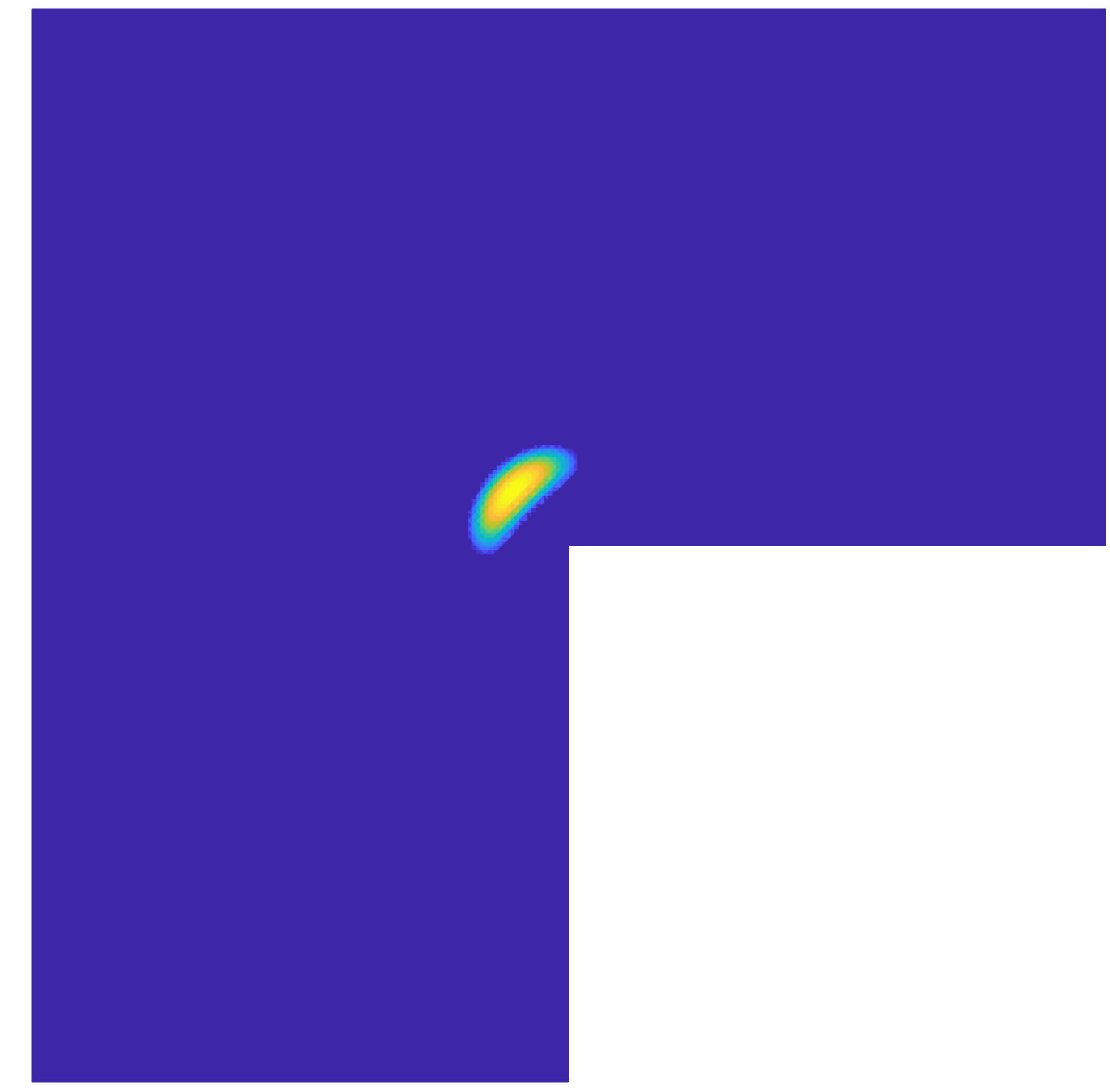}
	               & \rowincludegraphics[width=.25\linewidth]{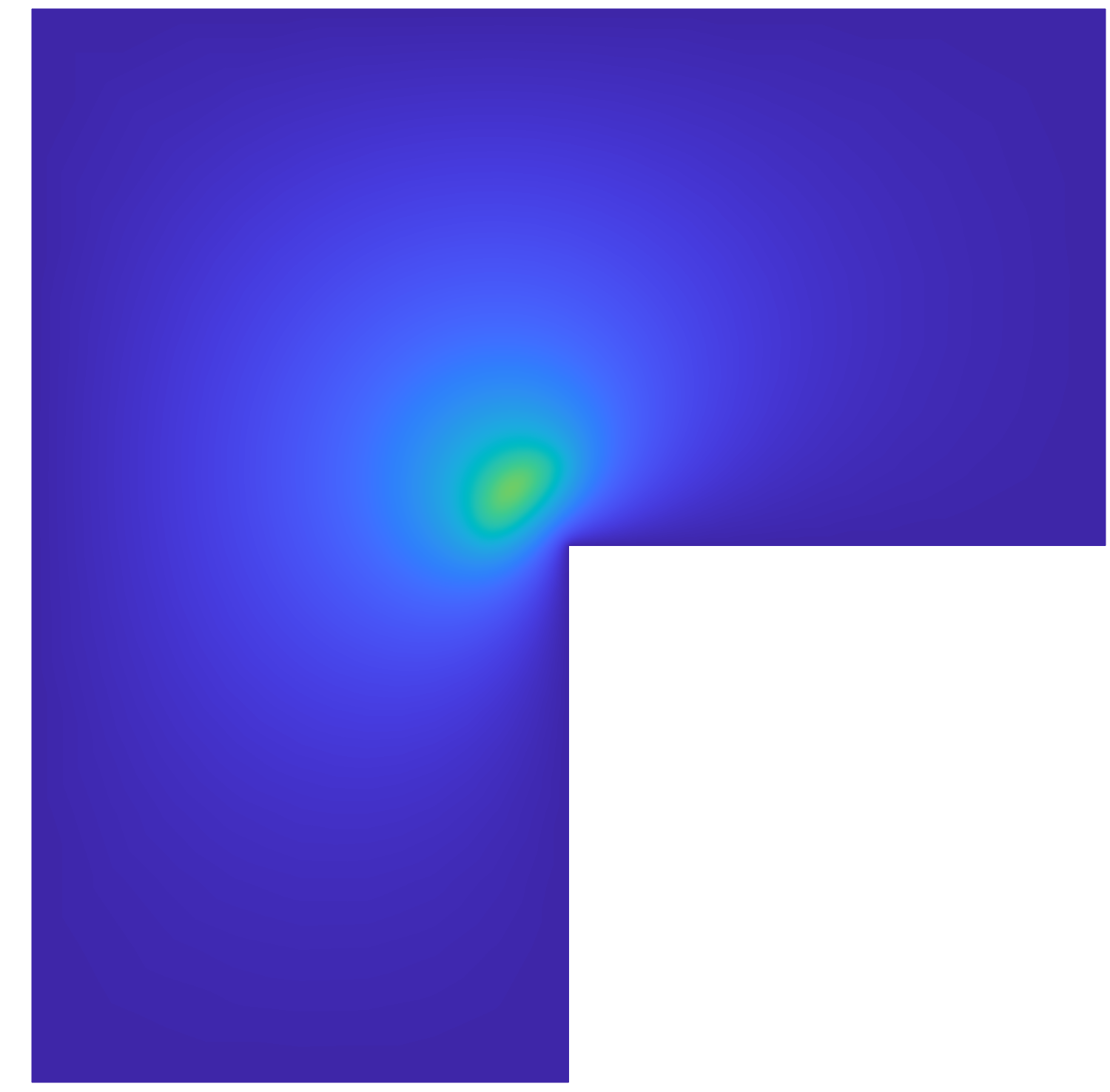}\\
\bottomrule
\end{tabular}
\caption{Results of \Cref{alg} for the Poisson control problem.
The last row represents the final iteration mesh, control, and state.  
The maximum DoFs limit of 10,000 is reached at iteration 5, resulting in 3 full
budget iterations.
}
\label{tbl:poissonControl}
\end{figure}
\Cref{tbl:poissonControl} depicts the control at iterations $k\in\{0,4,5,8\}$, 
where $k=8$ is the final iteration.
We observe that as \Cref{alg} progresses, the state and control are refined
around the re-entrant corner, accurately capturing the sharp local behavior
induced by the corner.

\subsection{Heat Conduction Topology Optimization}
\label{subscn:top}
We turn our attention to the heat-conduction topology optimization problem
in which we compute a distribution of material $z:\Omega \rightarrow [0,1]$
in $\Omega=(0,1)^2$ that minimizes energy while satisfying a volume constraint:
\begin{equation}
\label{eq:top_full}
\begin{aligned}
  &\min_{z \in \cZ}   \;     \int_\Omega [q \cdot S(z)](x) \textup{d} x
  &\textup{subject to}\quad  \int_\Omega z(x) \, \textup{d}x = v_0\vert \Omega \vert, \quad 0\le z \le 1\;\;\text{a.e.},
\end{aligned}
\end{equation}
where $S(z) = u\in \cU = H^1_{\Gamma_D}(\Omega)$ solves the diffusion equation
\begin{subequations}\label{eq:heat}
\begin{align}
  -\dv[K(\mathbb{F}z)\nabla u]   & = q \quad  \text{in} \,          \Omega,   \label{eq:F1} \\
   K(\mathbb{F}z)\nabla u\cdot n & = 0 \quad  \text{on} \, \partial \Omega_N, \label{eq:F2} \\
                               u & = 0 \quad  \text{on} \, \partial \Omega_D, \label{eq:F3}
\end{align}
\end{subequations}
where $K(\rho)$ is the cubic solid isotropic material with penalization (SIMP)
material model, i.e.,
\(
  K(\rho) = K_{\min} + (K_{\max} - K_{\min})\rho^3
\)
and $\rho = \mathbb{F}z\in H^1(\Omega)$ is the Helmholtz-filtered density,
which is computed by solving
\begin{subequations}\label{eq:filter}
\begin{align}
-r \Delta \rho + \rho & = z \quad \text{in}\, \Omega, \label{midfilter1} \\
\nabla \rho \cdot n & = 0 \quad \text{on} \, \partial \Omega. \label{midfilter2}
\end{align}
\end{subequations}
With respect to \eqref{eq:prob}, $J(S(\cdot),\cdot)$ is the objective function
in \eqref{eq:top_full} and $\phi$ is the indicator function of the constraints.
The sets $\partial\Omega_D$ and $\partial\Omega_N$ are disjoint subsets of the
boundary $\partial\Omega$ satisfying $\partial\Omega_D\neq\emptyset$, and
$\partial\Omega_D\cup\partial\Omega_N=\partial\Omega$.
We consider the two examples depicted in \Cref{fig:bcs}.
In the first (i.e.\ the left panel in \Cref{fig:bcs}), we have 
$$\partial\Omega_D = (\{0\}\times [0,1])\cup([0,1]\times\{1\})
\quad\text{and}\quad v_0=0.4,$$
while in the second (i.e.\ the right graphic of \Cref{fig:bcs}), we have
$$\partial \Omega_D = \{0\}\times[0.4,0.6]
\quad\text{and}\quad v_0=0.1.$$
For both examples, we set
$\bar{\kappa}_{\textup{val}}=\bar{\kappa}_{\textup{der}}=10^9$ and the maximum
number of degrees of freedom to be 150,000; this number was informed by
\cite{MR2765487}. The reasoning for such high values of
$\bar{\kappa}_{\textup{val}}$ and $\bar{\kappa}_{\textup{der}}$ is to deter
refinement at the beginning of the algorithm, thereby avoiding the slow convergence
associated with the linear discretization.
For both examples, we set $q\equiv 10^{-2}$, $K_{\max}=1$, $K_{\min}=10^{-3}$ and 
$r\equiv 10^{-2}/(2\sqrt{3})$.
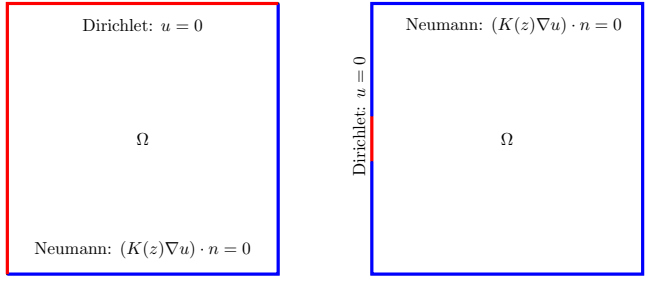
\begin{figure}
   \centering
    \begin{tikzpicture}[scale=0.6,  transform shape]
    \coordinate (A) at (0,0);
    \coordinate (B) at (6,0);
    \coordinate (C) at (6,6);
    \coordinate (D) at (0,6);
    \draw[black, thick] (A) -- (B) -- (C) -- (D) -- cycle;
    \node at (3,3) {$\Omega$};
    \draw[blue, very thick] (A) -- (B) node[pos=.5,above=2mm, black] {Neumann: $(K(z)\nabla u) \cdot {n}=0$} -- (C)  ;
    \draw[red, very thick] (C) -- (D) node[midway,below=2mm, black] {Dirichlet: $u=0$}
		-- (A) ;
    \end{tikzpicture}
    \qquad
    \begin{tikzpicture}[scale=0.6, transform shape]
    \coordinate (A) at (0,0);
    \coordinate (B) at (6,0);
    \coordinate (C) at (6,6);
    \coordinate (D) at (0,6);
    \coordinate (E) at (0,3.5);
    \coordinate (F) at (0,2.5);
    \draw[black, thick] (A) -- (B) -- (C) -- (D) -- (E) -- (F) -- cycle;
    \node at (3,3) {$\Omega$};
    \draw[red, very thick] (E) -- (F) node[midway, above=5mm, black, rotate=90] {Dirichlet: $u=0$};
    \draw[blue, very thick] (F) -- (A) -- (B) -- (C) -- (D) -- (E)  node[pos=.2, right=6mm, black] {Neumann: $(K(z)\nabla u) \cdot {n}=0$};
    \end{tikzpicture}
    \caption{Illustration of the boundary conditions for the heat conduction
        topology optimization problems. The left figure shows Neumann
        conditions on the top and left sides, and Dirichlet conditions on the
	bottom and right sides.	The right figure depicts Neumann conditions
	everywhere except an interval around $x=(0,0.5)$.}
    \label{fig:bcs}
\end{figure}

For a density $z\in\cZ$, let $\rho=\mathbb{F}z$ be the filtered density, which
solves \eqref{eq:filter}.
We discretize \eqref{eq:filter} using linear finite elements and denote
the finite-element approximation of $\rho$ by $\rho^h$.  As we will see, we
must bound the $L^\infty$-error in the density variables, which we do using
\eqref{reliability} with $A=r\mathbb{I}$, $b\equiv 1$, $q=z$ and
$\Gamma_D=\emptyset$.
To maintain coercivity and boundedness of \eqref{eq:filter},
we use a lumped mass matrix for the second term on the left-hand side to ensure
that the finite-element solution of the filter equation \eqref{eq:filter} does
not violate the discrete maximum principle
\cite{HAntil_DPKouri_DRidzal_DBRobinson_MSalloum_2023a}. 

In contrast to the filter equation, we discretize the state equation
\eqref{eq:heat} using quadratic finite elements.  We denote
the continuous state by $u=S(z)$, the continuous state associated with
the discrete filtered density $\rho^h$ by $\bar{u}^h$ and the fully
discrete state by $u^h$.  To bound the error $u-u^h$, we individually
bound the errors $u-\bar{u}^h$ and $\bar{u}^h-u^h$.  Since $u^h$ is the
finite-element approximation of $\bar{u}^h$, we can employ similar error
estimates as in \Cref{sec:poisson}, where we note here that the adjoint
and state variables associated with \eqref{eq:heat} are equal up to
scaling by $-1$.  Consequently, we do not solve the adjoint
equation or account for the adjoint error in mesh refinement.

To bound the error $w\coloneqq u - \bar{u}^h$, we first leverage
\eqref{eq:heat}, which yields
\begin{align} \label{subequations}
  \int_\Omega (K(\rho)\nabla w)\cdot\nabla v\,\textup{d}x = \int_\Omega ((K(\rho)-K(\rho^h))\nabla \bar{u}^h\cdot\nabla v\,\textup{d}x \quad\forall\,v\in\cU.
\end{align}
Setting $v=w$ in \eqref{subequations} and using $K(\cdot) \ge K_{\min}$ yields
\begin{align} \label{estimate}
  K_{\min} \|\nabla w\|_{L^2(\Omega)}^2 \le 
  \int_\Omega K(\rho)|\nabla w|^2\,\textup{d}x = -\int_\Omega ((K(\rho)-K(\rho^h))\nabla \bar{u}^h)\cdot\nabla w\,\textup{d}x \\
  \le \|K(\rho) - K(\rho^h)\|_{L^\infty(\Omega)} \,
\|\nabla \bar{u}^h\|_{L^2(\Omega)} \,
\|\nabla w\|_{L^2(\Omega)}
\end{align}
where the final inequality follows from H\"older's inequality.
Hence
\begin{align*}
K_{\min}\|\nabla (u - \bar{u}^h)\|_{L^2(\Omega)}
& \le (K_{\max}-K_{\min}) \| \rho^3 - (\rho^h)^3 \|_{L^\infty(\Omega)} \,
\|\nabla \bar{u}^h\|_{L^2(\Omega)}  \nonumber \\
& \hspace*{-1.4cm}  = (K_{\max}-K_{\min}) \| \rho - \rho^h\|_{L^\infty(\Omega)} \| \rho^2 + \rho\rho^h + (\rho^h)^2\|_{L^\infty(\Omega)} \,
\|\nabla \bar{u}^h\|_{L^2(\Omega)}.
\end{align*}
Finally, since $\bar{u}^h\in\cU$, we bound the error $w$ using
\eqref{reliability}.  Combining this with \eqref{reliability-U} (with
$A=K(\rho^h)\mathbb{I}$, $b\equiv 0$ and $\Gamma_N=\partial\Omega_N$) to bound
the error between $\bar{u}^h$ and $u^h$ yields the error bound 
\[
  \|u-u^h\|_{\cU} \leq C |\log h|^2 \sum_{i=1}^3 \xi_{\infty,i}^h
  + \left(\sum_{i=1}^3 (\xi_{\cU,i}^h)^2\right)^{\frac{1}{2}}.
\]
where $C>0$ is a positive constant independent of $h$.

Owing to the symmetry of the domains and boundary conditions in \Cref{fig:bcs},
the solutions to \eqref{eq:top_full} are also symmetric.  To enforce this
symmetry, we discretize and optimize over half the domain.  For the first
example, the half domain is $\{(x,y)\in\Omega\,\vert\,x\le 1-y\}$ and the
second is $(0,1)\times(0,0.5)$. The initial mesh for both examples on the full
domain was generated by bisecting the elements of a uniform $64 \times 64$ mesh
of quadrilateral elements.  After restricting to the half domains, this
resulted in 4225 DoFs for the first example and 4193 DoFs for the
second; note the discrepancy stems from the Neumann conditions. 

We depict the results in \Cref{tbl:top2} and \Cref{tbl:top1}.
In \Cref{tbl:top2}, we observe refinement around the sharp features of the
density, thereby resulting in a crisper image. The refinement starts on
iteration~48, and \Cref{alg} achieves the stopping tolerance at iteration 90.
Note that our optimization routine does roughly 10 iterations with 150,000
DoF. Likewise, \Cref{tbl:top1} depicts refinement again around
the sharp features of the density. Additionally, we hit the DoF
limit near the end of the run, again performing about 10 iterations
with the full DoF budget. Importantly, the refinement during the final few iterations
of \Cref{alg} addresses the fine-scaled features near the edges of the density,
cf.\ the last rows of \Cref{tbl:top2,tbl:top1}.
\begin{figure}[!ht]
\centering
\begin{tabular}{@{}c|c|@{}c@{}c@{}c}
\toprule
$N_{\rm dof}$  & $ k$   & Grid &  $z$ & $u^h$ \\
\hline
$4225$          & 26 & \rowincludegraphics[width=.25\linewidth]{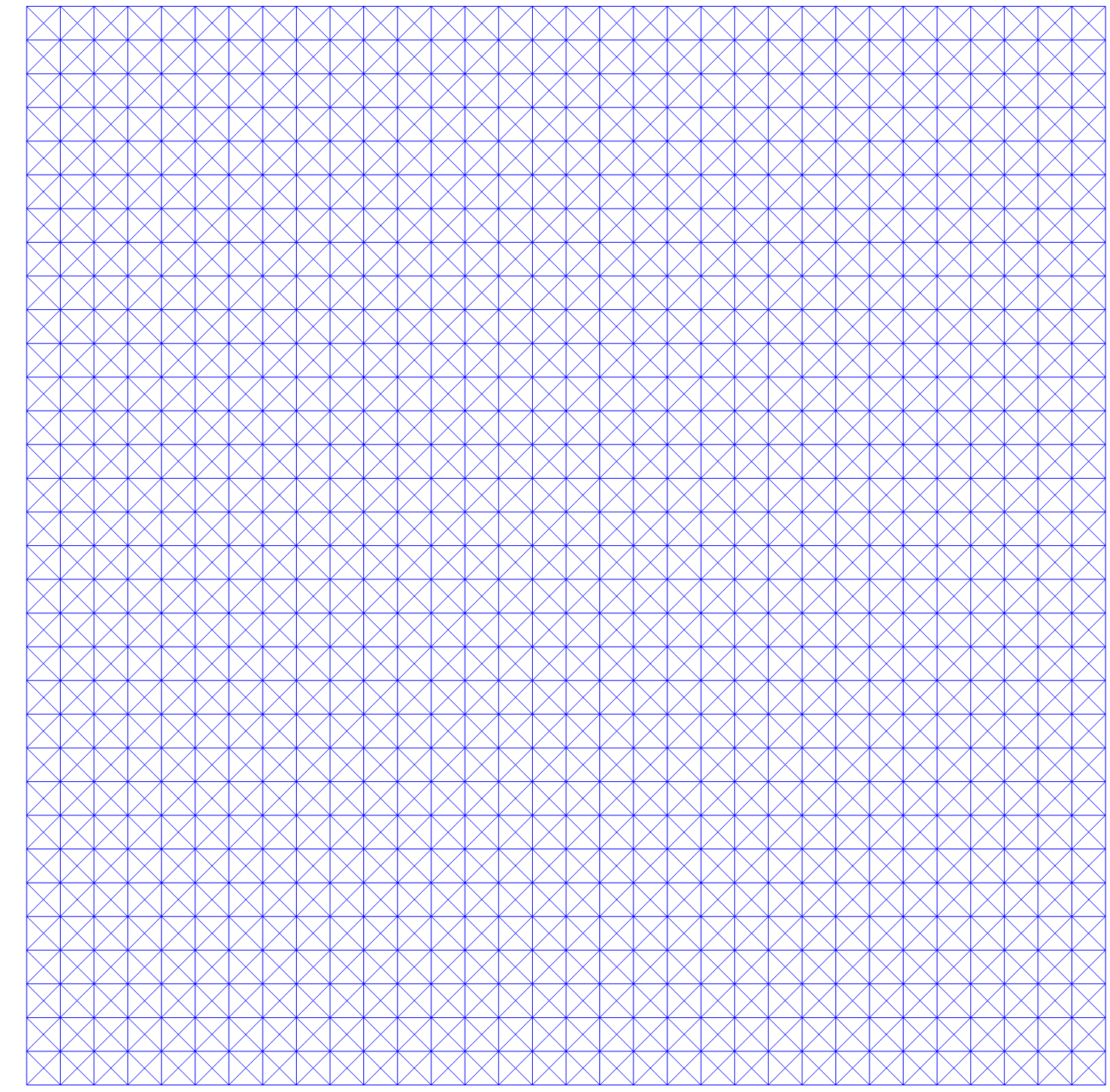}
                     & \rowincludegraphics[width=.25\linewidth]{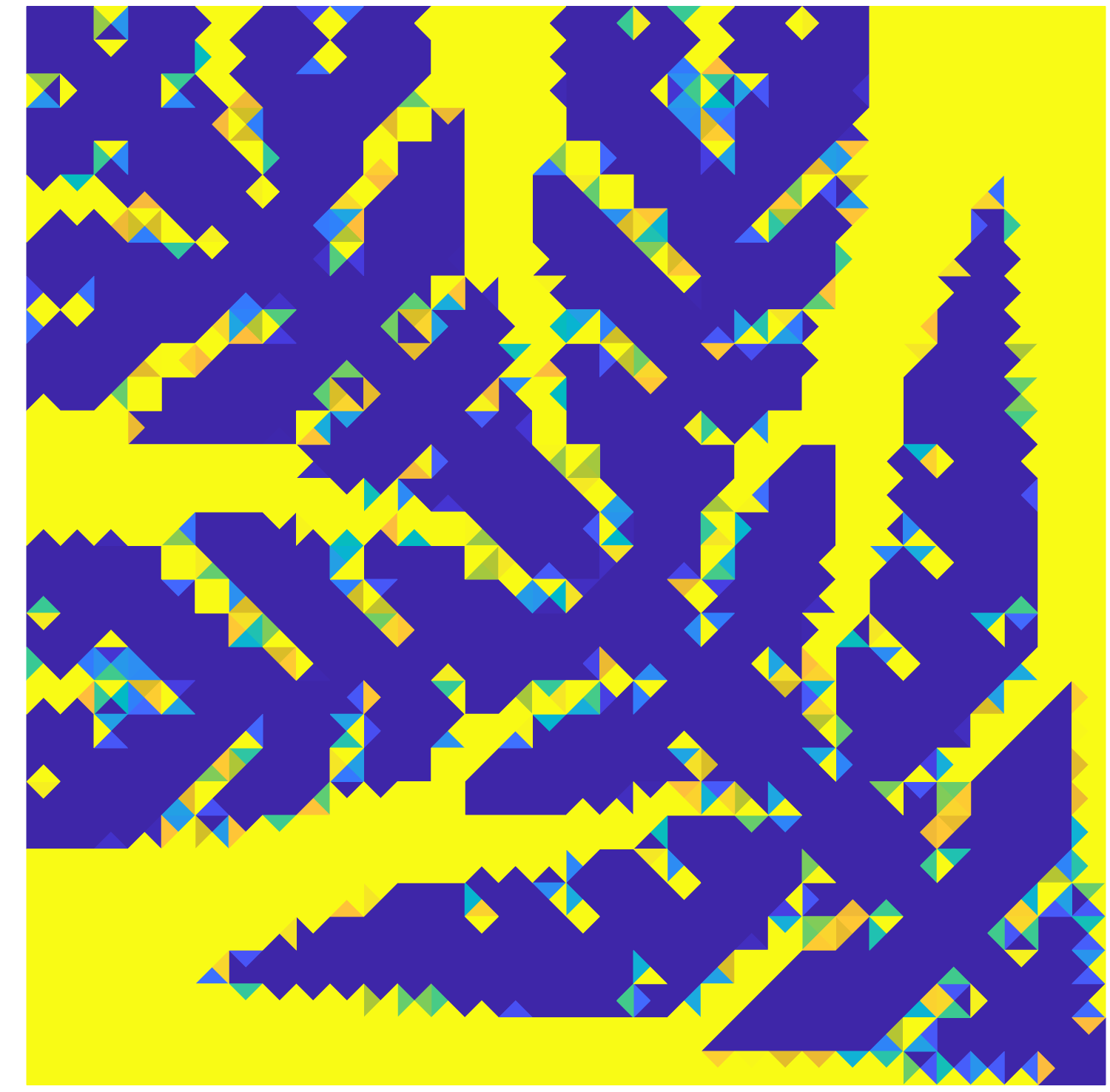}
                     & \rowincludegraphics[width=.25\linewidth]{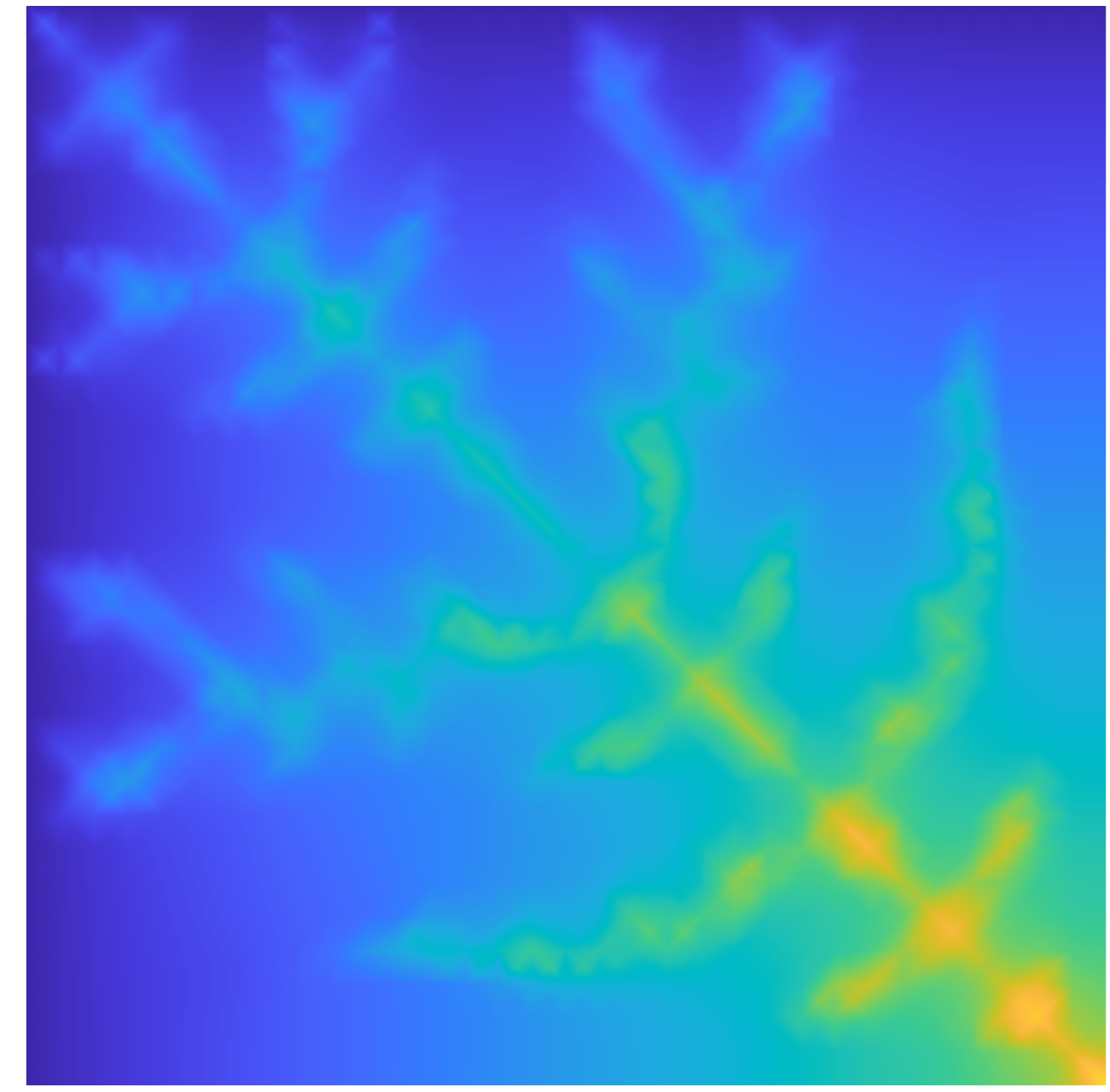}\\
$5\cdot 10^4$   & 48 & \rowincludegraphics[width=.25\linewidth]{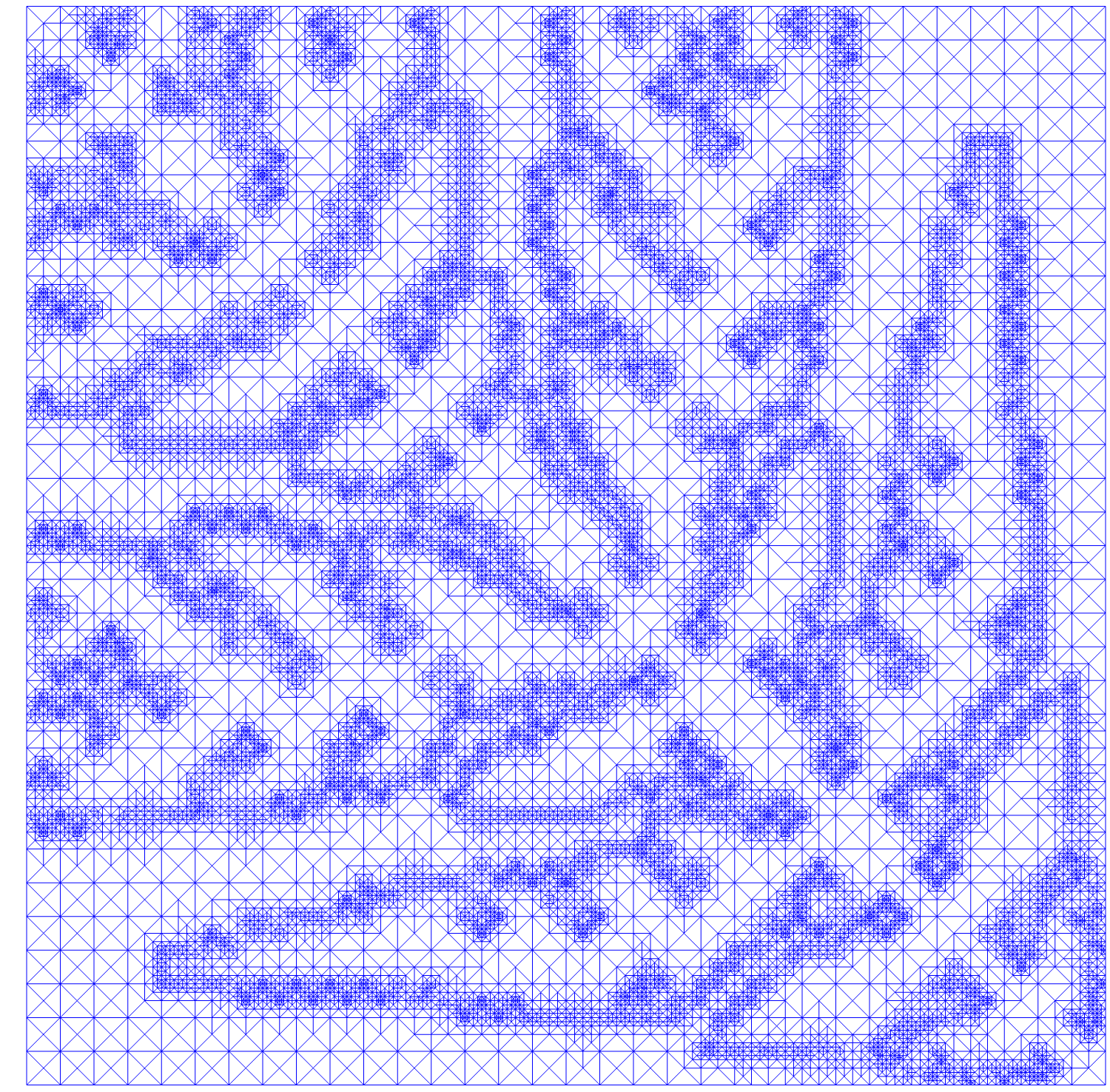}
                     & \rowincludegraphics[width=.25\linewidth]{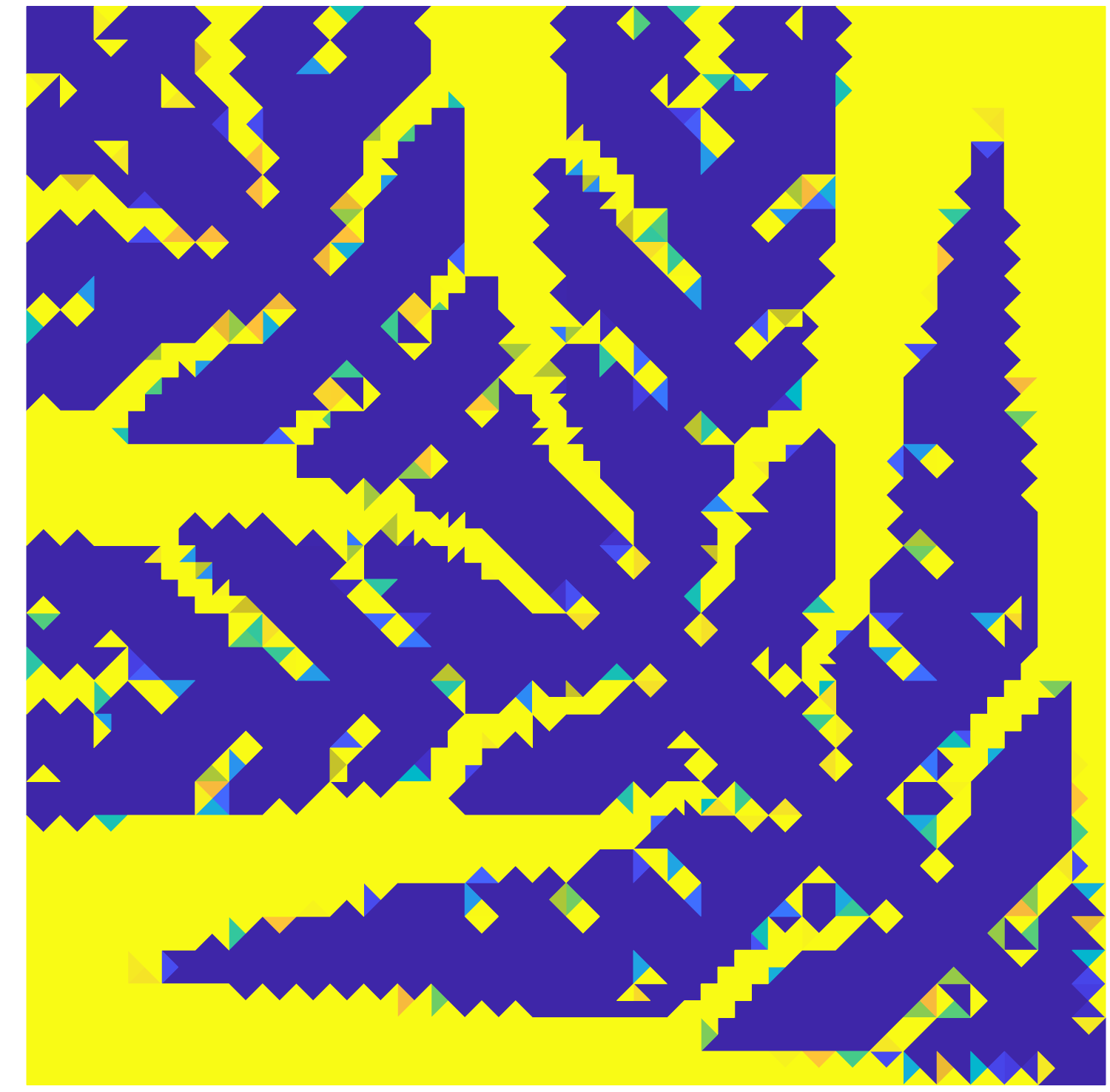}
                     & \rowincludegraphics[width=.25\linewidth]{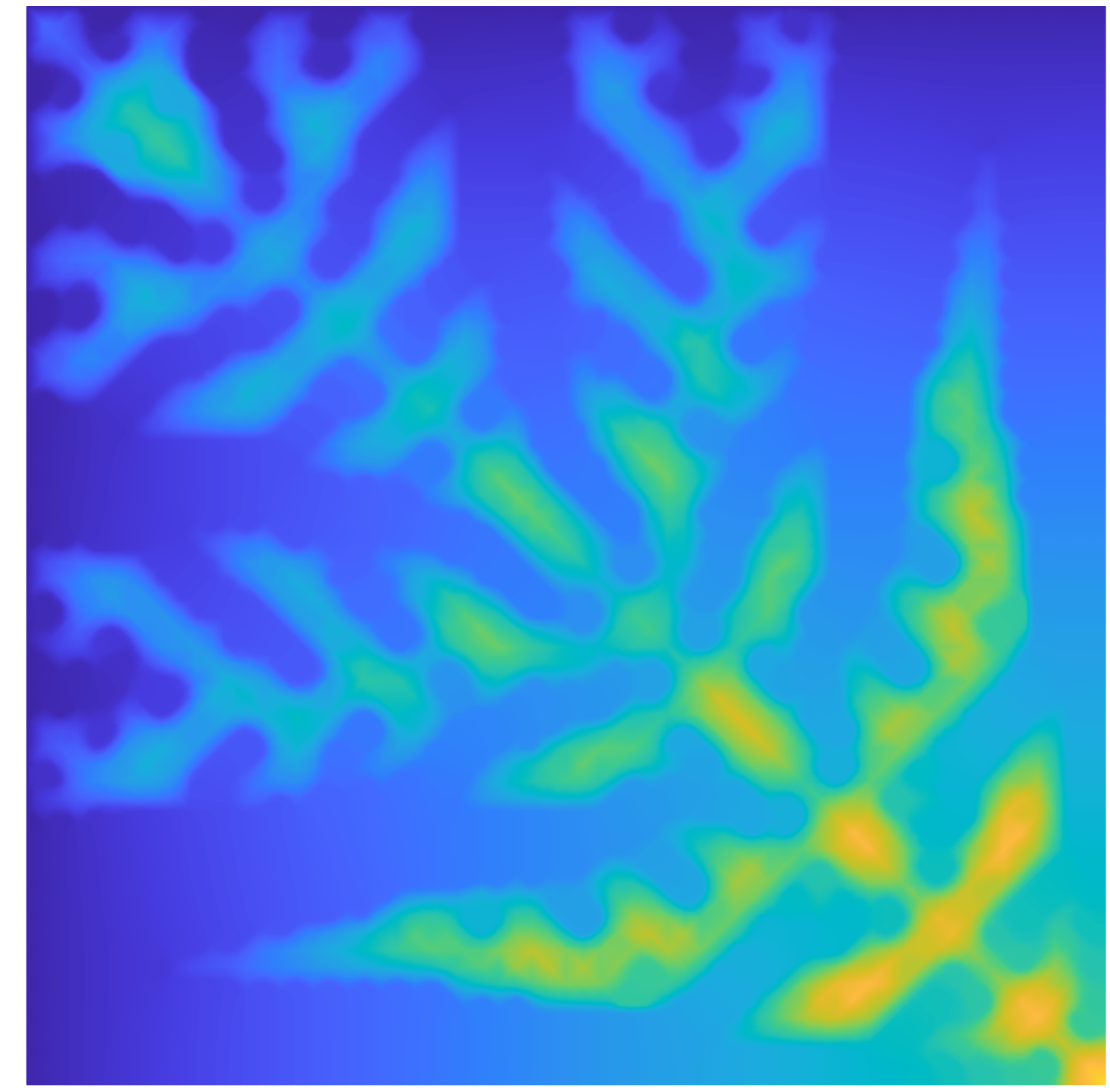}\\
$1.1\cdot 10^5$ & 55 & \rowincludegraphics[width=.25\linewidth]{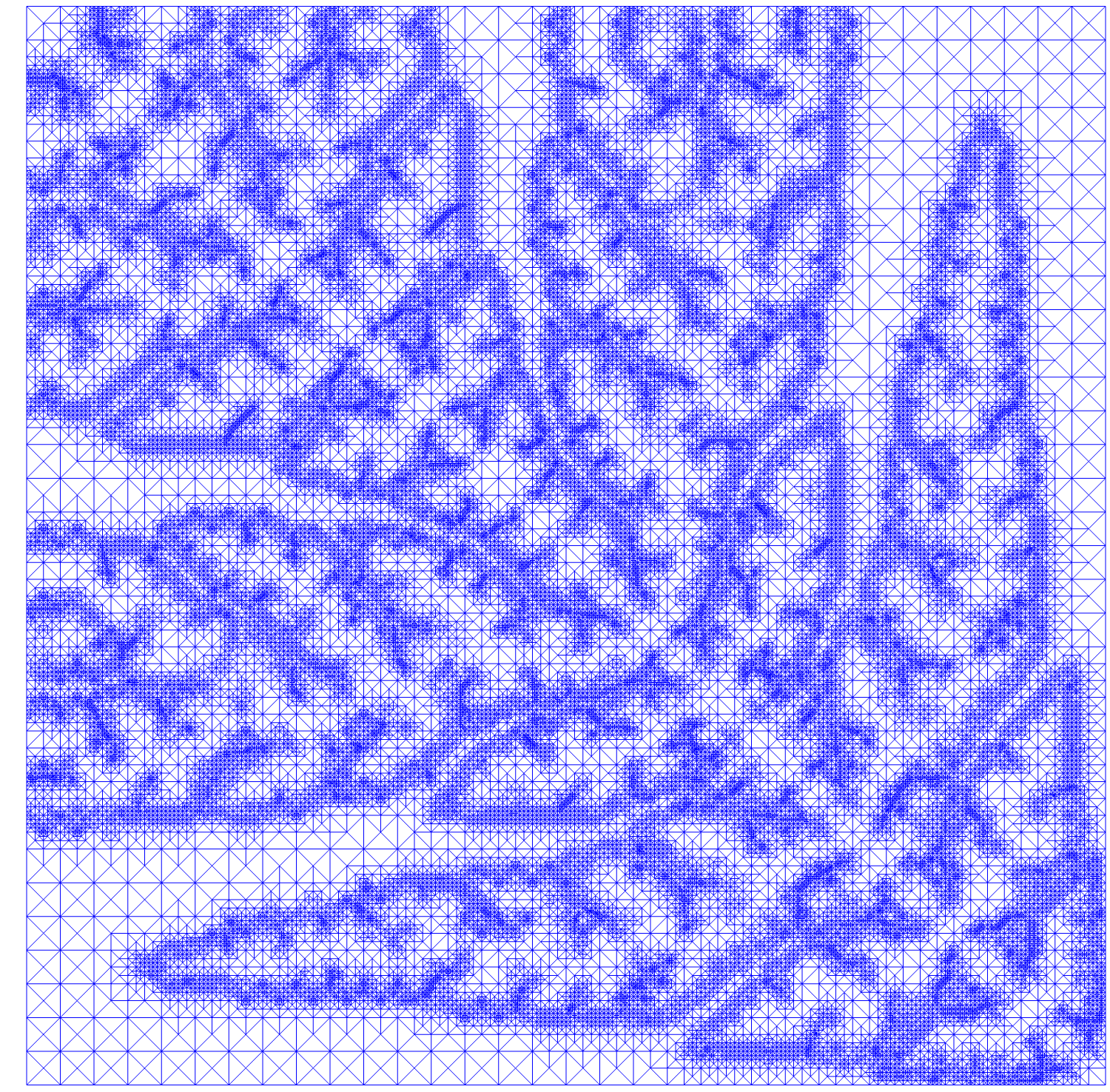}
                     & \rowincludegraphics[width=.25\linewidth]{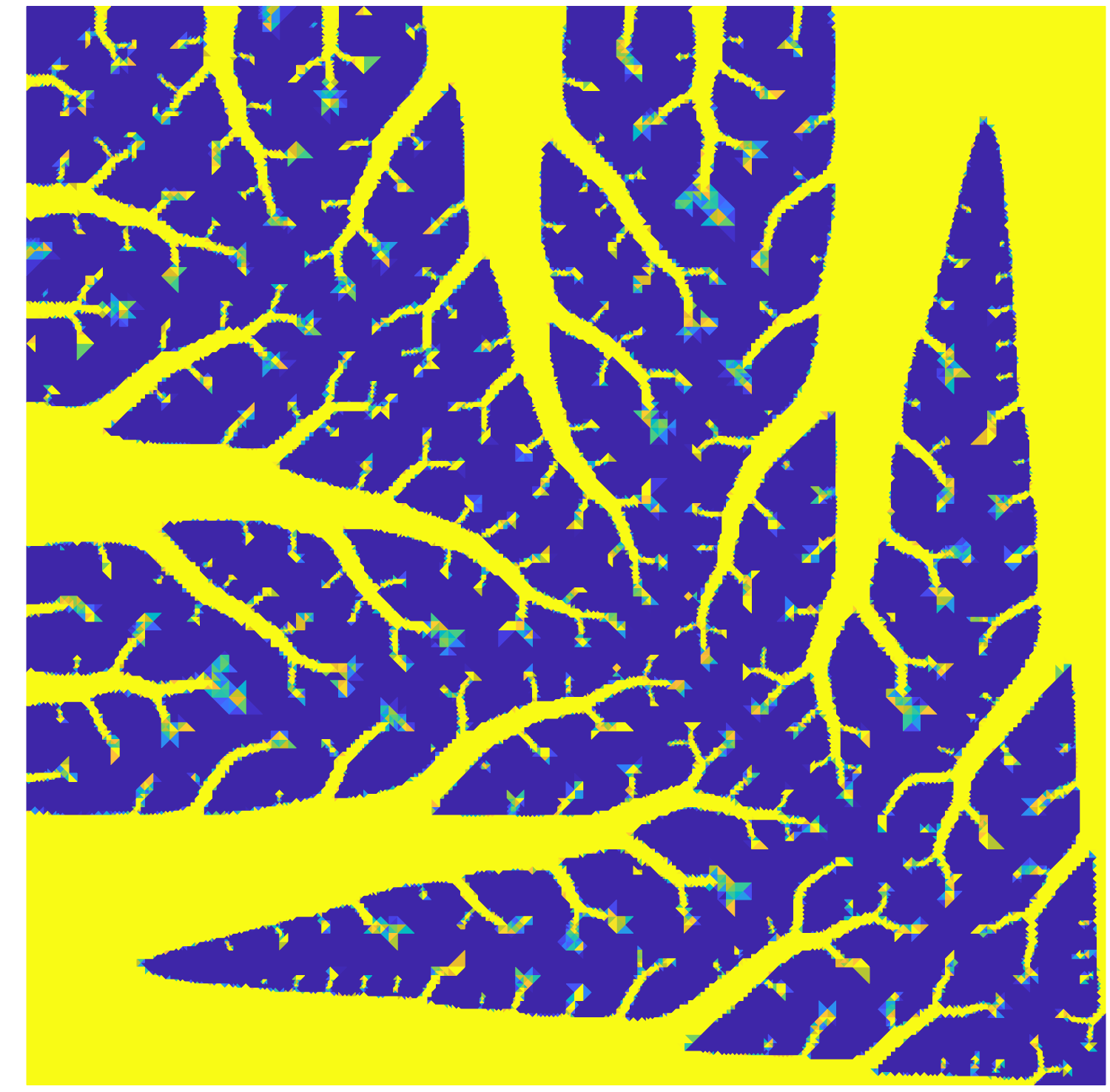}
                     & \rowincludegraphics[width=.25\linewidth]{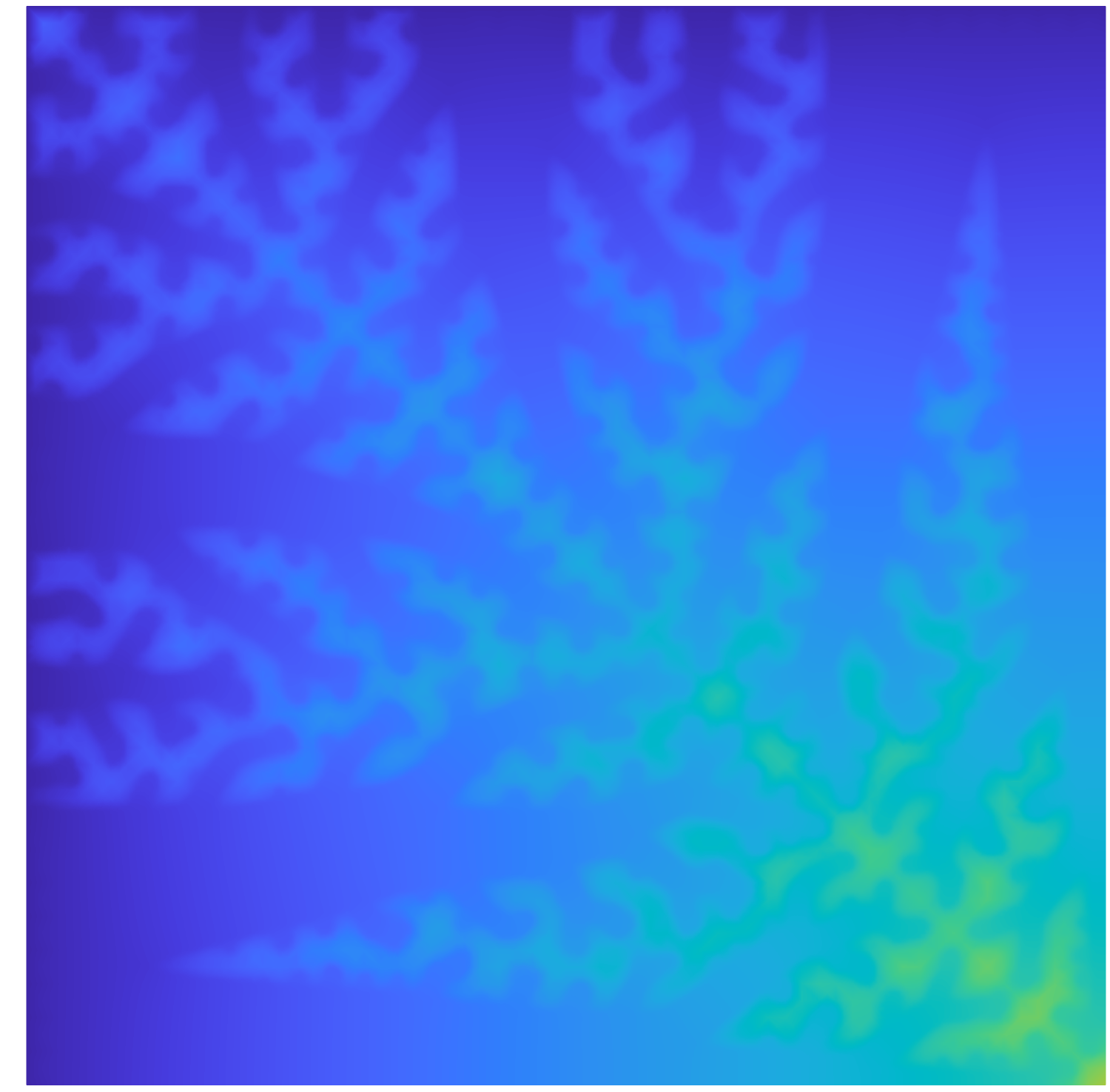}\\
$1.5\cdot10^5$  & 90 & \rowincludegraphics[width=.25\linewidth]{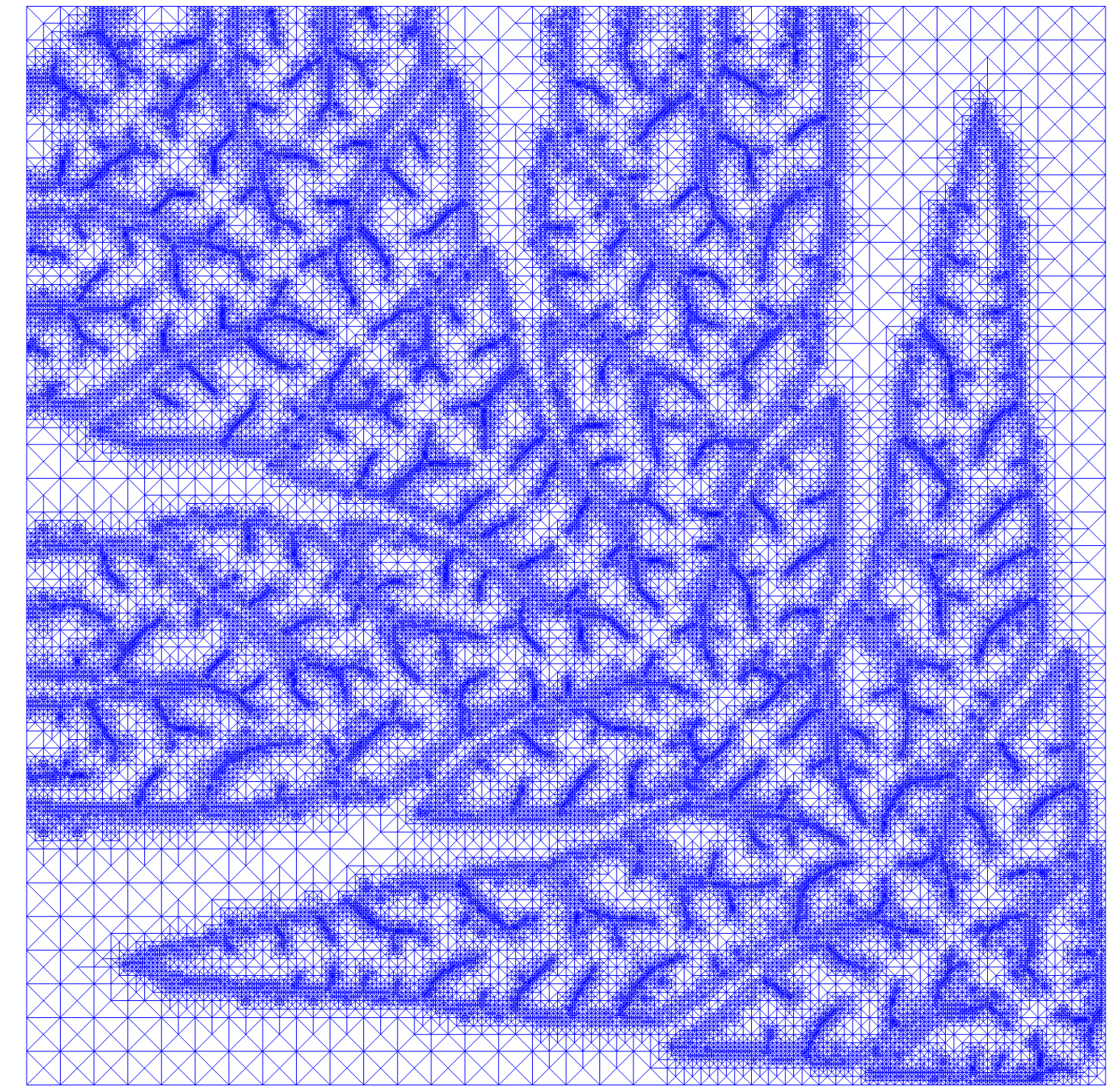}
                     & \rowincludegraphics[width=.25\linewidth]{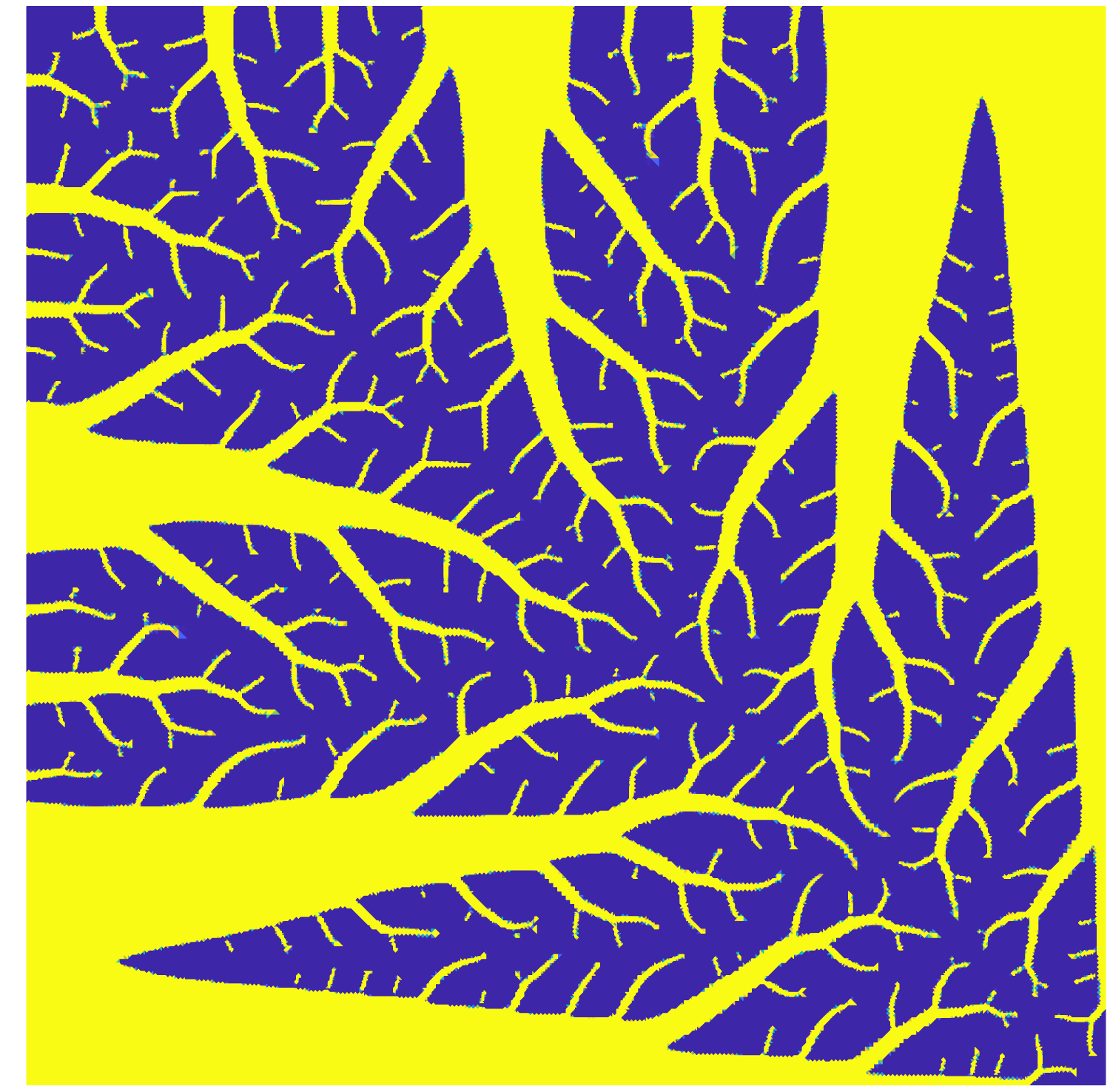}
                     & \rowincludegraphics[width=.25\linewidth]{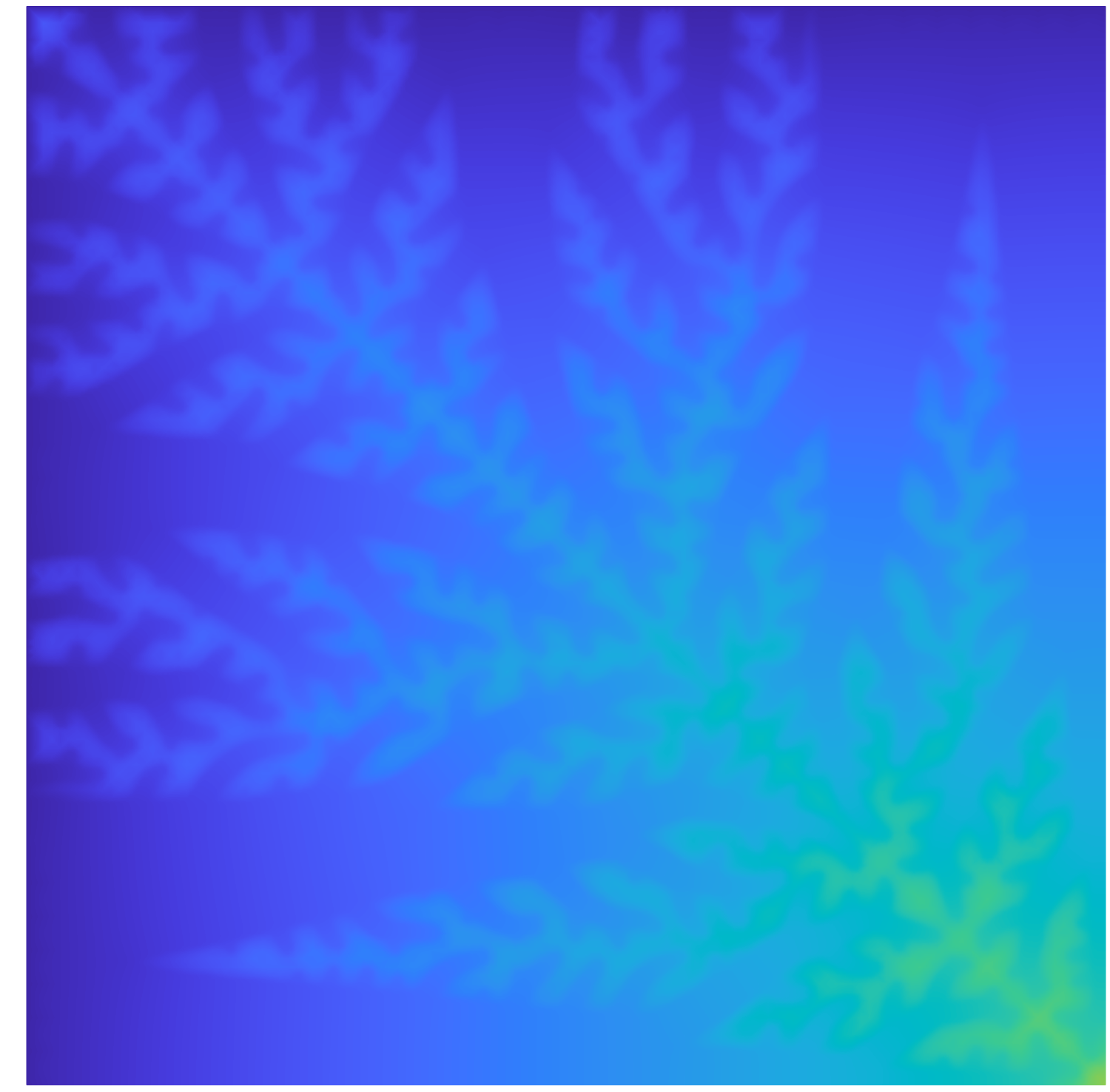}\\
\bottomrule
\end{tabular}
\caption{Results of \Cref{alg} for the first heat conduction topology
optimization example. Our algorithm terminated after 90 iterations. Note
$N_{\rm dof}$s are approximate.}
\label{tbl:top2}
\end{figure}

\begin{figure}[!ht]
\centering
\begin{tabular}{@{}c|@{}c@{}|@{}c@{}c@{}c}
\toprule
$N_{\rm dof}$     & $k$   & Grid & $z$ & $u^h$ \\
\hline
$4193$          & 21  & \rowincludegraphics[width=.25\linewidth]{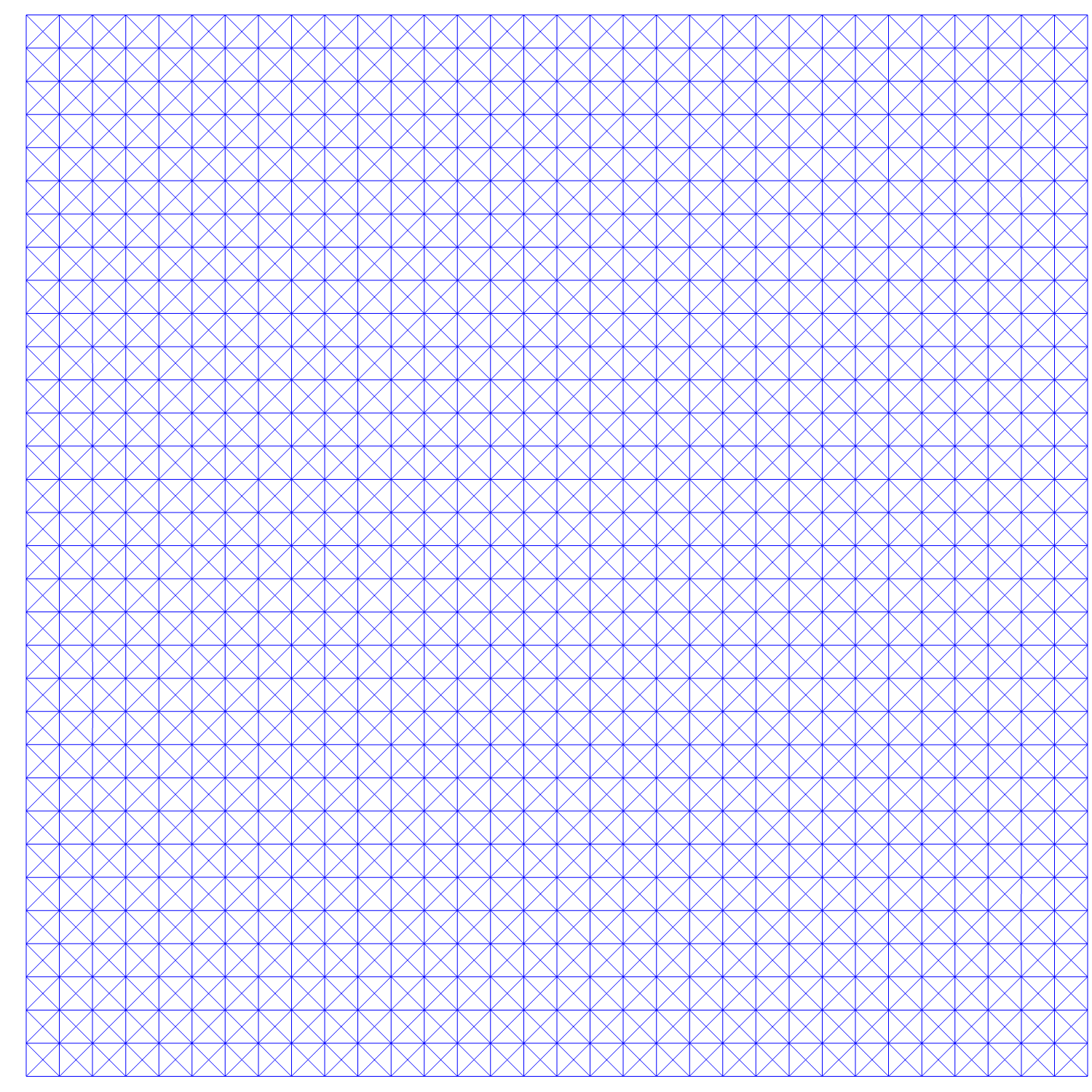}
                      & \rowincludegraphics[width=.25\linewidth]{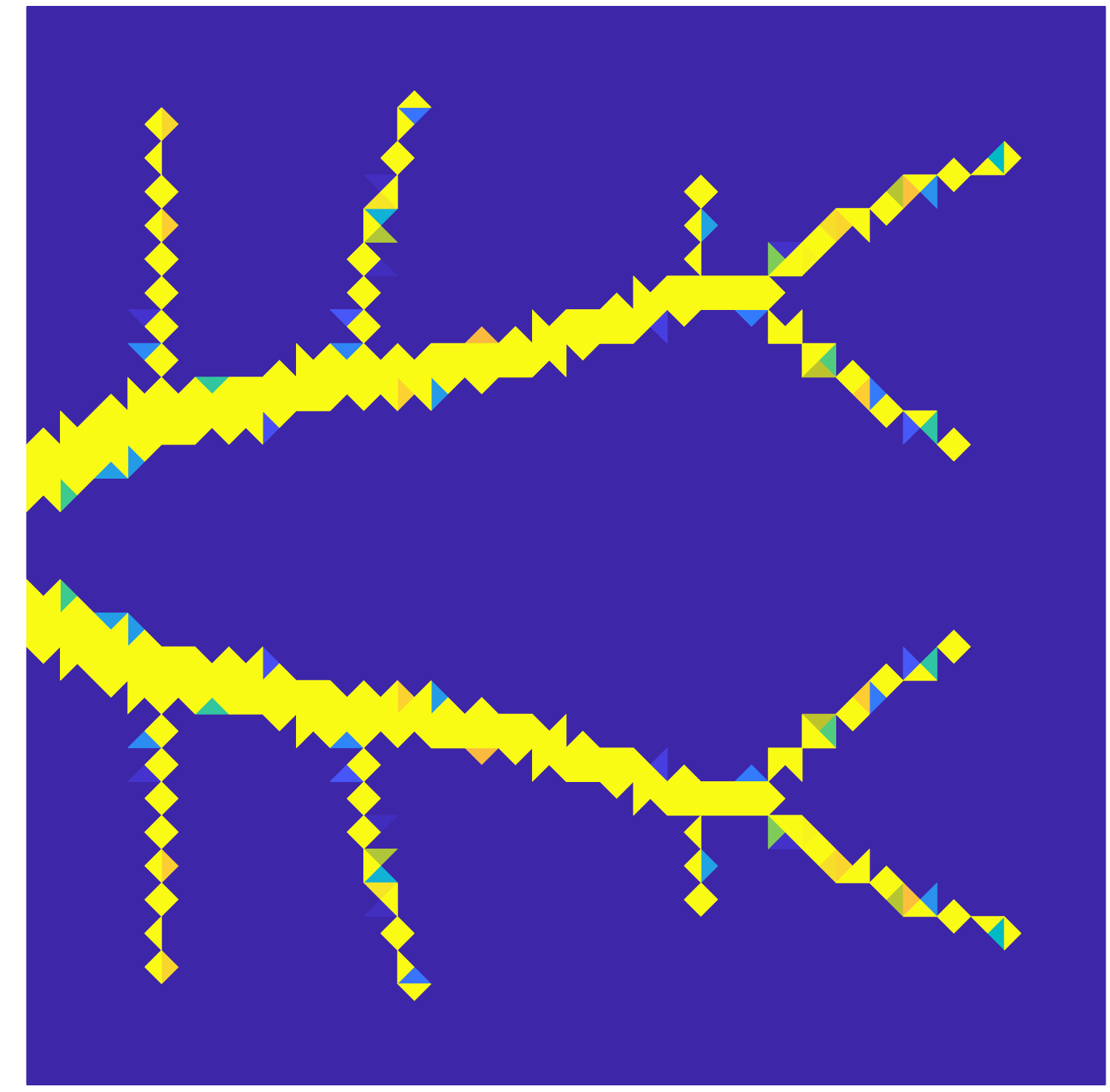}
                      & \rowincludegraphics[width=.25\linewidth]{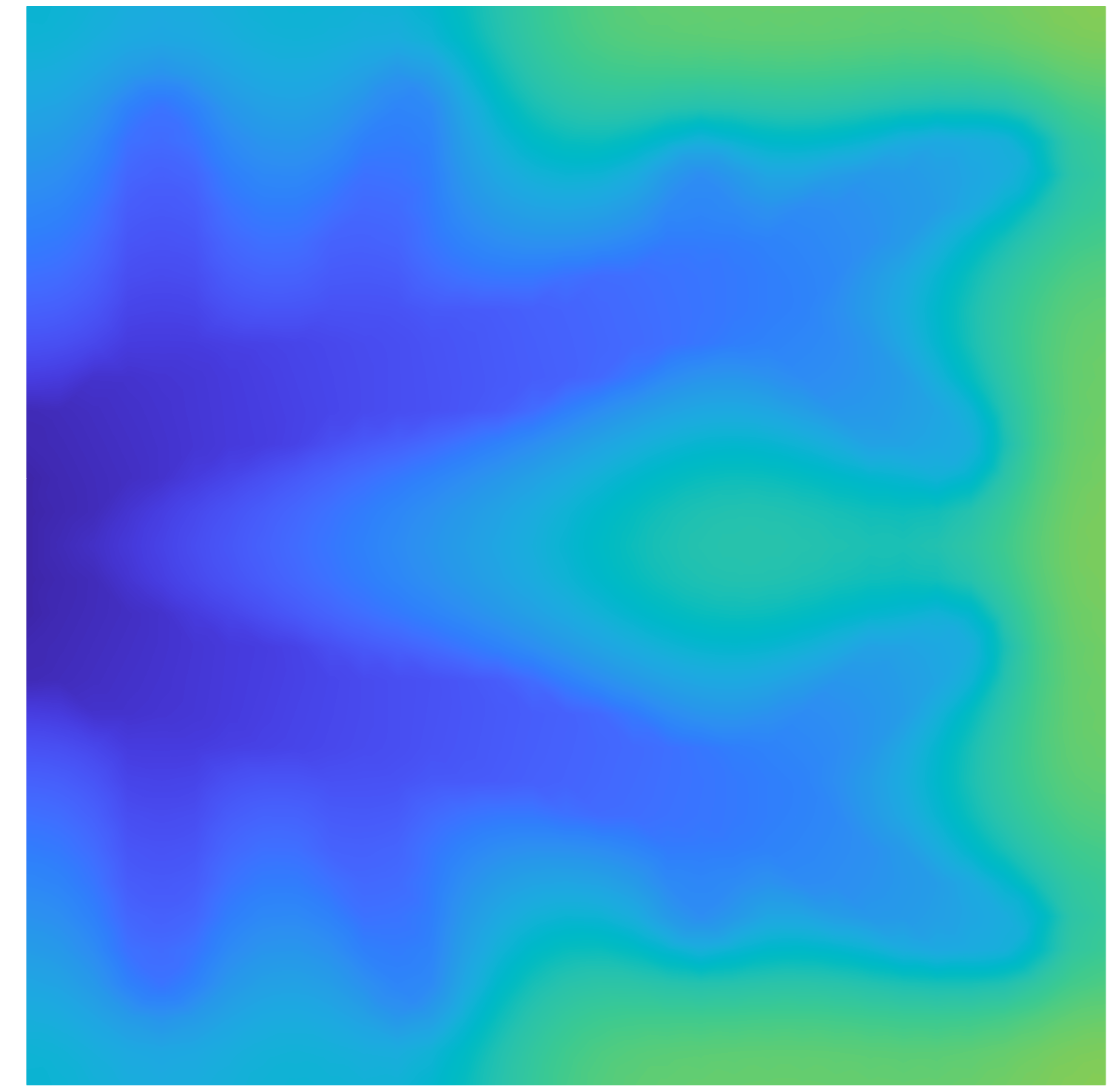}\\
$2\cdot 10^4$   & 52  & \rowincludegraphics[width=.25\linewidth]{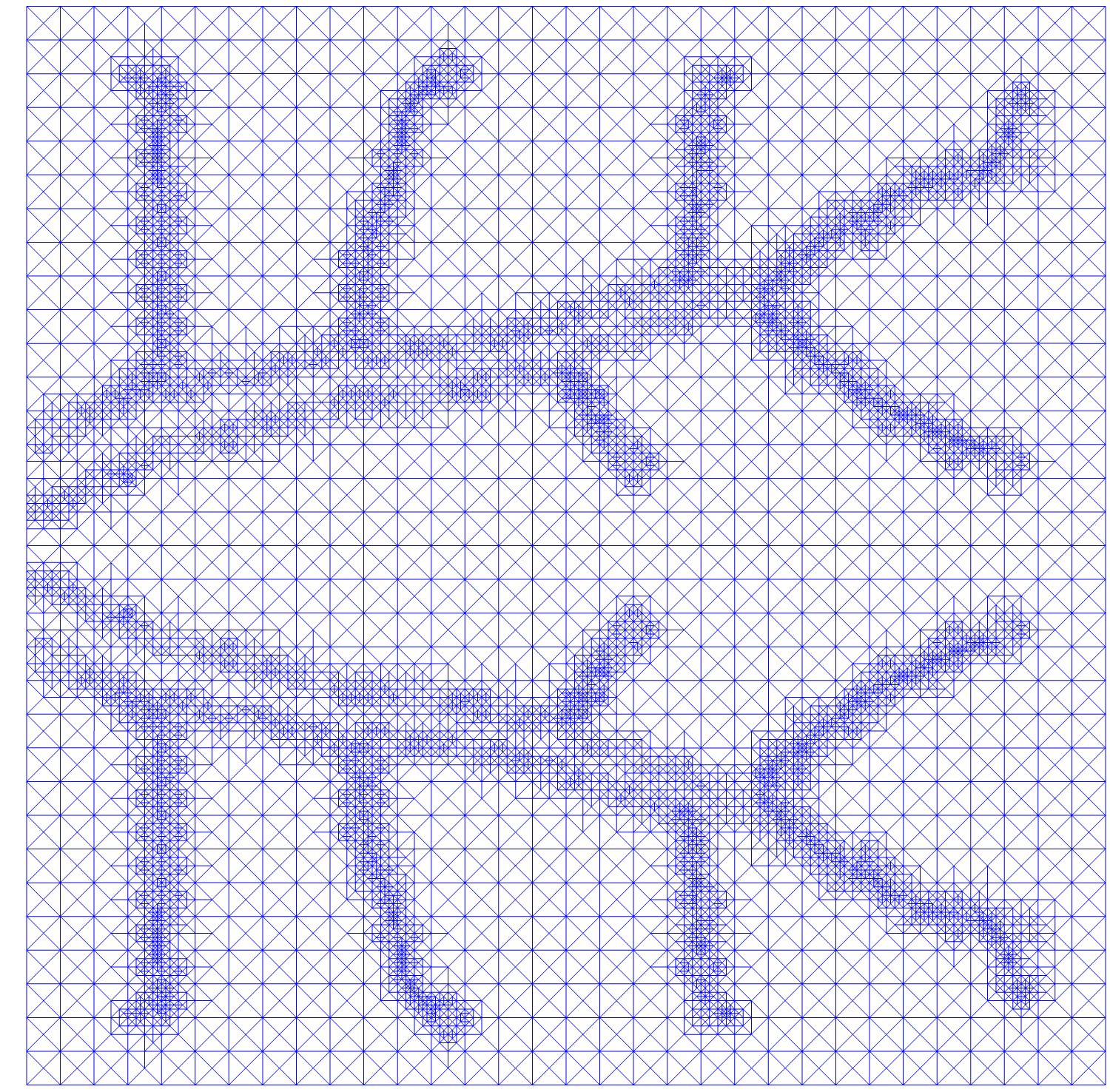}
                      & \rowincludegraphics[width=.25\linewidth]{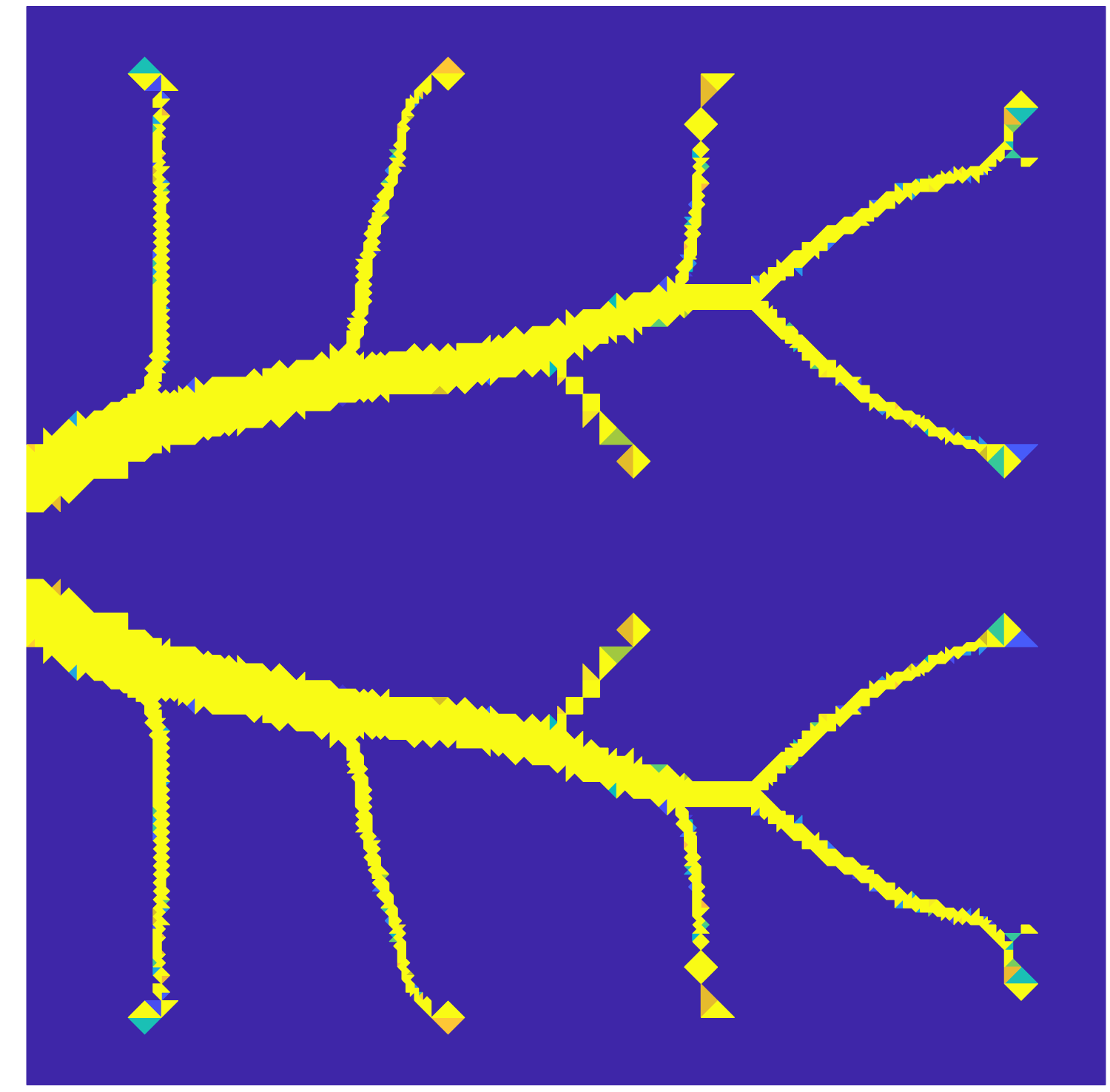}
                      & \rowincludegraphics[width=.25\linewidth]{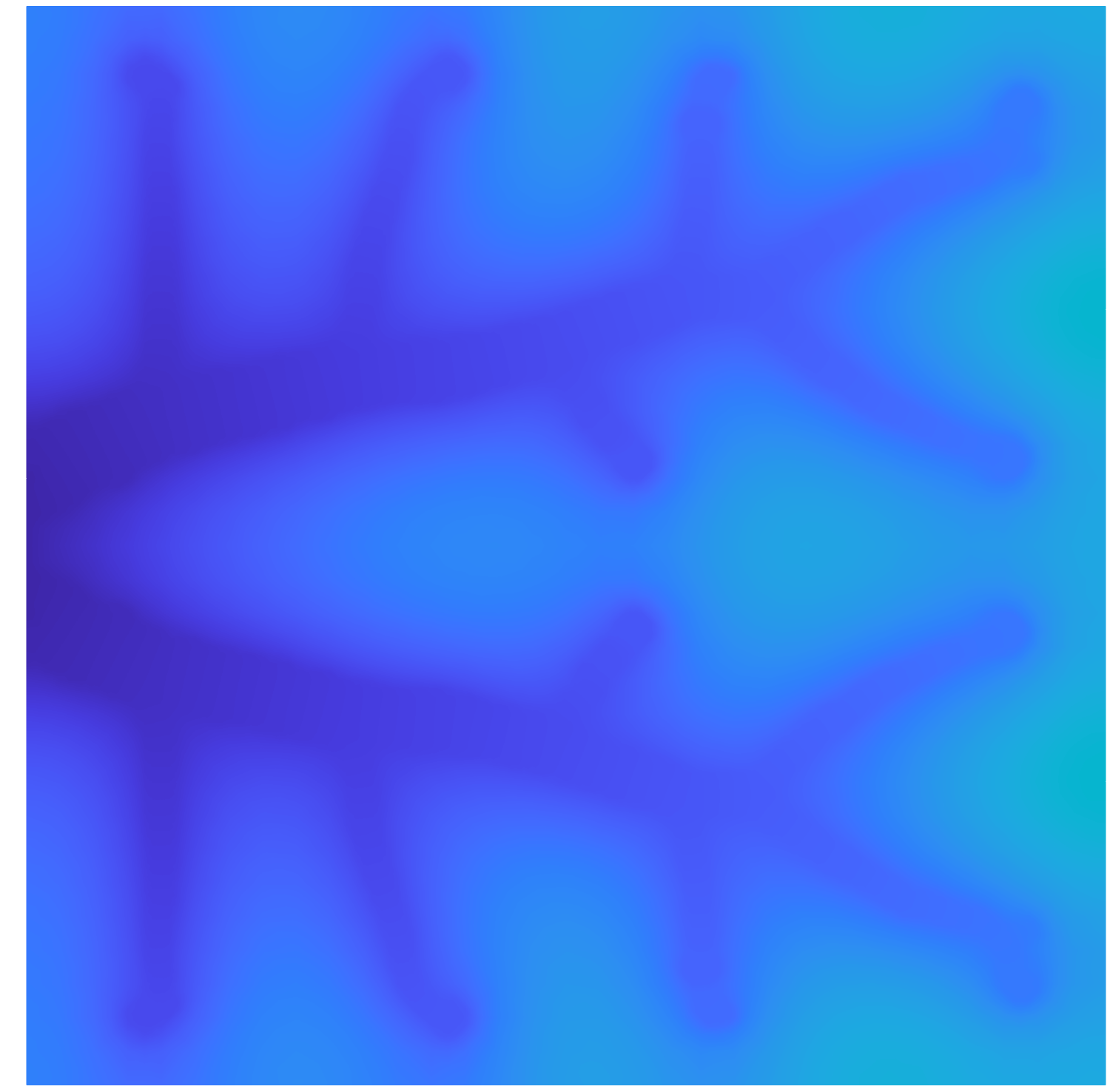}\\
$5\cdot 10^4$   & 76  & \rowincludegraphics[width=.25\linewidth]{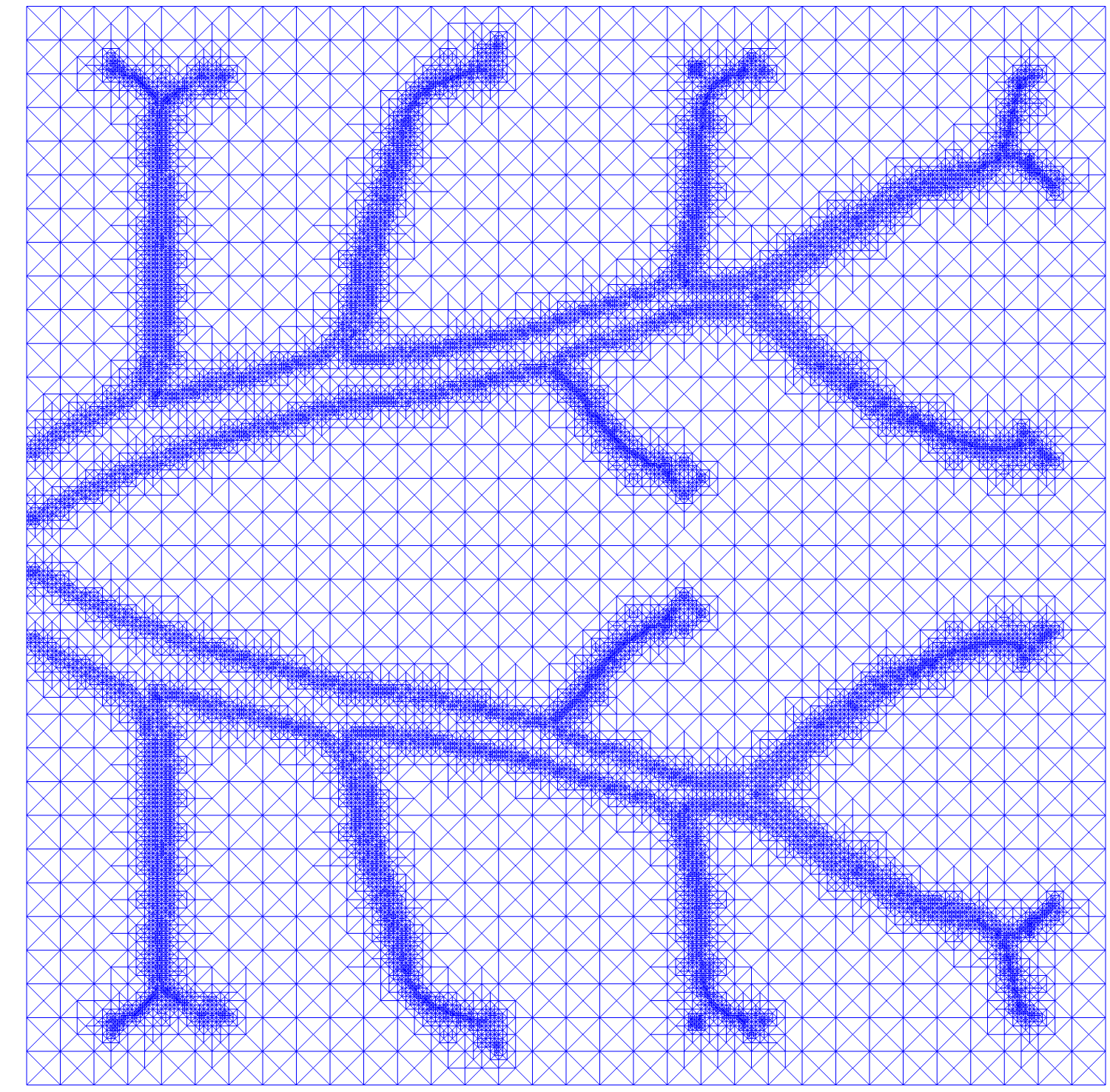}
                      & \rowincludegraphics[width=.25\linewidth]{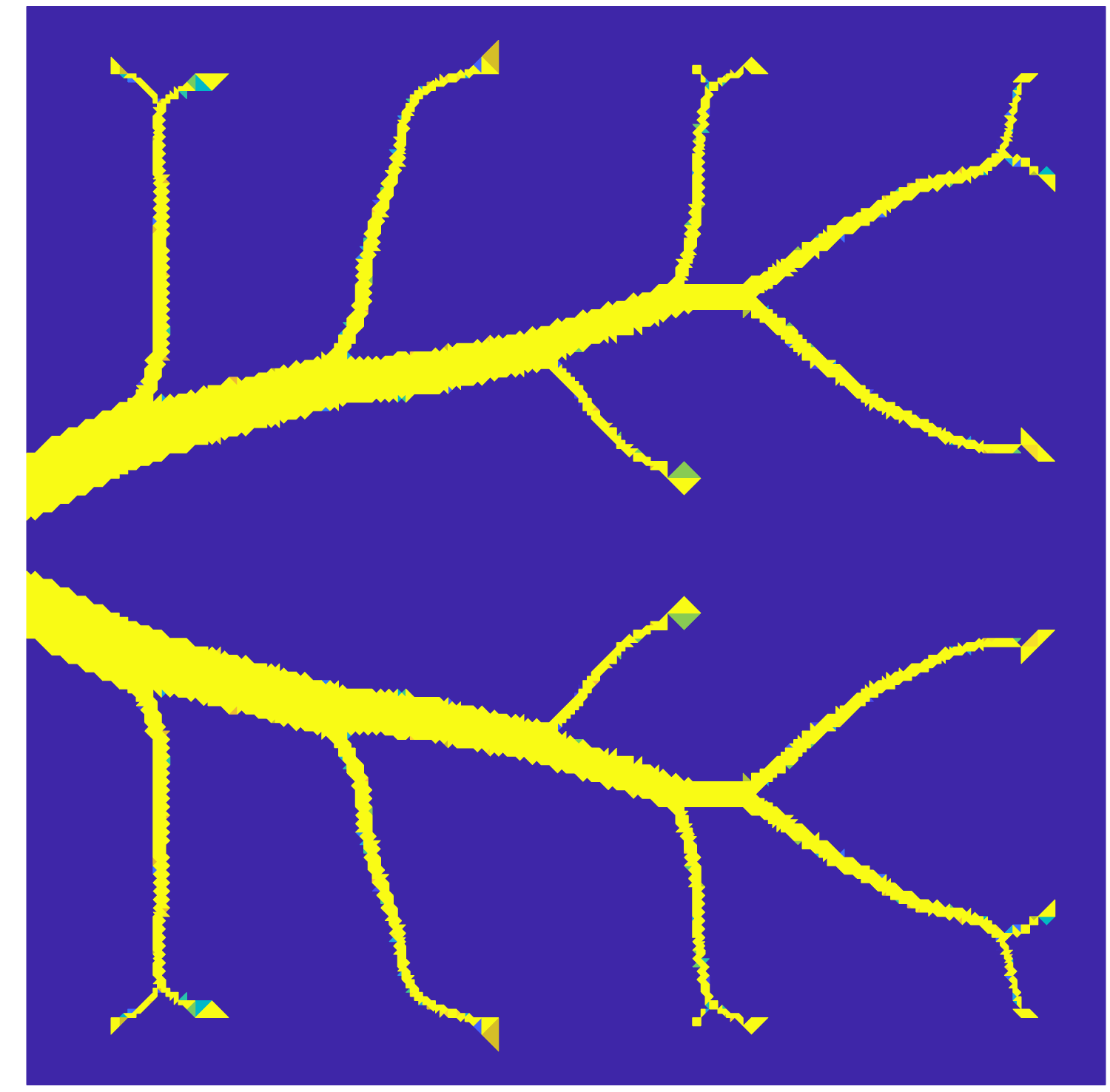}
                      & \rowincludegraphics[width=.25\linewidth]{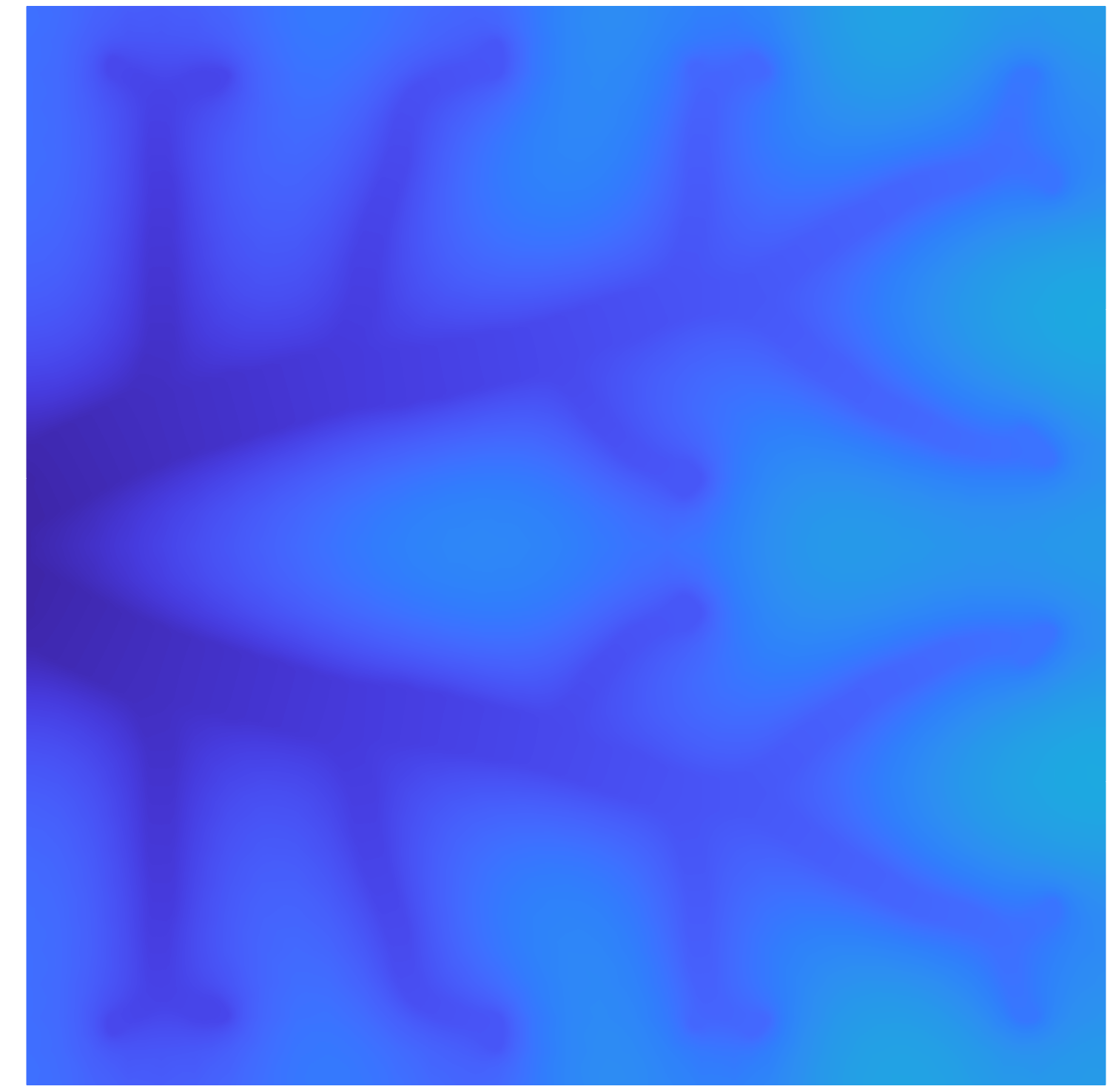}\\
$1.5\cdot 10^5$ & 147 & \rowincludegraphics[width=.25\linewidth]{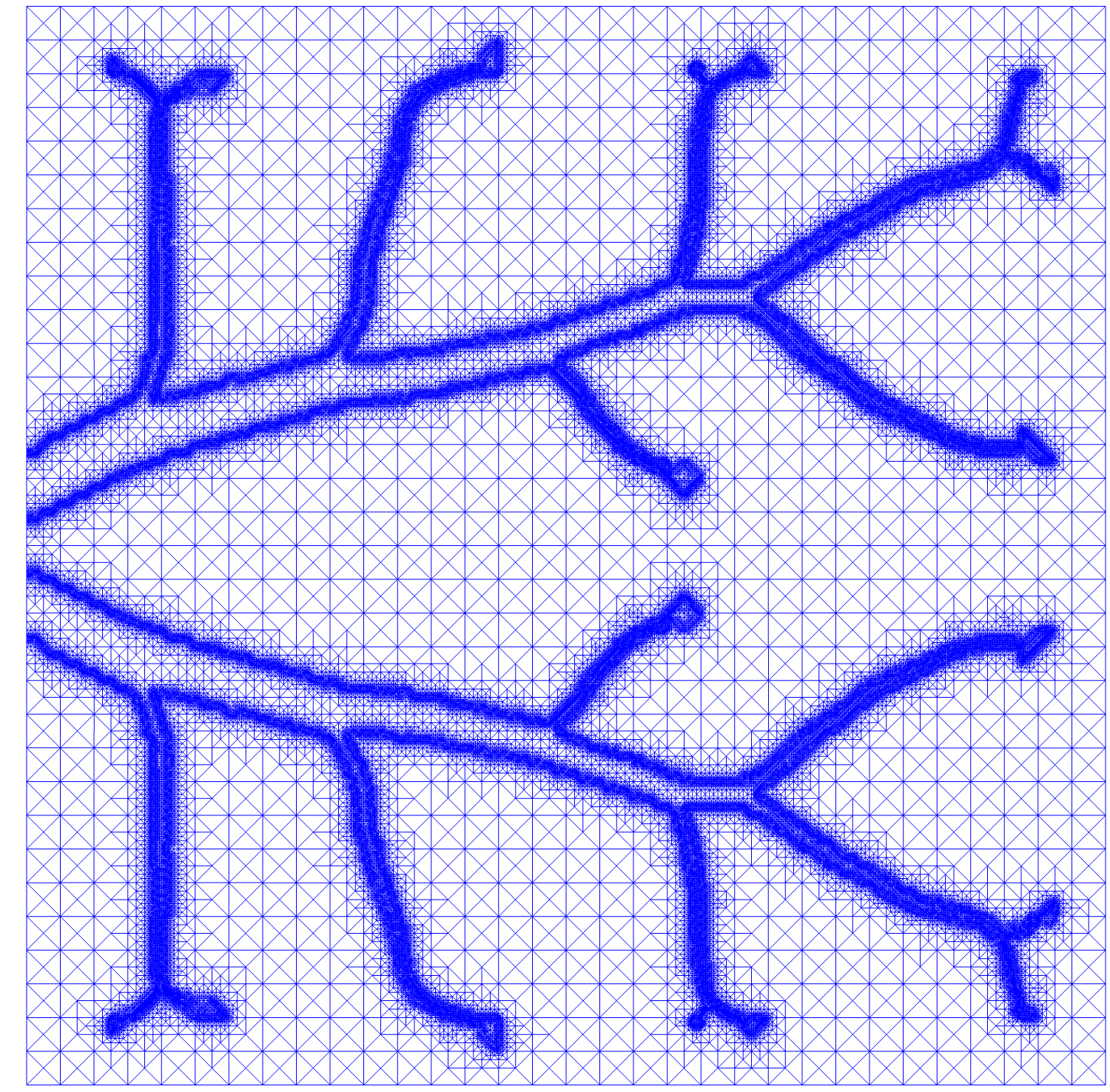}
                      & \rowincludegraphics[width=.25\linewidth]{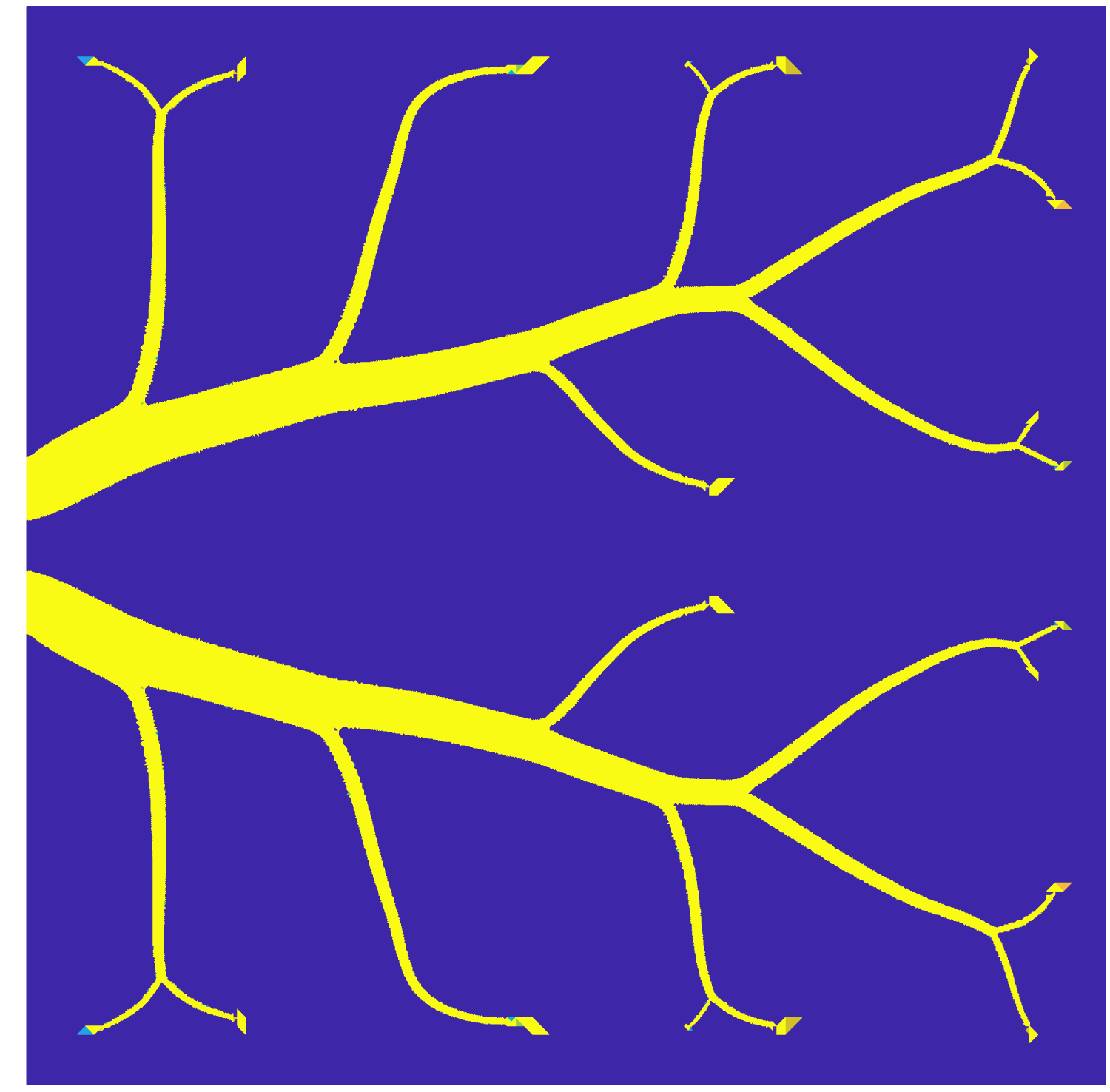}
                      & \rowincludegraphics[width=.25\linewidth]{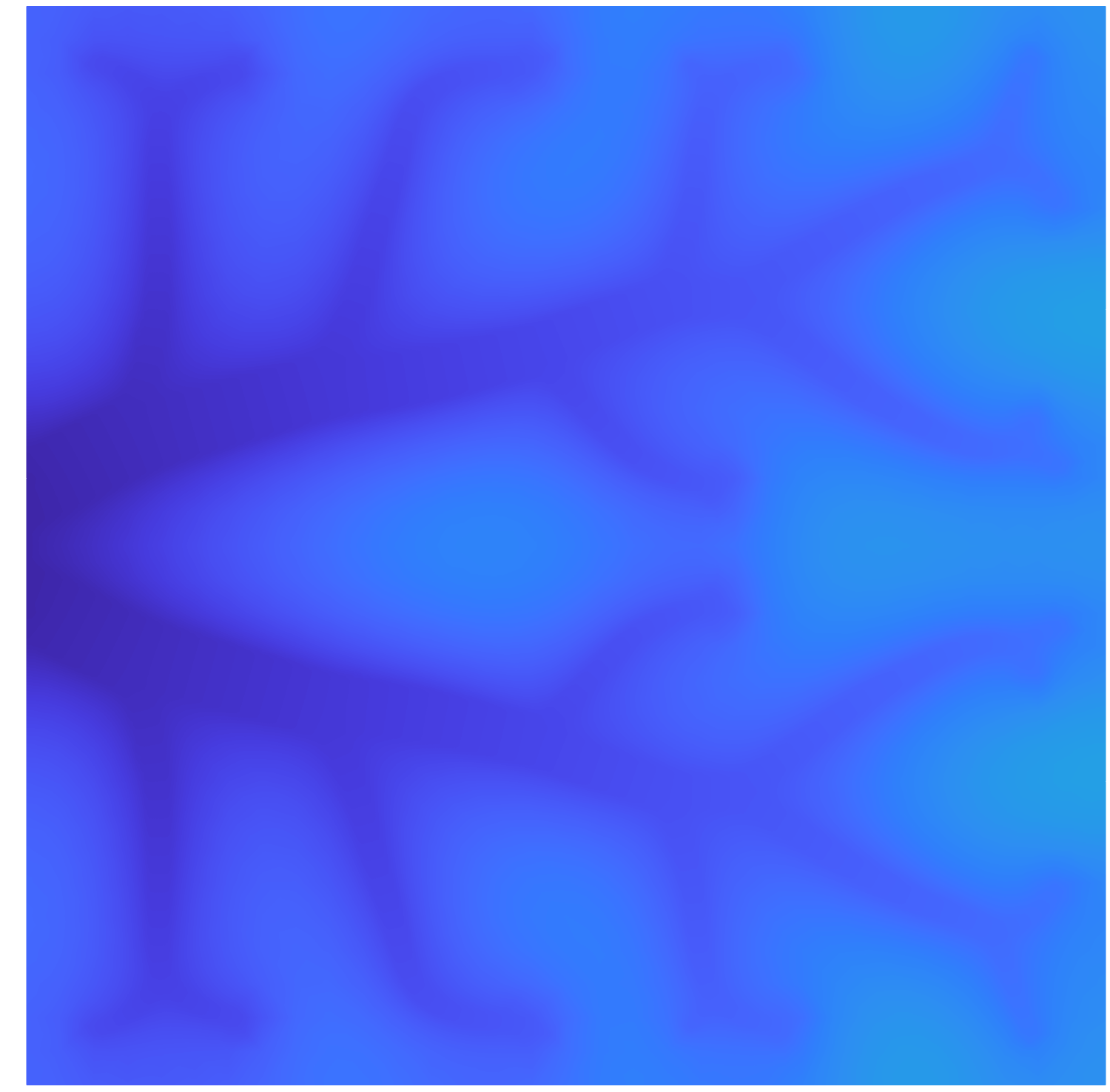}\\
\bottomrule
\end{tabular}
\caption{Results of \Cref{alg} for the second heat conduction topology
optimization example. Our algorithm terminated after 147 iterations.
Note the $N_{\rm dof}$s are approximate.}
\label{tbl:top1}
\end{figure}